\documentclass[12pt]{article}
\usepackage[inner=32mm,outer=32mm,top=28mm,bottom=28mm]{geometry}
\usepackage{amsfonts}
\usepackage{epsfig}
\usepackage{graphics}
\usepackage{inputenc}
\usepackage{lineno,hyperref}
\usepackage{graphicx}%
\usepackage{multirow}%
\usepackage{amsmath,amssymb,amsfonts}%
\usepackage{amsthm}%
\usepackage{mathrsfs}%
\usepackage[title]{appendix}%
\usepackage{xcolor}%
\usepackage{textcomp}%
\usepackage{manyfoot}%
\usepackage{booktabs}%
\usepackage{algorithm}%
\usepackage{algorithmicx}%
\usepackage{algpseudocode}%
\usepackage{listings}%
\theoremstyle{thmstyleone}%
\newtheorem{theorem}{Theorem}%  meant for continuous numbers
\newtheorem{lemma}[theorem]{Lemma}
\theoremstyle{thmstyletwo}%
\newtheorem{remark}{Remark}%

\theoremstyle{thmstylethree}%

\newcommand{\ve}{\varepsilon}

\usepackage{authblk}

\title%[Article Title]
{Bifurcation of finger-like structures in traveling waves of epithelial tissues  spreading}

\author[1]{Leonid Berlyand}
\author[2]{Antonina Rybalko}
\author[3]{Volodymyr Rybalko}
\author[4]{Clarke Alex Safsten}
\affil[1]{Department of Mathematics and Huck Institute for Life Sciences, The Pennsylvania State University, USA, 
e-mail: berlyand@math.psu.edu}
\affil[2]{Department of Mathematical Sciences, Chalmers University of Technology and Gothenburg University, Sweden,
e-mail: rybalkoa@chalmers.se}
\affil[3]{B. Verkin Institute for Low Temperature Physics and Engineering of the National Academy of Sciences of Ukraine, Ukraine;
Department of Mathematical Sciences, Chalmers University of Technology and Gothenburg University, Sweden,
e-mail: vrybalko@ilt.kharkov.ua, rybalko@chalmers.se}
\affil[4]{Department of Mathematics, The Pennsylvania State University, USA, email: cas774@psu.edu}

\begin{document}

\maketitle

\begin{abstract} We consider a continuum active polar fluid model for the spreading of epithelial monolayers introduced by R. Alert, C. Blanch-Mercader, and J. Casademunt,  2019.  The corresponding free boundary problem  possesses flat front  traveling wave solutions. Linear stability of these solutions under periodic perturbations is considered. It is shown that the solutions are stable for short-wave perturbations while exhibiting long-wave instability under certain conditions on the model parameters (if the traction force is sufficiently strong). Then, considering the prescribed period as the bifurcation parameter, we establish the emergence of nontrivial traveling wave solutions with a finger-like periodic structure (pattern). We also construct asymptotic expansions of the solutions in the vicinity of the bifurcation point and study their stability. We show that, depending on the value of the 
	contractility coefficient, the bifurcation can be a subcritical or a supercritical pitchfork.
\end{abstract}




\section{Introduction}\label{sec1}

The spreading of epithelial tissues plays an important role in the physiology of  living organisms. 
For instance, epithelial cells heal wounds by the collective migration of large sheets of cells bound together by intercellular connections \cite{LiHeZhaJia20213}. 
Other examples include tissue morphogenesis and tumor invasion. It is observed in experiments both in vivo and in vitro that the tissue front experiences instabilities similar to the celebrated 
Saffman-Taylor instabilities \cite{SafTay1958}, leading via multicellular protrusions to the formation of  finger-like patterns, see, e.g., \cite{PouGraHerJouChaLadBugSil2007,Vish2018,OmeVasGelFedBon2003}. 

In this work, we study this phenomenon in the framework of a  free-boundary 
model for epithelial monolayers spreading introduced in  \cite{AleBlaCas2019}, that is 
based on the theory of active polar fluids \cite{Prost2015}. The epithelial monolayer is regarded as a compressible  fluid
flowing subject to hydrodynamic viscous forces, %cell-substrate active traction forces, 
cell-substrate friction, surface tension, and active traction and contractile forces. 
The active forces are described by the polarity field. 

We establish nontrivial traveling wave solutions describing % modelling
the onset of finger-like patterns. These patterns emerge %in flat front solutions 
for a critical scale as the result of competition of destabilizing traction forces with stabilizing contractile stresses and surface tension. It was observed  in \cite{AleBlaCas2019}  by means of linear stability analysis that  solutions with flat interfaces are unstable under long-wavelength perturbations via the following kinematic mechanism.
%The underline kinematic mechanism :
A small perturbation of the monolayer edge
results in a velocity gradient that makes peaks move faster than troughs.
This leads, as shown  in \cite{TreBon2021} via numerical simulations, to the formation of finger-like patterns.

In the present work we {\it analytically} establish traveling wave solutions with finger-like patterns.  Specifically, 
we find flat-front traveling wave solutions, study their stability under periodic perturbations with a prescribed 
period, and,  considering the period as a bifurcation parameter show that at a critical value of the period, 
a pitchfork bifurcation occurs and a new branch of nontrivial traveling wave solutions emerges. Next, we study the linear stability of these new traveling wave solutions and identify whether the bifurcation is subcritical or supercritical.
% pitchfork. 
This stability issue has important biophysical implications. Namely, a subcritical 
bifurcation corresponds to an abrupt onset of  finger-like patterns while a supercritical bifurcation implies a gradual transition.  
% \textcolor{red}{(abrupt, discontinuous vs gradual, continuous onset of  finger-like patterns)}. 
We show that the type of bifurcation depends on the mechanical properties of tissue and both subcritical and supercritical pitchforks can 
happen. Specifically, varying the contractility parameter we observe that a subcritical pitchfork corresponds to large or sufficiently small values of the  contractility, while in another range of rather small values of the  contractility, there occurs  a supercritical pitchfork.

%\textcolor{red}{Pattern formation has been discovered on many nonequilibrium physics systems such as, e.g., dendritic crystals, granular materials or hydrodynamic systems \cite{Gollub1999}. }
%In particular, interfacial patterns in fluids confined in a quasi-two-dimensional geometry, the Hele-Shaw \textcolor{red}{model}, 
%%cell, 
% were first described by P. Saffman and G. I. Taylor in the seminal work \cite{SafTay1958}.
%\newpage
%The fascinating subject of pattern formation has been recently one of mainstreams of nonequilibrium physics
%Interfacial pattern formation is a fascinated phenomenon that occurs in many systems of the soft matter physics. % fluid dynamics.
%In the seminal work \cite{SafTay1958}  P. Saffman and G. I. Taylor described formation of finger-like patterns in the confined geometry of  Hele-Shaw cell. The problem attracted a huge interest in both physical and mathematical communities. 

%Patterns appear in many nonequilibrium  physics systems, in particular in hydrodynamic systems \cite{Gollub1999}. 

Many non-equilibrium physics systems, in particular hydrodynamic systems manifest pattern formation phenomenon  \cite{Gollub1999}. 
An important example of interfacial patterns in fluids confined in a quasi-two-dimensional geometry, 
the Hele-Shaw cell, 
 was first addressed by P. Saffman and G. I. Taylor in the seminal work \cite{SafTay1958}. The corresponding free boundary model has been attracting a lot 
 of attention in both physical and mathematical communities, see, e.g.,  \cite{Casa2004},  \cite{GuVas2006}, \cite{MorMor21}. 
 Moreover, free boundary problems of this type (with additional scalar field)
 %, supplemented with equations for an additional scalar field,  
 appear in recent biological models of tumor growth  \cite{Friedman2004}, 
 \cite{FriedmanHu2006}
  or cell motility  \cite{BlaMerCas2013}, \cite{Ryba2018}, \cite{Lavi2020},  \cite{Cucchi2022}, \cite{Ryba2023}. 
 In \cite{Friedman2001} A. Friedman and F. Reitich discovered symmetry-breaking steady states bifurcating from radial 
 solutions of the tumor growth free boundary problem, thus revealing pattern formation in this model. Another example of 
 symmetry breaking bifurcation is studied in work \cite{Ryba2023}  dealing with a cell motility model, where stability issue is also 
 addressed.

 The paper is organized as follows. Section \ref{SecModel}  is devoted to the 
description of the model. In Section \ref{prop_A}  we study the linearized operator,
 in particular, we show that it has a discrete spectrum. Next, we consider the flat front traveling wave solution and calculate its spectral representation via the Fourier analysis. The explicit 
 formula for eigenvalues is given  and analyzed in Section \ref{sec2_ffs}, while its 
 derivation  is presented in Appendix \ref{mister_first}. Then we study the case of the critical period 
 such that the kernel of the linearized operator (around the flat traveling wave) 
 has nonconstant eigenfunction. We show that a new branch of traveling waves with a finger-like structure 
 bifurcates and study their stability. The theory of  M.~Crandall and P.~Rabinowitz  \cite{CranRab},
  \cite{CranRab73} is used to study both bifurcation and stability questions. Namely, bifurcation of nonflat traveling  
waves is established in Section \ref{Bifur}, where we exploit symmetries of the problem to adjust functional setting 
for applying  Theorem 1.7  from  \cite{CranRab} (Theorem \ref{CR_theorem} below). Addressing
stability of traveling waves, 
we use results of Section \ref{prop_A} and Theorem 1.16 from  \cite{CranRab73} (Theorem \ref{CrRabSecond} below) to conclude 
that stability is determined by the fact whether the period (bifurcation parameter in the problem)  
 increases or decreases when departing from the bifurcation point. This makes us construct 
 several terms in the asymptotic expansion of the traveling wave solutions. Section  \ref{Expansions}  deals with
 these constructions, while many technical calculations are transferred  to Appendices \ref{app_v_21}--\ref{finish_nutochnovzhe}.  
%Finally, % In Conclusions section $=$ 
%Section \ref{conclusions_graf} contains numerical interpretation of results  realisation demonstrate conclusions
%\textcolor{red}{Finally, Section  \ref{conclusions_graf}  contains numerical realization of some analytical results.} 
Finally, Section~\ref{conclusions_graf}  contains some numerical results and conclusions. There, in particular, 
we describe how the stability/instability of bifurcating traveling waves depends on the value of the  contractility parameter.

\section{Model}
\label{SecModel}

Following \cite{AleBlaCas2019}, we employ a continuum active polar fluid model of tissue spreading, described by a polarity field 
${\bf p}(x,y,t)$ and a velocity field $\mathbf v(x,y,t)$. 
A tissue monolayer spreads by extending its edge towards free space. The phenomenon is mainly caused by traction forces generated by cells close to the monolayer edge. 
These cells polarize perpendicular to the edge, where we prescribe ${\bf p}={\bf n}$ (the unit outward normal).   
The field ${\bf p}$ is assumed to follow purely relaxational dynamics and equilibrate fast (compared to the spreading dynamics)  to the minimum of the energy with density 
$L_c^2|\nabla {\bf p}|^2+| {\bf p}|^2$, where $L_c$ is the characteristic length describing the decay rate of  ${\bf p}$ in the bulk. For simplicity,  we set $L_c=1$ that can always be achieved by an appropriate scaling of spatial variables. 
Thus,  ${\bf p}$  solves
\begin{equation}\label{eq:p}
\Delta \textbf p=\textbf p\quad \text{in} \ \Omega(t),\quad\quad\quad
\textbf p=\textbf n\quad \text{on}\ \partial\Omega(t),
\end{equation}
where $\Omega(t)$ denotes the domain occupied by the tissue and ${\bf n}$ is the unit outward normal vector to the boundary.

The force balance equation reads
\begin{equation}\label{eq:force_balance}
{\rm div} \sigma+\mathbf{f}=0 \quad \text{in} \ \Omega(t), 
\end{equation}
where $\sigma$ is the stress tensor and ${\bf f}$ is the stress field given by the following 
constitutive equations for a compressible active polar fluid:
\begin{equation}\label{eq:constitutive_equation}
\sigma=\mu(\nabla\mathbf v+(\nabla\mathbf v)^T)-\zeta\mathbf p\otimes\mathbf p, \quad \quad
\mathbf{f} = -\xi \mathbf{v}+\zeta_i \mathbf{p}\quad \quad \text{in} \ \Omega(t),
\end{equation}
where $\mathbf{v}$ is the velocity field, $\mu>0$ is the constant effective viscosity, $\zeta<0$ is the constant contractility coefficient,  $\xi>0$ is the constant friction coefficient, and $\zeta_i$ 
is the constant contact active force coefficient. On the free boundary %$\partial\Omega(t)$, 
$\sigma$ satisfies 
\begin{equation}\label{eq:no_stress}
\sigma\cdot\mathbf{n}=-\gamma\kappa \mathbf{n} \quad \text{on} \ \partial\Omega(t),
\end{equation} 
where $\kappa$ denotes the curvature of the boundary and $\gamma>0$ is the constant surface tension of the monolayer edge.

While the model in \cite{AleBlaCas2019} deals with small (linear) perturbations of a rectangular monolayer of  epithelial tissue, in this work we consider half-plane type domains. 
This corresponds to modeling of the local behavior near the boundary of sufficiently large tissue specimens. Mathematically, considering half-plane type domains 
allows us to go beyond linear stability analysis and describe the formation of finger 
patterns via bifurcation of traveling wave solutions.

%\textcolor{blue}{It means that they can serve us as a we study   which allow the existence of traveling wave solutions.} 

The evolution of the boundary $\partial\Omega(t)$ is described by equation $y=\rho(x,t)$, assuming that $\Omega(t)=\{(x,y)|\,y<\rho(x,t)\}$. Then the normal vector is given by 
\begin{equation}
\textbf n=\frac{1}{\sqrt{1+{\rho'}^2}}\binom{-\rho'}{1},
\end{equation}
where $\rho^\prime$ denotes the partial derivative of $\rho$ in $x$.

Assuming the continuity of velocities up to the boundary we have the following 
kinematic boundary condition relating the normal velocity of the boundary and the normal component of the tissue velocity field:
\begin{equation}\label{eq:kinematic_bc}
\mathbf v\cdot\mathbf{n} =V_\textbf{n}=\frac{1}{\sqrt{1+{\rho'}^2}}\frac{\partial\rho}{\partial t}.
\end{equation}
Taking \eqref{eq:p}--\eqref{eq:kinematic_bc} together, we have the equation
\begin{equation}\label{eq:nonlinear_evolution_equation}
\frac{\partial\rho}{\partial t}=\mathcal A(\rho), \quad \mathcal A(\rho)=\left(v_y-v_x\rho'\right)\big|_{y=\rho(x,t)},
\end{equation}
with $\mathbf v=(v_x,v_y), \ \mathbf p=(p_x,p_y)$ solving
\begin{align}
&\Delta {\bf p}={\bf p} & \text{for}\ y<\rho(x,t)\ \label{eq:p_bulk}\\
&\mu(\Delta{\bf v}+\nabla{\rm div}{\bf v})-\zeta{\rm div}(\mathbf p\otimes\mathbf p)-\xi\mathbf v+\zeta_i\mathbf p=0 \quad& \text{for}\  y<\rho(x,t) \ \label{eq:v_bulk}\\
&\mathbf p=\mathbf{n}  & \text{for}\ y=\rho(x,t)\ \label{eq:p_bdy}\\
&\left(\mu(\nabla\mathbf v+(\nabla\mathbf v)^T)-\zeta\mathbf p\otimes\mathbf p\right)\mathbf{n}=-\gamma\kappa\mathbf{n}  \quad  & \text{for}\  y=\rho(x,t)\label{eq:v_bdy}.
\end{align}
We also assume that $\mathbf v$ and $\mathbf p$ vanish as $y\to -\infty$.

\section{Linearized operator and its spectrum}	\label{prop_A}

Let $\rho(x)$ be an arbitrary function from the space $C^{k,\delta}_{ \#}(0,\Pi)$ of $k$ ($k\in \{3,4,\dots\}$) times  differentiable $\Pi$-periodic functions whose $k$-th derivatives are H{\"o}lder continuous with the exponent $0<\delta<1$.
Problem   \eqref{eq:p_bulk}--\eqref{eq:v_bdy} has a unique $\Pi$-periodic in $x$ and vanishing as $y\to-\infty$ solution in the subgraph domain $y<\rho(x)$. Therefore the operator $\mathcal{A}(\rho)$ is well-defined by \eqref{eq:nonlinear_evolution_equation}. By applying elliptic estimates from \cite{Agmon1964} to  problem \eqref{eq:p_bulk}-\eqref{eq:v_bdy} we get that the operator $\mathcal{A}$ maps  $\rho\in C^{k,\delta}_{ \#}(0,\Pi)$ to $\mathcal{A}(\rho)\in C^{k-1,\delta}_{ \#}(0,\Pi)$.

In this section, we consider the linearized operator $\partial_\rho\mathcal{A}(\rho)$ and show that it has a discrete spectrum and high magnitude eigenvalues are stable (have negative real parts).  
The following lemma establishes differentiability of  $\mathcal{A}(\rho)$ and provides a formula for the first derivative.
\begin{lemma}
\label{th:oper_prop}
The operator $\mathcal{A}(\rho)$ is of the class 
$C^\infty\left(C^{k,\delta}_{\#}(0,\Pi), C^{(k-1),\delta}_{\#}(0,\Pi)\right)$, $k\in \{3, 4\dots\}$. Its first derivative is given by 
\begin{align*}
	\partial_\rho\mathcal{A}(\rho) \tilde \rho=\left(\tilde w_y-\rho'\tilde w_x-\tilde \rho'v_x\right)\bigl.\bigr|_{y=\rho(x)}
\end{align*}
where ${\bf p}$, ${\bf v}$ solve \eqref{eq:p_bulk}-\eqref{eq:v_bdy}, $\tilde {\bf w}$ is the solution to the system
\begin{align}
	\notag 
	\mu \left(\Delta \tilde {w}_x-\partial_x{\rm div} \tilde {\bf w} -2\tilde\rho^{\prime\prime}\partial_y v_x \right.&\left. -\tilde\rho^\prime (\partial_{yy}^2v_y+4\partial^2_{xy} v_x)\right)
	-\xi \tilde w_x \\
	+\zeta_i \tilde q_x +\zeta \tilde \rho^\prime \partial _{y} p_x^2&=\zeta \bigl( {\rm div} \bigl({\bf p}\otimes \tilde {\bf q} + \tilde {\bf q}\otimes {\bf p}\bigr)\bigr)_x
	\ \quad \text{for}\ y<\rho(x),\label{4}\\
	\notag 
	\mu \left(\Delta \tilde {w}_y-\partial_y   {\rm div} \tilde {\bf w}-\tilde\rho^{\prime\prime}\partial_y v_y \right.&\left. -\tilde\rho^\prime (\partial_{yy}^2v_x+2\partial^2_{xy} v_y)\right)
	-\xi \tilde w_y\\
	+\zeta_i \tilde q_y +\zeta \tilde \rho^\prime \partial _{y} (p_x p_y)&=\zeta \bigl( {\rm div} \bigl({\bf p}\otimes \tilde {\bf q} + \tilde {\bf q}\otimes {\bf p}\bigr)\bigr)_y
	\ \quad \text{for}\ y<\rho(x), \label{5}
\end{align}
with boundary conditions
\begin{align}
	\notag -2\mu\rho^\prime \partial_x\tilde w_x&+\mu\left(\partial_y \tilde w_x+ \partial_x\tilde w_y\right)-2\mu\bigl(\tilde \rho^\prime \partial_x v_x -\rho^\prime \tilde \rho^\prime \partial_y v_x\bigr)
	-\mu\tilde \rho^\prime \partial_y v_y \\ 
	&=- \tilde \rho^\prime (\zeta-\gamma \kappa)
	- \textstyle\gamma \rho^\prime \Bigl(\bigl(1+{\rho^\prime}^2\bigr)^{-3/2}\tilde \rho^{\prime}\Bigr)^\prime
	%-{\textstyle\gamma\frac{\tilde \rho^{\prime\prime}\rho^\prime%+\rho^{\prime\prime} \tilde\rho^\prime
			%}{\bigl(1+{\rho^\prime}^2\bigr)^{3/2}}}-
	%3\gamma\kappa 
	%{\textstyle\frac{{\rho^\prime}^2\tilde \rho^\prime}{1+{\rho^\prime}^2}}
	\ \quad\quad\quad 
	\text{for}\ y=\rho(x), 
	\label{6}\\ 
	\notag -\mu\rho'\bigl(\partial_y\tilde w_x&+ \partial_x\tilde w_y\bigr)+2\mu\partial_y \tilde w_y 
	-\mu\tilde \rho^\prime\left(\partial_y v_x-\rho^\prime \partial_y v_y + \partial_x v_y\right)\\ 
	&=
	%\gamma\tilde \rho^{\prime\prime}\bigl(1+{\rho^\prime}^2\bigr)^{-3/2}-3\gamma\kappa \rho^\prime\tilde \rho^\prime\bigl(1+{\rho^\prime}^2\bigr)^{-1}\quad
	%
	%{\textstyle\gamma\frac{\tilde \rho^{\prime\prime}}{\bigl(1+{\rho^\prime}^2\bigr)^{3/2}}}+3 {\textstyle\gamma\kappa  \frac{\rho^\prime\tilde \rho^\prime}{1+{\rho^\prime}^2}} 
	%=
	\textstyle\gamma \Bigl(\bigl(1+{\rho^\prime}^2\bigr)^{-3/2}\tilde \rho^{\prime}\Bigr)^\prime\ \quad\quad\quad\quad\quad\quad\quad\quad\quad\quad
	\text{for}\ y=\rho(x), \label{7}
\end{align}
and $\tilde {\bf q}$ satisfies
\begin{align}
	 &\Delta \tilde {\bf q}-\tilde \rho^{\prime\prime} \partial_y{\bf p}
	-2 \tilde\rho^\prime \partial^2_{xy}{\bf p}
	%+{(\ve \tilde \rho^{\prime})}^2
	%\partial^2_{yy}
	%\left({\bf p}+\ve \tilde {\bf p}^{(\ve)}\right)
	=\tilde {\bf q} &\text{for}\ y<\rho(x),\label{tilde_p}\\
	\label{tilde_p2} &\tilde q_x=- {\textstyle\frac{\tilde \rho^\prime}{\bigl(1+{(\rho^\prime)}^2\bigr)^{3/2}}}&\text{for}\ y=\rho(x),
	%	\label{tilde p} &\tilde p_x=-\tilde \rho^\prime\bigl(1+{(\rho^\prime)}^2\bigr)^{-3/2} &\text{for}\ y=\rho(x),
	\\
	\label{tilde_p3}&\tilde q_y=- {\textstyle \frac{\rho^\prime \tilde \rho^\prime}{\bigl(1+{(\rho^\prime)}^2\bigr)^{3/2}}} & \text{for}\ y=\rho(x).
	%	\notag &\tilde p_y=-\rho^\prime \tilde \rho^\prime\bigl(1+{(\rho^\prime)}^2\bigr)^{-3/2}  & \text{for}\ y=\rho(x).
\end{align}
\end{lemma}
\begin{proof}
To find $\partial_\rho \mathcal{A}(\rho)$ consider the perturbation $\rho^{(\ve)}=\rho+\ve\tilde \rho$ of the domain, where $\ve$ is a small parameter. 
Let ${\bf p}$, ${\bf p}^{(\ve)}$ be the solutions of the following problems in domains with boundaries $y=\rho(x)$ and $y=\rho(x)+\ve \tilde \rho(x)$ respectively:
\begin{align}
	\label{p}
	&\Delta{\bf p}={\bf p}  \quad \text{for} \ y<\rho(x),\quad\quad% \notag \\ 
	{\bf p}={\bf n} \quad \text{for}\ y=\rho(x),\\
	%\end{align}
	%\begin{align}
	\label{p_ve}
	&\Delta {\bf p}^{(\ve)}={\bf p}^{(\ve)}  \quad \text{for} \ y<\rho(x)+\ve \tilde \rho(x), %\notag \\ 
	\quad\quad {\bf p}^{(\ve)}={\bf n} \quad\text{for}\ y=\rho(x)+\ve \tilde \rho(x).
\end{align}
Represent  ${\bf p}^{(\ve)}$ in the form
\begin{equation}
	{\bf p}^{(\ve)}(x,y)={\bf p}(x,y-\ve \tilde \rho(x))+\ve \tilde {\bf p}^{(\ve)}(x,y-\ve \tilde \rho(x)),
\end{equation}   
substitute in the equation \eqref{p_ve} to find (after changing variables) that  	
\begin{align}
	\notag &\Delta \tilde {\bf p}^{(\ve)}-\tilde \rho^{\prime\prime} \partial_y\bigl({\bf p}+\ve \tilde {\bf p}^{(\ve)}\bigr)
	-\tilde\rho^\prime \bigl(2\partial_x-\ve \tilde \rho^{\prime}\partial_y\bigr) \partial_{y}\bigl({\bf p}+\ve \tilde {\bf p}^{(\ve)}\bigr)
	%+{(\ve \tilde \rho^{\prime})}^2
	%\partial^2_{yy}
	%\left({\bf p}+\ve \tilde {\bf p}^{(\ve)}\right)
	=\tilde {\bf p}^{(\ve)} &\text{for}\ y<\rho(x),\\
	\notag &\tilde p^{(\ve)}_x={\textstyle\frac{\rho^\prime}{\ve\sqrt{1+{(\rho^\prime)}^2}}}-{\textstyle\frac{\rho^\prime+\ve\tilde \rho^\prime}{\ve\sqrt{1+{(\rho^\prime+\ve\tilde \rho^\prime)}^2}}} &\text{for}\ y=\rho(x),
	%	\notag &\tilde p^{(\ve)}_x=\ve^{-1}\rho^\prime\left(1+{(\rho^\prime)}^2\right)^{-1/2}	-\ve^{-1}(\rho^\prime+\ve\tilde \rho^\prime)	\left(1+{(\rho^\prime+\ve\tilde \rho^\prime)}^2\right)^{-1/2}
	\\
	\notag &\tilde p^{(\ve)}_y=
	%\ve^{-1}\left(1+{(\rho^\prime+\ve\tilde \rho^\prime)}^2\right)^{-1/2}-\ve^{-1}\left(1+{(\rho^\prime)}^2\right)^{-1/2}
	{\textstyle \frac{1}{\ve\sqrt{1+{(\rho^\prime+\ve\tilde \rho^\prime)}^2}}}- {\textstyle \frac{1}{\ve\sqrt{1+{(\rho^\prime)}^2}}}  &\text{for}\ y=\rho(x)
\end{align}
Passing to the limit as $\ve\to 0$ in this problem using elliptic estimates we see that $\tilde {\bf p}^{(\ve)}$ converges in $C^{k,\delta}(K)$ on every compact $K\subset \{(x,y)|\, y\leq \rho(x)\}$ to the solution $\tilde {\bf p}$ of \eqref{tilde_p}--\eqref{tilde_p3}.

Similarly one can show that if ${\bf v}^{(\ve)}$ is represented 
as 
\begin{equation}
	{\bf v}^{(\ve)}(x,y)={\bf v}(x,y-\ve \tilde \rho(x))+\ve \tilde {\bf v}^{(\ve)}(x,y-\ve \tilde \rho(x)),
\end{equation}  
then $\tilde {\bf v}^{(\ve)}$ converges in $C^{k,\delta}(K)$ on every compact $K\subset \{(x,y)|\, y\leq \rho(x)\}$ to 
the solution $\tilde {\bf v}$ of \eqref{4}-\eqref{7}.  Here the limit transition can be justified by using elliptic estimates from \cite{Agmon1964}.

Reasoning analogously one establishes the existence of higher order derivatives.
\end{proof}

Notice that  setting
\begin{equation}
\tilde {\bf w}(x,y)=\tilde \rho(x)\partial_y {\bf v}(x,y)+{\bf w}(x,y), \ \ \tilde {\bf q}(x,y)=
\tilde \rho(x)\partial_y {\bf p}(x,y)+ {\bf q}(x,y),
\label{changing}
\end{equation}
we can simplify the boundary value problem \eqref{4}--\eqref{7} as follows.
The operator $\partial_\rho \mathcal{A}(\rho)$  can be written as  
\begin{equation}
\partial_\rho\mathcal{A}(\rho)\tilde \rho =\left({\bf w}\cdot {\bf n}+\tilde \rho \partial_y {\bf v}\cdot {\bf n}\right)\bigl.\bigr|_{y=\rho(x)}\sqrt{1+{\rho^\prime}^2}-\tilde\rho^\prime v_x\bigl.\bigr|_{y=\rho(x)}
\label{OperA_new_def}
\end{equation}
with ${\bf w}, {\bf q}$ solving
\begin{align}
&\Delta {\bf q}={\bf q} & \text{for}\ y<\rho(x), \label{eq:q_bulk}\\
&\mu(\Delta{\bf w}+\nabla{\rm div}{\bf w})-\zeta{\rm div}(\mathbf p\otimes\mathbf q+\mathbf q\otimes\mathbf p)-\xi\mathbf w+\zeta_i\mathbf q=0 &\text{for}\  y<\rho(x), \label{eq:w_bulk}\\
&\mathbf q   ={\textstyle\frac{\tilde \rho^\prime}{1+{(\rho^\prime)}^2}} {\bf t}-\tilde\rho\partial_y {\bf p} &\text{for}\ y=\rho(x), \label{eq:q}\\
&\mu(\nabla\mathbf w+(\nabla\mathbf w)^T)\mathbf{n}=
{\textstyle\gamma \Bigl(\bigl(1+{\rho^\prime}^2\bigr)^{-3/2}\tilde \rho^{\prime}\Bigr)^\prime}{\bf n}+{\bf G}
&\text{for}\  y=\rho(x),\label{same_golovne_bc:w} 
\end{align}
where 
\begin{align*}
G_x=&{\textstyle\frac{\mu}{\sqrt{1+{(\rho^\prime)}^2}}} \bigl(\tilde \rho^\prime\left(2\partial_x v_x+\rho^\prime\left(\partial_x v_y+\partial_y v_x\right)-2\partial_y v_y\right)+2\rho^\prime\tilde\rho \partial^2_{xy} v_x
\bigr. \\ \bigl.
&\quad\quad\quad\quad- \tilde\rho\left(\partial^2_{xy} v_y +\partial^2_{yy} v_x\right)\bigr), \\ 
G_y=&{\textstyle\frac{\mu}{\sqrt{1+{(\rho^\prime)}^2}}} \left(\tilde \rho^\prime\left(\partial_y v_x+ \partial_x v_y\right)+ \rho^\prime \tilde\rho\left(\partial^2_{xy} v_y +\partial^2_{yy} v_x\right)-2\tilde \rho \partial_{yy}v_y\right),
\end{align*}
and $\bf t$ is the unit tangent vector, $t_x=- {\textstyle\frac{1}{\sqrt{1+{(\rho^\prime)}^2}}}$, $t_y=- {\textstyle\frac{\rho^\prime}{\sqrt{1+{(\rho^\prime)}^2}}}$. 
% The operator $\partial_\rho\mathcal{A}(\rho): H^{3/2}_{ \#}(0, \Pi)/\mathbb{R}\to  	H^{1/2}_{ \#}(0, \Pi)\textcolor{red}{/\mathbb{R}}$ is given by 

For every $\tilde\rho\in H^{3/2}_{ \#}(0,\Pi) $ formula \eqref{OperA_new_def}
defines $\partial_\rho\mathcal{A}(\rho)\tilde\rho \in H^{1/2}_{ \#}(0,\Pi) $ since
%One can see that $\partial_\rho\mathcal{A}(\rho)$ extends to a bounded operator  $\partial_\rho\mathcal{A}(\rho):H^{3/2}_{ \#}(0,\Pi) \to H^{1/2}_{ \#}(0,\Pi) $ since 
%$\partial_\rho\mathcal{A}(\rho)\tilde\rho$ \   is well-defined for every $\tilde\rho\in H^{3/2}_{ \#}(0,\Pi) $ since
problem
\eqref{eq:q_bulk}--\eqref{same_golovne_bc:w} has a unique solution pair  (vanishing as $y\to -\infty$) %in ${\bf q}$ and ${\bf w}$ in  H^1_\#(\Omega_{\#,\rho})$ 
and 
$\|{\bf w}\|_{H^1(\Omega_{\#,\rho})}\leq C\|\tilde\rho\|_{H^{3/2}_{\#}(0,\Pi)}$, where $\Omega_{{\#},\rho}=\{(x,y)\,|\, 0\leq x<\Pi,\, y<\rho(x)\}$ (the period of the domain). 
Hereafter $C$ denotes a generic finite constant whose value may 
change from line to line.

\begin{theorem} \label{teorema_pro_spectr}
Assume that  $\rho\in C^{3,\delta}_{\#}(0,\Pi)$, $\delta>0$. Then  $\partial_\rho\mathcal{A}(\rho)$ % extends to 
is a closed operator in $H^{1/2}_{\#}(0,\Pi)$  and the domain of  $\partial_\rho\mathcal{A}(\rho)$  is $H^{3/2}_{ \#}(0,\Pi)$. 
Its spectrum comprises at most countable set of eigenvalues (of finite multiplicities) without finite accumulation points. Moreover, for every eigenvalue $\lambda $ it holds that
\begin{equation}
	{\rm Re}(\lambda)\leq   C - \eta_* |\lambda |,
	\label{eta_zvezda}
\end{equation}
where $ \eta_*>0$ is independent of $\lambda$.

\end{theorem}
\begin{proof}

Notice that by elliptic estimates $\partial_\rho\mathcal{A}(\rho)\tilde \rho\in H^{1/2}_{\#}(0,\Pi)$, $\forall\tilde \rho\in H^{3/2}_{\#}(0,\Pi)$, and 
the operator $\partial_\rho\mathcal{A}(\rho)$  annihilates constant functions, therefore it is well defined on $H^{3/2}_{\#}(0,\Pi)/\mathbb{R}$. Consider for $\Lambda>0$ the equation
\begin{equation}
	\partial_\rho\mathcal{A}(\rho)\tilde \rho-\Lambda\tilde \rho =f\quad \text{in}\  H^{1/2}_{\#}(0,\Pi)/\mathbb{R} 
	\label{OperA_resolventa}
\end{equation}
and show that for sufficiently large $\Lambda>0$ there is a unique solution $\tilde\rho\in H^{3/2}_{\#}(0,\Pi)/\mathbb{R}$ 
for every  $f\in H^{1/2}_{\#}(0,\Pi)/\mathbb{R}$. 
To this end, we write down a weak formulation of \eqref{OperA_resolventa} multiplying 
by  
$\Bigl(\bigl(1+{\rho^\prime}^2\bigr)^{-3/2}\phi^{\prime}\Bigr)^\prime$
and integrating over the period,
\begin{equation}
	\mathcal{E}(\tilde\rho,\phi)=-\int_0^\Pi f^\prime \phi^\prime \bigl(1+{\rho^\prime}^2\bigr)^{-3/2}\ dx\quad \forall\phi\in H^{3/2}_{\#}(0,\Pi),
	\label{OperA_resolventa_weakformulation}
\end{equation}
where the form $\mathcal{E}(\tilde\rho,\phi)$ is given by
\begin{equation}
	\mathcal{E}(\tilde\rho,\phi)=\int_0^\Pi \left(\Bigl(\bigl(1+{\rho^\prime}^2\bigr)^{-3/2}\phi^{\prime}\Bigr)^\prime\partial_\rho \mathcal{A}(\rho)\tilde \rho+\Lambda\bigl(1+{\rho^\prime}^2\bigr)^{-3/2}\tilde \rho^\prime \phi^\prime\right)dx.
	\label{def_OF_form}
\end{equation}
Since $\forall\tilde\phi\in H^{-1/2}_\# (0, \Pi)$ with zero mean value the equation  
$\Bigl(\bigl(1+{\rho^\prime}^2\bigr)^{-3/2}\phi^{\prime}\Bigr)^\prime=\tilde\phi$ has a  solution $\phi\in H^{3/2}_{\#}(0,\Pi)$, the variational problem \eqref{OperA_resolventa_weakformulation} 
and the equation \eqref{OperA_resolventa} are equivalent. Next we show that for sufficiently large $\Lambda$ the form \eqref{def_OF_form} satisfies conditions of the Lax-Milgram theorem. 
The continuity of $\mathcal{E}(\tilde\rho,\phi)$ on $H^{3/2}_{\#}(0,\Pi)/\mathbb{R}$ follows from the boundness of the operator \eqref{OperA_new_def}, and we proceed with its coercivity.
%Let us show that this form $\mathcal{E}(\tilde\rho,\phi)$ is also coercive on $H^{3/2}_{\#}(0,\Pi)/\mathbb{R}$  when $\Lambda$ is sufficiently large. 
% (the form is well defined on this quotient space since $\mathcal{A}(\rho)$ annihilates constant functions). 
%Integrating by parts we have 
%\begin{multline}
%\frac{\mu}{2}\int_{y<\rho(x)} \bigl| \nabla {\bf w} + \nabla {\bf w}^T\bigr|^2dxdy+\zeta {\rm div}(\mathbf p\otimes\mathbf q+\mathbf q\otimes\mathbf p) \cdot {\mathbf w}+\xi |{\mathbf w}|^2-\zeta_i {\mathbf q}\cdot {\mathbf w} dxdy\\
%=\mu\int_{y=\rho(x)}(\nabla\mathbf w+(\nabla\mathbf w)^T)\mathbf{n}\cdot {\bf w}\, ds
%%+\int_{y=\rho(x)}{\mathbf{G}}\cdot {\bf w}ds
%\end{multline}
Consider an arbitrary function $\tilde \rho\in H^{3/2}_{ \#}(0, \Pi)$  with zero mean value. Take the dot product of \eqref{eq:w_bulk} with ${\bf w}$ and  integrate over $\Omega_{{\#},\rho}$. Using  \eqref{eq:w_bulk}, \eqref{same_golovne_bc:w} %--\eqref{OperA_new_def} 
we obtain via integration by parts, 
\begin{equation}
	\label{Chastynamy_integr}
	\begin{aligned}
		%\int_0^\Pi \int_{-\infty}^{\rho(x)}
		\int_{\Omega_{{\#},\rho}}
		%\Bigl( 
		%.       {\textstyle \frac{\mu}{2}}\bigl| \nabla {\bf w} + \nabla {\bf w}^T\bigr|^2+\zeta{\rm div}(\mathbf p\otimes\mathbf q+\mathbf q\otimes\mathbf p) \cdot {\mathbf w}+\xi |{\mathbf w}|^2-\zeta_i {\mathbf q}\cdot {\mathbf w} \Bigr)dxdy
		\Bigl( 
		{\textstyle \frac{\mu}{2}}\Bigr.&\Bigl.\bigl| \nabla {\bf w} + \nabla {\bf w}^T\bigr|^2-\zeta \bigl( \nabla {\bf w} + \nabla {\bf w}^T\bigr) {\bf p}\cdot {\bf q} +
		\xi |{\mathbf w}|^2-\zeta_i {\mathbf q}\cdot {\mathbf w} \Bigr)dxdy\\
		&+\zeta \int_0^\Pi ({\bf n}\cdot {\bf q}\, {\bf p}\cdot {\bf w}+ {\bf n}\cdot {\bf p}\, {\bf q}\cdot {\bf w})\bigl.\bigr|_{y=\rho(x)} \sqrt{1+{\rho^\prime}^2}dx
		\\
		&=\mu\int_{0}^{\Pi} (\nabla\mathbf w+(\nabla\mathbf w)^T)\mathbf{n} \cdot {\bf w}\bigl.\bigr|_{y=\rho(x)} \sqrt{1+{\rho^\prime}^2} dx, %\\
	\end{aligned}
\end{equation}
Furthermore,  by \eqref{same_golovne_bc:w} and  \eqref{OperA_new_def} we have 
\begin{equation}
	\begin{aligned}
		\mu\int_{0}^{\Pi} & (\nabla\mathbf w+(\nabla\mathbf w)^T)\mathbf{n} \cdot {\bf w}\bigl.\bigr|_{y=\rho(x)} \sqrt{1+{\rho^\prime}^2} dx
		%\\
		%=\textcolor{green}{\int_0^\Pi \left({\textstyle \gamma \Bigl(\bigl(1+{\rho^\prime}^2\bigr)^{-3/2}\tilde \rho^{\prime}\Bigr)^\prime} {\bf w}\cdot {\bf n}
			%+{\bf G}\cdot {\bf w}\right) \sqrt{1+(\rho^\prime)^2} dx }\\
		=\gamma \mathcal{E}(\tilde\rho,\tilde\rho)\\
		&
		-{%\textstyle 
			\frac{\gamma}{2}}\int_{0}^{\Pi}\left(\bigl(1+{\rho^\prime}^2\bigr)^{3/2} v_x\right)^\prime 
		{
			\textstyle 
			\frac{\left.\tilde\rho^\prime\right.^2}
			{\left(1+{\rho^\prime}^2\right)^3}
		}dx-\gamma\Lambda\int_0^\Pi \bigl(1+{\rho^\prime}^2\bigr)^{-3/2}\left.\tilde \rho^\prime\right.^2dx \\
		&+\gamma\int_{0}^{\Pi}
		\Bigl(\tilde \rho \partial_y {\bf v}\cdot {\bf n}\sqrt{1+{\rho^\prime}^2}\Bigr)^\prime{\textstyle\frac{\tilde\rho^\prime dx} {\left(1+{\rho^\prime}^2\right)^{3/2}} } +
		\int_0^\Pi {\bf G}\cdot {\bf w} \sqrt{1+{\rho^\prime}^2} dx.
	\end{aligned}
	\label{Trohygromizdko}
\end{equation}
We combine \eqref{Chastynamy_integr} and \eqref{Trohygromizdko}, then applying the Korn inequality %Poincar\'e's inequality, and % together with 
and the Cauchy-Schwarz inequality 
we find 
\begin{equation}
	\label{Zibraly}
	\begin{aligned}
		\mathcal{E}(\tilde\rho,\tilde\rho)\geq &\eta_1\| {\bf w}\|^2_{H^1(\Omega_{{\#},\rho})} - C_1\| {\bf w}\|_{H^1(\Omega_{{\#},\rho})} \| {\bf q}\|_{L^2(\Omega_{{\#},\rho})}
		\\
		&-C_2 \| \tilde \rho \|_{H^1(0,\Pi)}\| {\bf w}(x,\rho(x))\|_{H^{1/2}_{\#}(0,\Pi)}+(\eta_2\Lambda-C_3) \| \tilde \rho \|^2_{H^1(0,\Pi)},
	\end{aligned}
\end{equation}
where $\eta_{1,2}>0$ are independent of $\tilde \rho$ and $\Lambda$. Next, we use the inequality for traces $ \| {\bf w}(x,\rho(x))\|_{H^{1/2}_\#(0,\Pi)}\leq C\| {\bf w}\|_{H^1(\Omega_{{\#},\rho})}$
and the following bound 
$$
\int_{\Omega_{{\#},\rho}} |{\bf q}|^2 dxdy \leq C \int_0^\Pi |\tilde \rho^\prime|^2 dx,
$$
which follows from \eqref{eq:q_bulk}, \eqref{eq:q}. %In the case of flat boundary ($\rho =0$)  this latter bound can be established with the help of separation of variables in \eqref{eq:q_bulk}, while in general case one can use  a conformal mapping to flatten the boundary and then construct a barrier.
Thus, for sufficiently large $\Lambda>0$ \eqref{Zibraly} yields the following bound 
\begin{equation}
	\label{coercive}
	\mathcal{E}(\tilde\rho,\tilde\rho)\geq \eta_3\| {\bf w}\|^2_{H^1(\Omega_{{\#},\rho})}+\|\tilde \rho \|^2_{H^1(0,\Pi)}, 
\end{equation}
with $\eta_3>0$.  This inequality in turn implies that   
%
%Next we use Korn's inequality 
%$$
%%\int_0^\Pi \int_{-\infty}^{\rho(x)}
%\int_{\Omega_{{\#},\rho}}
%\Bigl( 
%{\textstyle \frac{\mu}{2}}\bigl| \nabla {\bf w} + \nabla {\bf w}^T\bigr|^2+{\textstyle \xi} |{\mathbf w}|^2\Bigr) dxdy \geq \eta_1
%%\int_0^\Pi \int_{-\infty}^{\rho(x)}
%\int_{\Omega_{{\#},\rho}}
%\Bigl( 
%\bigl| \nabla {\bf w} \bigr|^2+ |{\mathbf w}|^2\Bigr)dxdy,
%$$
%the bound  for traces 
%$$
%\int_0^\Pi |{\bf w}|^2 \bigl.\bigr|_{y=\rho(x)}dx\leq C_1\| {\bf w}\|^2_{H^1(\Omega_{{\#},\rho})},
%$$
%and the following bound
%$$
%\int_{\Omega_{{\#},\rho}} |{\bf q}|^2 dxdy \leq C_2 \int_0^\Pi |\tilde \rho^\prime|^2 dx.
%$$
%
%
%
%
%\textcolor{green}{Second term by parts:
	%\begin{multline*}
	%%\int_0^\Pi \int_{-\infty}^{\rho(x)}
	% \int_{\Omega_{{\#},\rho}}
	%{\rm div}(\mathbf p\otimes\mathbf q+\mathbf q\otimes\mathbf p) \cdot {\mathbf w} dxdy 
	%=\int_0^\Pi ({\bf n}\cdot {\bf q}\, {\bf p}\cdot {\bf w}+ {\bf n}\cdot {\bf p}\, {\bf q}\cdot {\bf w})\bigl.\bigr|_{y=\rho(x)} \sqrt{1+{\rho^\prime}^2}dx \\
	%- \int_{\Omega_{{\#},\rho}}\bigl( \nabla {\bf w} + \nabla {\bf w}^T\bigr) {\bf p}\cdot {\bf q} dxdy
	%%\int_0^\Pi \int_{-\infty}^{\rho(x)}
	%\end{multline*}
	%}
%
%
%... to  find that
%\begin{equation}
%\label{coercive}
%\mathcal{E}(\tilde\rho,\tilde\rho)\geq \eta_?\| w\|^2_{H^1(\Omega_{{\#},\rho})}+\|\tilde \rho \|^2_{H^1(0,\Pi)}.
%\end{equation}
$\mathcal{E}(\tilde\rho,\tilde\rho) \geq \eta_4 \|\tilde \rho \|^2_{H^{3/2}_{\#}(0,\Pi)}$ for some $\eta_4>0$.
Otherwise,  for a sequence $\tilde \rho^{(n)}$ with $\|\tilde \rho^{(n)}\|_{H^{3/2}_{\#}(0,\Pi)}=1$ it holds that  $\|\tilde \rho^{(n)} \|_{H^1(0,\Pi)}\to 0$ 
and the corresponding solutions ${\bf w}^{(n)}$ converge to $0$ strongly in $H^1(\Omega_{{\#},\rho})$. Then in view of \eqref{eq:w_bulk} and 
\eqref{same_golovne_bc:w} we have that $(\tilde \rho^{(n)})^{\prime\prime}\to 0$ strongly in $H^{-1/2}_{\#}(0,\Pi)$. Therefore, by  
compactness of the embedding of  $H^{3/2}_{\#}(0,\Pi)$ into $H^{1/2}_{\#}(0,\Pi)$ the spectrum of $\partial_\rho\mathcal{A}(\rho)$ is discrete.

Assume now that $\lambda$ is an eigenvalue and $\tilde\rho$ is a corresponding eigenfunction. Subtract from $\tilde \rho$ its mean value then the resulting function, still denoted by $\tilde \rho$, satisfies
$\partial_\rho\mathcal{A}(\rho)\tilde \rho-\Lambda\tilde \rho =(\lambda-\Lambda)\tilde \rho $ up to a constant. Therefore arguing as above (but working with real and imaginary parts of $\tilde \rho $) one can  
show that if $\tilde \rho $   is normalized by $\|\tilde \rho\|_{H^1(0,\Pi)}=1$ then   
\begin{align}
	{\rm Re}(\lambda)\leq   C - \eta_4 \|\tilde \rho \|^2_{H^{3/2}_{\#}(0,\Pi)}.
	\label{former}
\end{align}
On the other hand 
\begin{align}
	|\lambda|&\leq C|\lambda|\int_0^\Pi \bigl(1+{\rho^\prime}^2\bigr)^{-3/2} |\tilde\rho^\prime |^2 dx\notag\\  
	&\leq C_1 \|\partial_\rho\mathcal{A}(\rho)\tilde \rho\|_{H^{1/2}_{\#}(0,\Pi)}
	\left(\|\tilde \rho''\|_{H^{-1/2}_{\#}(0,\Pi)}+\|\tilde \rho'\|_{H^{-1/2}_{\#}(0,\Pi)}
	\right)\leq  C_2\|\tilde \rho \|^2_{H^{3/2}_{\#}(0,\Pi)}.\notag
\end{align}
Combining this bound with \eqref{former} we obtain \eqref{eta_zvezda}.
\end{proof}

\section{Flat front solutions and their stability analysis}\label{sec2_ffs}
We are interested in a particular form of solutions of problem  \eqref{eq:p_bulk}--\eqref{eq:v_bdy} that evolve translationally, traveling waves.
The system \eqref{eq:p_bulk}--\eqref{eq:v_bdy} has, inter alia, a flat front traveling wave solution whose boundary is a straight line moving with constant velocity $V^{(0)}$ along $y$-axis. This solution does not depend on the 
$x$-variable and is stationary in the moving frame: ${\bf v}(x,y,t)={\bf V}(y-V^{(0)}t)$, ${\bf p}(x,y,t)={\bf P}(y-V^{(0)}t)$. %Such a solution 
Moreover, it is defined up to a translation in the direction of $y$-axis and we stick to the one with  $\rho=0$. Then we have  $P_x=0$, $V_x=0$ and
\begin{equation*}
\begin{cases} \partial^2_{yy} P_y=P_y &\text{for}\ y<0,\\
	P_y=1\quad&\text{for}\ y=0,
\end{cases}
\quad\quad \quad\quad \quad\text{i.e.} \  P_y=e^y,
\end{equation*} 
%i.e. $P_y=e^y$,
\begin{equation}
\label{UrForVy}
\begin{cases} 2\mu \partial^2_{yy} V_y-\xi V_y -2\zeta e^{2y}+\zeta_i e^y=0\quad &\text{for}\ y<0,\\
	2\mu \partial_y V_y=\zeta\quad&\text{for}\ y=0.
\end{cases}
\end{equation}    
The unique (vanishing as $y\to -\infty$) solution of \eqref{UrForVy} is given by 
\begin{equation}
\label{Vy} 
V_y(y)= \textstyle{\frac{\zeta}{8\mu-\xi}}\Bigl(2 e^{2y} -\frac{\sqrt{\xi}}{\sqrt{2\mu}}e^{\sqrt{\xi}y/\sqrt{2\mu}} \Bigr) 
+\frac{\zeta_i}{2\mu-\xi} \left(\frac{\sqrt{2\mu}}{\sqrt{\xi}}e^{\sqrt{\xi}y/\sqrt{2\mu}}-e^y\right),
\end{equation}
so that it travels with constant velocity
\begin{equation}
\label{V0}
V^{(0)}=V_y\big|_{y=0}\bigr.= \textstyle{\frac{\zeta}{4\mu+\sqrt{2\mu\xi}}}+\frac{\zeta_i}{\sqrt{2\mu\xi}+\xi}.
\end{equation}
%\textcolor{blue}{Also, for later use
%\begin{equation*}\partial_{yy}V_y \big|_{y=0}\bigr.= \frac{\zeta}{2\mu\Bigl(4\mu+\sqrt{2\mu\xi}\Bigr)}-\frac{\zeta_i}{\sqrt{2\mu\xi}+2\mu}+\frac{\zeta}{\mu}.
%\end{equation*}}

A flat front traveling wave solution can be considered as a periodic one with an arbitrary period $\Pi$. 
Our interest however is in finding non-flat $\Pi$-periodic in $x$ traveling wave solutions, that is we seek a pair %$(\rho, C_\rho)$ 
of  $\Pi$-periodic function $\rho(x)$ and a % unknown
constant  $C_\rho$ (velocity of this wave) such that  
\begin{equation}
	\mathcal{A}(\rho)=C_\rho,
	\label{eqAW}
\end{equation}
where the operator $\mathcal{A}(\rho)$ is given by \eqref{eq:nonlinear_evolution_equation} via $\Pi$-periodic in $x$ and vanishing as $y\to-\infty$ solution of  \eqref{eq:p_bulk}--\eqref{eq:v_bdy}. 

In order to perform the bifurcation analysis of the flat front traveling wave solution ($\rho=0$) consider the linearized operator
\begin{equation}
	\mathcal{L}\tilde\rho= \partial_\rho\mathcal{A}(0)\tilde\rho.
\end{equation}
%One can easily see that 
By \eqref{eq:q_bulk}--\eqref{same_golovne_bc:w},
\begin{equation}
	\mathcal{L}\rho= \left(\partial_yV_y\rho+v_y\right)\bigr|_{y=0}=\frac{ \zeta}{2\mu}\rho +\bigl.v_y\bigr|_{y=0},
\end{equation}
where  %in order to evaluate $\bigl.v_y\bigr|_{y=0}$ one has to solve the coupled system  of
the (linearized) velocity  ${\bf v}=(v_x,v_y)$ and the (linearized) polarity ${\bf p}=(p_x,p_y)$ fields solve the  system
\begin{equation}
	\label{linearized_system}
	\begin{cases}
		\mu (\Delta  {\bf v}+\nabla {\rm div}{\bf v})-\zeta{\rm div} ({\bf p}\otimes {\bf P}+{\bf P}\otimes {\bf p})-\xi {\bf v}+\zeta_i {\bf p}=0\quad & \text{for}\ y<0,\\
		\Delta {\bf p}={\bf p}\quad &\text{for}\ y<0,\\ 
		\mu (\partial_x v_y+\partial_y v_x)= -\zeta \rho^\prime,\ \ \ 2\mu \partial_y v_y=-2\mu\partial^2_{yy}V_y\bigr|_{y=0}\rho+\gamma \rho^{\prime\prime} & \text{for}\ y=0,\\
		p_x=-\rho^\prime,\ p_y=-\partial_y P_y\bigr|_{y=0}\rho=-\rho & \text{for}\ y=0.
	\end{cases}
\end{equation} 
%The correspondence between the solution $\tilde {\bf v}$, $\tilde {\bf p}$ of \eqref{4}-\eqref{7} and the solution of  \eqref{linearized_system} is the following
%\begin{align*}
%\tilde {\bf v}(x,y)=\tilde \rho(x)\partial_y{\bf V}(y)+ {\bf v}(x,y),\quad
%\tilde {\bf p}(x,y)=\tilde \rho(x)\partial_y{\bf P}(y)+ {\bf p}(x,y).
%\end{align*}

Next, we study the spectral properties of the operator $\mathcal{L}$ that amounts to finding  solutions  ${\bf v}=e^{iqx}\hat {\bf v}(y)$, ${\bf p}=e^{iqx}\hat {\bf p}(y)$ (Fourier modes)  for $\rho(x) =e^{iqx}$,
so that 
\begin{equation}
	\label{lin opr}
	\mathcal{L}e^{iqx}=  \bigl(\textstyle{\frac{ \zeta}{2\mu}}+\hat v_y(0)\bigr)e^{iqx}.
\end{equation}
We have 
\begin{equation}
	p_x=-iq e^{iqx +\sqrt{q^2+1}y},\quad  p_y=-
	e^{iqx +\sqrt{q^2+1}y},
\end{equation}
while components of ${\bf v}$ satisfy for $y<0$ the equations
\begin{align}
	\label{lin pr1}
	\mu (\Delta  v_x&+\partial_x {\rm div }  {\bf v})-\xi v_x
	=iq \zeta_i e^{iqx +\sqrt{q^2+1}y}\nonumber\\ 
	&-iq\zeta(\sqrt{q^2+1}+1) e^{iqx +(\sqrt{q^2+1}+1)y},  \\
	%\end{align}
	%begin{align}
	\label{lin pr2}
	\mu (\Delta  v_y&+\partial_y {\rm div } {\bf v})-\xi  v_y=\zeta_i e^{iqx +\sqrt{q^2+1}y} \nonumber\\
	&+q^2\zeta e^{iqx +(\sqrt{q^2+1}+1)y} -2\zeta(\sqrt{q^2+1}+1)
	e^{iqx +(\sqrt{q^2+1}+1)y}
\end{align}
with boundary conditions on the line $y=0$,
\begin{align}
	&\mu (\partial_x v_y+\partial_y v_x)=-iq\zeta e^{iqx}, \label{lin bc1}\\
	&
	2\mu \partial_y v_y=-\zeta
	\textstyle{\Bigl(\frac{\xi}{4\mu+\sqrt{2\mu\xi}}+2\Bigr)} e^{iqx}+\zeta_i	\textstyle{\frac{\sqrt{2\mu}}{\sqrt{\xi}+\sqrt{2\mu}}}  e^{iqx}-q^2\gamma e^{iqx}.
	\label{lin bc2}
\end{align}
Here we have used the formula 
\begin{equation*}
	\partial^2_{yy}V_y \big|_{y=0}\bigr.= \frac{\xi\zeta}{2\mu\bigl(4\mu+\sqrt{2\mu\xi}\bigr)}-\frac{\zeta_i}{\sqrt{2\mu\xi}+2\mu}+\frac{\zeta}{\mu}
\end{equation*}
which follows from \eqref{UrForVy} and \eqref{V0}.

We represent the  solution of \eqref{lin pr1}--\eqref{lin bc2} as the sum  
\begin{equation}
	\label{tryKrysyvyhVektora}
	{\bf v}= \zeta_i {\bf v}^t +\zeta{\bf v}^c+ \gamma{\bf v}^s
\end{equation}
of three terms caused by the traction force, the contractile stress and the surface tension. More precisely, ${\bf v^t}$ solves \eqref{lin pr1}--\eqref{lin bc2} with $\zeta_i=1$, $\zeta=\gamma=0$; for ${\bf v^c}$ we set $\zeta=1$, $\zeta_i=\gamma=0$; for ${\bf v^s}$ we set $\gamma=1$, $\zeta=\zeta_i=0$. Solutions of the corresponding problems \eqref{split 1}, \eqref{split 2}, \eqref{split 3} are found explicitly in Appendix \ref{mister_first}, so that
\begin{equation}
	\bigl.v^t_y\bigr|_{y=0}=\Lambda^t(q,\mu,\xi)\, e^{iqx},\ \bigl.v^c_y\bigr|_{y=0}=\Lambda^c(q,\mu,\xi)\, e^{iqx},\ \bigl.
	v^s_y\bigr|_{y=0}=\Lambda^s(q,\mu,\xi)\, e^{iqx},
\end{equation}
where
\begin{equation}
	\label{Lambda_t}
	\begin{aligned}
		&\Lambda^t(q,\mu,\xi) = \frac{\xi \textstyle{\sqrt{q^2+\frac{\xi}{2\mu}}} \Bigl(  2\mu\bigl(1-\sqrt{q^2+1}\bigr)+\sqrt{2\mu \xi}\bigl(1 -\textstyle{\sqrt{1+\frac{2\mu q^2}{\xi}}}\bigr)\Bigr)}
		{D(q,\mu,\xi) \Bigl(\textstyle{\sqrt{q^2+1} +\textstyle{\sqrt{q^2+\frac{\xi}{2\mu}}}}\Bigr)\bigl(\sqrt{2\mu \xi}+2\mu\bigr)}
		\\
		&
		+\frac{\xi q^2\bigl(1-\sqrt{q^2+1}\bigr) \textstyle{\sqrt{q^2+\frac{\xi}{2\mu}}}}
		{D(q,\mu,\xi)
			\Bigl(\textstyle{\sqrt{q^2+\frac{\xi}{\mu}} +\sqrt{q^2+\frac{\xi}{2\mu}}}\Bigr)
			\Bigl(\sqrt{q^2+1} +\textstyle{\sqrt{q^2+\frac{\xi}{2\mu}}}\Bigr)\Bigl(\textstyle{\sqrt{q^2+\frac{\xi}{\mu}}} +\sqrt{q^2+1}\Bigr)}
	\end{aligned}
\end{equation}
%\textcolor{red}{$D<0, \ D\sim -\mu\xi q^2\quad\text{as} \  |q|\to \infty, \quad \bigl. D \bigr|_{q=0}=-\xi^2$,}
\begin{equation}
	\label{Lambda_c}
	\begin{aligned}
		\Lambda^c(q,\mu,\xi) = &\frac{\xi{\textstyle{\sqrt{q^2+\frac{\xi}{2\mu}}}}}
		{D(q,\mu,\xi) \left(\sqrt{q^2+1} +1+\textstyle{\sqrt{q^2+\frac{\xi}{2\mu}}}\right)}
		\Biggl(
		q^2-2\sqrt{q^2+1}-2\Biggr.
		\\
		%\frac{\xi \textstyle{\sqrt{q^2+\frac{\xi}{2\mu}}} \Bigl(  2\mu\bigl(1-\sqrt{q^2+1}\bigr)+\sqrt{2\mu \xi}\bigl(1 -\sqrt{1+2\mu q^2/\xi}\bigr)\Bigr)}{D
			%\left(\sqrt{q^2+1} +\sqrt{q^2+\xi/(2\mu)}\right)(\sqrt{2\mu \xi}+2\mu)}
		%\\
		%&
		%+\frac{\xi q^2\bigl(1-\sqrt{q^2+1}\bigr) \textstyle{\sqrt{q^2+\frac{\xi}{2\mu}}}}{D
			%\Bigl(\textstyle{\sqrt{q^2+\frac{\xi}{\mu}} +\sqrt{q^2+\frac{\xi}{2\mu}}}\Bigr)
			%\Bigl(\sqrt{q^2+1} +\textstyle{\sqrt{q^2+\frac{\xi}{2\mu}}}\Bigr)\Bigl(\textstyle{\sqrt{q^2+\frac{\xi}{\mu}}} +\sqrt{q^2+1}\Bigr)}
		&
		+\Biggl.\frac{2 q^4}%\xi\sqrt{q^2+\xi/(2\mu)}}
	{%D
		\bigl({\sqrt{q^2+1}+1+\textstyle{\sqrt{q^2+\frac{\xi}{\mu}}}}\bigr) 
		%(\textstyle{\sqrt{q^2+\frac{\xi}{2\mu}}}+\sqrt{q^2+1}+1)
		\bigl(\textstyle{\sqrt{q^2+\frac{\xi}{2\mu}}}+\textstyle{\sqrt{q^2+\frac{\xi}{\mu}}}\bigr)}\Biggr)
	\\
	&
	+\frac{\xi \textstyle{\sqrt{q^2+\frac{\xi}{2\mu}}}}{D(q,\mu,\xi)}\Biggr(\textstyle{\frac{\xi}{4\mu +\sqrt{2\mu\xi}}}
	+%\frac{\sqrt{q^2+\xi/(2\mu)}}{D} 
	2-\frac{ q^2} {\textstyle{\sqrt{q^2+\frac{\xi}{2\mu}} }+\textstyle{\sqrt{q^2+\frac{\xi}{\mu}}}}
	\Biggr),
	%\quad \text{(from the boundary)}
	%\\
	%&
	%\textcolor{red}{+\text{term from velocity $=1/(2\mu)$ ?}} 
	%\\
	%&\to \frac{1}{4\mu} \ \text{as $|q|\to \infty$}
\end{aligned}
\end{equation}
\begin{equation}
\label{Lambda_s}
\Lambda^{s}(q,\mu,\xi)=\frac{\xi \textstyle{\sqrt{q^2+\frac{\xi}{2\mu}}}}{D(q,\mu,\xi)}q^2,
\end{equation}
\begin{equation}
\label{det_D}
D(q,\mu,\xi)=4\mu^2q^2\textstyle{\sqrt{q^2+\frac{\xi}{2\mu}}}\textstyle{\sqrt{q^2+\frac{\xi}{\mu}}}- (2\mu q^2+\xi)^2,
%\textcolor{red}{D<0, \ D\sim -\mu\xi q^2\quad\text{as} \  |q|\to \infty, \quad \bigl. D \bigr|_{q=0}=-\xi^2
\end{equation}

Thus, according to \eqref{lin opr} we have the following formula for the eigenvalue of the linearized operator corresponding to the mode $e^{iqx}$:
\begin{equation}
	\label{Lambda_formula}
	\Lambda(q, \mu,\zeta, \zeta_i, \xi, \gamma)= \zeta_i \Lambda^{t}(q,\mu,\xi)+\textstyle{\frac{\zeta}{2\mu}}\left(1+2\mu\Lambda^{c}(q,\mu,\xi)\right)+\gamma\Lambda^{s}(q,\mu,\xi).   
\end{equation}
For simplicity hereafter we omit dependence on the parameters $\mu,\zeta, \zeta_i, \xi, \gamma$ and write $\Lambda(q)$ and $\Lambda^{t,c,s}(q)$. Notice that $\Lambda^{c}(q) \to -\textstyle{\frac{1}{4\mu}}$,  $\Lambda^{s}(q)=-\textstyle{\frac{|q|}{\mu}}+ O(\textstyle{\frac{1}{|q|}})$ and $\Lambda^{t}(q)=O(\textstyle{\frac{1}{q^2}})$ as $|q|\to \infty$, i.e. 
\begin{equation}
	\Lambda(q)= -\textstyle{\frac{\gamma}{\mu}}|q|  +\textstyle{\frac{\zeta}{4\mu}} + O(\textstyle{\frac{1}{|q|}}) \quad \text{as}\ |q|\to \infty.
	\label{asymp_lambda}
\end{equation}
Expanding 	$\Lambda$ in a neighborhood of $q=0$ we get
\begin{equation}
	\Lambda(q)= \left(\textstyle{\frac{\mu\zeta_i}{\xi(\sqrt{2\mu\xi}+2\mu)}}
	+\textstyle{\frac{\zeta}{2\mu\xi}}\bigl(3-2\sqrt{2}-\frac{2\sqrt{2\mu}(\sqrt{\xi}+\sqrt{2\mu})}{(2\sqrt{2\mu}+\sqrt{\xi})^2}\bigr)- \textstyle{\frac{\gamma}{\sqrt{2\mu\xi}}}\right)q^2+O(q^4).
\end{equation}
Thus in a neighborhood of $q=0$,  $\Lambda>0$ for  large enough $\zeta_i$  while  $\Lambda\to-\infty$ when $|q|\to \infty$. Then there exists a non-zero root $q=q_0>0$ of the transcendental equation 
\begin{equation}
	\Lambda(q)=0.
	\label{eq_for_q0}
\end{equation}
We show below that this $q_0$ defines a critical period $\Pi_0=2\pi/q_0$ 
for which a bifurcation of nontrivial traveling wave solutions occurs. 
%$\Pi_0=\textstyle{\frac{2\pi}{q_0}}$

%\section{Rescaling of the problem}
\section{Bifurcation of traveling wave solutions}
\label{Bifur}
Let $q_0>0$ be a solution of \eqref{eq_for_q0}. We apply the celebrated theory of bifurcation from the simple eigenvalue (see Theorem \ref{CR_theorem}) 
%\cite{CranRab} 
to show that there emerges a family of nontrivial traveling wave solutions with periods close to the critical period $\Pi_0=2\pi/q_0$.

It is convenient to pass from the prescribed period $\Pi$ to another bifurcation parameter $\theta=\Pi/\Pi_0$ (scaling factor). Introduce new coordinates $x=\theta\tilde x$, $y=\theta\tilde y$ and change the unknowns  
%$	\rho(x)$ 
%with 
$\tilde \rho=\frac{1}{\theta} \rho (\theta \tilde x)$,  
%${\bf v}(x,y)$ by 
$\tilde{\bf v}=\frac{1}{\theta}{\bf v}(\theta\tilde  x,\theta\tilde  y)$ and  
%${\bf p}(x,y)$ by
$\tilde{\bf p}={\bf p}(\theta\tilde  x, \theta\tilde  y)$.
%\begin{equation}
%	x'\to \theta x,\quad y\to \theta y,\quad v\to \theta v, \quad
% 	\rho\to \theta \rho
%\end{equation}
This allows us to reduce the analysis to the fixed period $\Pi_0$, while the parameter $\theta$ appears in the rescaled version of problem \eqref{eq:p_bulk}-\eqref{eq:v_bdy} (where we drop $\tilde \ $ to simplify the notation)
\begin{equation}
	\label{rescaled_system}%\label{eqAtheta} 
	%\begin{cases}
	\begin{aligned}
		&\Delta {\bf p}=\theta^2{\bf p}\quad &\text{for}\ y<\rho(x),\\ 
		&{\bf p}={\bf n} & \text{for}\ y=\rho(x),\\ 
		&\mu (\Delta  {\bf v}+\nabla {\rm div}{\bf v})-\zeta {\rm div} ({\bf p}\otimes {\bf p})-\theta^2\xi {\bf v}+\theta\zeta_i {\bf p}=0\quad & \text{for}\ y<\rho(x),\\
		&\mu (\nabla {\bf v}+(\nabla {\bf v})^T ){\bf n} =(\zeta-\frac{\gamma}{\theta}\kappa){\bf n} & \text{for}\ y=\rho(x).
	\end{aligned}		
	%\end{cases}
\end{equation}
Then the problem \eqref{eqAW} of finding traveling wave solutions is equivalent to % its rescaled version 
\begin{equation}
	\label{eqAtheta} 
	\mathcal{A}(\rho,\theta)=\left(v_y-v_x\rho'\right)\bigr|_{y=\rho(x)}=C_\rho,
\end{equation} 
where %the operator $\mathcal{A}(\rho,\theta)$ is still given by \eqref{eq:nonlinear_evolution_equation} with 
$C_\rho$ is an unknown constant, and ${\bf v}$, ${\bf p}$ solve \eqref{rescaled_system}.
%This %change of the variables
%leads to modification of the operator $\mathcal{A}(\rho)$ by factor $\theta$ 
%\begin{align}
%	\mathcal{A}(\rho,\theta)=\theta {\bf v}_y-\rho'\theta {\bf v}_x
%\end{align}

%For $\theta=1$ the linearized operator $\mathcal{L}(\theta)\rho=d_\rho \mathcal{A}(0,\theta;\rho)$ has an eigenvalue $\Lambda=0$ whose corresponding eigenfunction is $\rho=e^{iq_0x}$. 
Notice that the linearized operator $\mathcal{L}(\theta)$ has the eigenvalue $\Lambda\left(\frac{q_0}{\theta},\mu,\zeta, \zeta_i, \xi, \gamma\right)$ given by \eqref{Lambda_formula} whose corresponding eigenfunction is $\rho=e^{iq_0x}$. 
This eigenvalue becomes zero for  $\theta=1$. In this case, however, zero is a multiple eigenvalue since $\rho=1$ and $\rho=e^{-iq_0x}$ are also eigenfunctions. To get rid of this multiplicity issue observe that 
%indeed the critical value corresponding to bifucation of non-flat traveling waves we notice that 
$\mathcal{A}(\rho+C,\theta)=\mathcal{A}(\rho,\theta)$, %for every constant 
$\forall C={\rm const}$. Therefore we can pass to the quotient spaces $C^{k,\delta}_{ \#}(\Pi_0)/ \mathbb{R}$ and $C^{k-1,\delta}_{ \#}(\Pi_0)/ \mathbb{R}$ identifying constant functions with zero. This eliminates the eigenfunction $\rho=1$. Furthermore, the multiplicity can be reduced to one by assuming the natural symmetry $\rho(x)=\rho(-x)$.  
%, so that \eqref{eqAW} can be rewritten as $\mathcal{A}(\rho)=0$ (mod $\mathbb{R}$).

In view of above mentioned we can apply the following

\begin{theorem}[Crandall-Rabinowitz \cite{CranRab}]\label{CR_theorem}
	Let $X,Y$ be Banach spaces. Let $U\subset X$ be a neighborhood of \,$0$ and let
	\begin{equation}
		\Phi:U\times(1-\theta_0,1+\theta_0)\to Y
	\end{equation}
	have the following properties:
	\begin{itemize}
		\item[(i)] $\Phi(0,\theta)=0$ for all $\theta\in(1-\theta_0,1+\theta_0)$,
		\item[(ii)] $\Phi \in C^2 (U\times(1-\theta_0,1+\theta_0))$,
		%\item[(ii)] The derivatives $\partial_x\Phi$, $\partial_\theta \Phi$, $\partial^2_{x x}\Phi$ %and $\partial^2_{\theta x}\Phi$ exist and are continuous in $U\times(1-\theta_0,1+\theta_0)$,
		\item[(iii)] ${\rm dim \, Ker}\left(\partial_x\Phi(0,1)\right)={\rm codim \, Ran}\left(\partial_x\Phi(0,0)\right)=1$,
		%\item[(iii)] ${\rm Ker}\left(\partial_x\Phi(0,1)\right)\subset X$ is one-dimensional, %and ${\rm Range}\left(\partial_x\Phi(0,0)\right)$ is closed in $Y$ and has codimension 1, and
		\item[(iv)] $\partial^2_{\theta x}\Phi(0,1)x_0\not\in{\rm Ran}\left(\partial_x\Phi(0,1)\right)$ where
		${\rm Ker}\left(\partial_x\Phi(0,1)\right)={\rm Span}\{x_0\}$.
	\end{itemize}
	Then if $\tilde X$ is any complement of ${\rm Span}\{x_0\}$ in $X$, there exists $\varepsilon>0$ and continuously differentiable functions $\psi:(-\varepsilon,\varepsilon)\to\mathbb R$ and $\phi:(-\varepsilon,\varepsilon)\to \tilde X$ such that $\phi(0)=0$, $\psi(0)=0$, and $\Phi(\alpha x_0+\alpha\phi(\alpha), 1+\psi(\alpha))=0\quad\forall\alpha\in(-\varepsilon,\varepsilon)$.
	Moreover, 	$\Phi^{-1}(\{0\})$ near $ (0,1)$ consists precisely of the curves 
	$x=0$ and $(\alpha x_0+\alpha\phi(\alpha), 1+\psi(\alpha))$, $\alpha\in(-\varepsilon,\varepsilon)$ 
	
	% there is a neighborhood $U_0\subset X\times(1-\theta_0,1+\theta_0)$ of $(0,1)$ such that
	%	\begin{equation}		\Phi^{-1}(0)\cap U_0=\{(\alpha x_0+\alpha\phi(\alpha), 1+\psi(\alpha)):-\varepsilon<\alpha<\varepsilon\}\cup\{(0,\theta):(0,\theta)\in U_0\}.\notag	\end{equation}
\end{theorem}

%In what follows we will write  $\Lambda(q)$ in place of $\Lambda(q, \mu,\zeta, \zeta_i, \xi, \gamma)$ for simplicity.
%\textcolor{red}{Using the above described functional setting we prove the following result.}

Using Theorem \ref{CR_theorem} we establish bifurcation of nontrivial traveling wave solutions for the problem \eqref{eq:p_bulk}--\eqref{eq:v_bdy},  \eqref{eqAW}.

\begin{theorem}\label{main_theorem}
	Assume that equation \eqref{eq_for_q0} has a root $q_0>0$ and $\Lambda(jq_0)\not=0$ for $j=2, 3, \dots$ Assume also that $\partial_q\Lambda(q_0)\not=0 $. Then there is a  family of nontrivial (non-flat) traveling wave solutions $\rho = \rho(x, \alpha)$ of \eqref{eqAW} with periods $\Pi=2\pi\theta(\alpha)/q_0 $, depending on a small parameter $\alpha$. Moreover, %$\rho (\,\cdot\, , \alpha)\in C^\infty(\Pi)$ and 
	$\rho(x,\alpha)$ and $\theta(\alpha)$ smoothly depend on the parameter $\alpha$ and $\rho(x,0)=0$, $\theta(0)=1$. 
	%\begin{equation}
	%\end{equation}
\end{theorem}
\begin{proof}
	Recall that %the bifurcation analysis of 
	problem \eqref{eqAW}, where the prescribed period  $\Pi$ is considered as a bifurcation parameter, is reduced via rescaling with the factor $\theta>0$ to equation \eqref{eqAtheta} with fixed period $\Pi_0=2\pi/q_0$.   
	Observe that for every even function $\rho\in C^{k,\delta}_{ \#}(0,\Pi_0)$ there is a unique $\Pi_0$-periodic in $x$ and vanishing as $y\to-\infty$ solution ${\bf v}$, ${\bf p}$ of the rescaled problem \eqref{rescaled_system}, and the symmetry 
	\begin{align*}
		v_y(-x, y)&=v_y(x, y), \quad \quad p_y(-x, y)=p_y(x, y),\\
		v_x(-x, y)&=-v_x(x, y), \quad \, p_x(-x, y)=-p_x(x, y)
	\end{align*}
	holds. Thus %operators $\mathcal{A}(\rho, \theta)$ map even functions to even ones. Also $\mathcal{A}(\rho, \theta)$
	we can apply Theorem \ref{CR_theorem} to the family of operators $\mathcal{A}(\rho, \theta)$ with $X$ and $Y$ being  subspaces of $C^{k,\delta}_{ \#}(0,\Pi_0)/\mathbb{R}$ and $C^{k-1,\delta}_{ \#}(0,\Pi_0)/\mathbb{R}$  ($0<\delta<1, k=2, 3, \dots$)  
	of even functions. 
	
	The flat front traveling wave solution for $\rho=0$ constructed in Section \ref{sec2_ffs}  satisfies %\textcolor{red}{fulfils} 
	the condition $(i)$ of Theorem \ref{CR_theorem}.  
	By virtue of Theorem \ref{th:oper_prop}, $(ii)$ is also satisfied.
	Since the linearized  operator $\mathcal{L}(\theta)$ has the following spectral representation: 
	\begin{align}
		\mathcal{L}(\theta): \cos jq_0x\mapsto \Lambda \left(\textstyle{\frac{jq_0}{\theta}}\right)\cos jq_0x, \quad j=1,2,\dots,
		\label{spectral_representation}
	\end{align} 
	the kernel of $\mathcal{L}=\mathcal{L}(1)$ is one-dimensional and is spanned by $\{\cos q_0x\}$. 
	
	We claim that
	\begin{equation}
		\text{Ran}\left(\mathcal{L}\right)=\left\{\rho \in C^{k-1,\delta}_{ \#}(0,\Pi_0)/\mathbb{R}\, \big| \, \text{$\rho$ is even,}\ \int_{\Pi_0}\rho \cos q_0xdx=0\right\}.
		\label{rangeL}
	\end{equation}
	Indeed, consider the equation $\mathcal{L}\varrho =\rho$, where $\rho$ belongs to the space given by the right-hand side of \eqref{rangeL}. Then expanding $\rho$ into Fourier series $\rho=\sum_{j\geq 2} c_j\cos jq_0x$ we have
	\begin{equation}
		\label{varrho}
		\begin{aligned}
			\varrho%&
			=\sum_{j\geq 2} \frac{c_j}{\Lambda(jq_0)} \cos jq_0x %\\ 
			&=-\sum_{j\geq 2} \left( \frac{\mu}{\gamma q_0 j}  + \frac{\mu\zeta}{4(\gamma q_0 j)^2}\right) c_j{\cos jq_0x}\\
			&+\sum_{j\geq 2}\left(\frac{1}{\Lambda(jq_0)}
			+\frac{\mu}{\gamma q_0 j} + \frac{\mu\zeta}{4(\gamma q_0 j)^2}\right)c_j \cos jq_0x.
		\end{aligned}
	\end{equation}
	Let us show that $\varrho\in C^{k,\delta}_{ \#}(0,\Pi_0)$. It follows from \eqref{asymp_lambda} that the second term in the right-hand side of \eqref{varrho} belongs to $W^{k+2, 2}_{ \#}(0,\Pi_0)$ and hence to $C^{k,\delta}_{ \#}(0,\Pi_0)$. 
	Using the Sokhotski-Plemelj formulas another term can be represented as 
	\begin{align}
		- \frac{\mu}{\gamma}\mathcal{K}\rho-	\frac{\mu\zeta}{4\gamma^2}\mathcal{K}^2\rho,   \quad \text{where}\   \mathcal{K}\rho= \frac{1}{\Pi_0} \text{p.v.}\int_{0}^{\Pi_0}\cot {\textstyle\frac{q_0(x-z)}{2}} \int_{0}^{z}\rho(s)dsdz.
		\label{hilbert+}
	\end{align}
	Since the Hilbert transform involved in \eqref{hilbert+} continuously maps  $C^{k,\delta}_{ \#}(0,\Pi_0)$ to $ C^{k,\delta}_{ \#}(0,\Pi_0)$, the first term of right-hand side of \eqref{varrho} also belongs to $ C^{k,\delta}_{ \#}(0,\Pi_0)$. %(Alternatively)
	
	And finally, the transversality condition  $(iv)$ %of Theorem \ref{CR_theorem} 
	is satisfied since
	\begin{align}
		\partial_\theta\mathcal{L}(1) \cos q_0x=-q_0 \partial_q \Lambda(q_0)\cos q_0x\not\in
		\text{Ran}\left(\mathcal{L}\right).
	\end{align}
	Thus all of the conditions of Theorem \ref{CR_theorem} are fulfilled. Also according to Theorem 1.18 from \cite{CranRab}, $\rho(x,\alpha)$ and $\theta(\alpha)$ are infinitely differentiable in $\alpha$ functions.   
\end{proof}

\begin{remark}\label{rem_even}
	Note that traveling wave problem \eqref{eqAtheta} is invariant with respect to shifts in $x$-axis, moreover, for any even solution $\rho$ of \eqref{eqAtheta}, its shift by the half-period is still an even solution. Thus we can assume that %after reparametrizing  
	$\forall \alpha$ $\rho(x,-\alpha)=\rho(x-\Pi_0/2,\alpha)$, $\theta(-\alpha)=\theta(\alpha)$.
\end{remark}

Now we address the issue of stability of traveling wave solutions. To this end we apply the following result obtained in \cite{CranRab73}. 

\begin{theorem} \label{CrRabSecond}
	Assume that conditions of Theorem \ref{CR_theorem} are fulfilled and $X\subseteq Y$ with continuous embedding. Then for sufficiently small $\alpha$ there exists the smallest in absolute value simple eigenvalue $\lambda (\alpha)$ of the linearized operator $\partial_x\Phi(\alpha x_0+\alpha\phi(\alpha),1+\psi(\alpha))$ and 
	\begin{align}
		\lambda (\alpha)= -\alpha\tilde\lambda'(1) \theta'(\alpha)(1+O(\alpha))  \ \text{as } \ \alpha \to 0,
	\end{align}	 
	where $\tilde\lambda(\theta)$ is the smallest in absolute value  eigenvalue  of $\partial_x\Phi(0,\theta)$.
\end{theorem}

%Considering generic  $\Pi_0$-periodic  perturbations 
%We have the following structure of the spectrum of
Consider traveling wave solution and linearized operator $\partial_\rho {\mathcal A}(\rho,\theta(\alpha))$ %representing linearization of  ${\mathcal A}(\rho,\theta(\alpha))$ 
near the bifurcation point, i.e. when $|\alpha|$ is small. 
%for sufficiently small $\alpha$:
Its spectrum has the following structure. There is a zero eigenvalue of multiplicity two corresponding to infinitesimal shifts with eigenfunctions equal $1$ (vertical shifts) and $\rho^\prime$ (horizontal shifts) respectively. By Theorem \ref{CrRabSecond} the smallest in absolute value nonzero eigenvalue  of $\partial_\rho {\mathcal A}(\rho,\theta(\alpha))$  is given by the asymptotic formula
%Applying Theorem \ref{CrRabSecond} to the operator linearized around traveling wave solutions of 	\eqref{rescaled_system}--\eqref{eqAtheta} we see that its smallest in absolute value  eigenvalue $\lambda (\alpha)$ is given by the asymptotic formula
\begin{align}
	\lambda (\alpha)=q_0 \partial_q \Lambda(q_0) \alpha\theta'(\alpha)(1+O(\alpha))  \ \text{as } \ \alpha \to 0,
	\label{asymptotic_lambda}
\end{align}
so the sign of $\lambda (\alpha)$ is determined by that of $\alpha\theta'(\alpha)$. Other eigenvalues either remain bounded and converge as $\alpha\to 0$ to those of $\partial_\rho {\mathcal A}(0,1)$, or have a negative sufficiently large in absolute value imaginary part (and therefore do not affect stability).

%In what follows we will write  $\Lambda(q)$ in place of $\Lambda(q, \mu,\zeta, \zeta_i, \xi, \gamma)$ for simplicity.
%\textcolor{red}{Using the above described functional setting we prove the following result.}

\section{Asymptotic expansions of traveling wave solutions near the bifurcation point}
\label{Expansions}
Let $q_0>0$ be a solution of \eqref{eq_for_q0}, and assume that other conditions of Theorem \ref{main_theorem} are satisfied. Then we have a family of $\Pi_0$-periodic ($\Pi_0=2\pi/q_0$) traveling wave solutions $\rho(x,\alpha)$ of problem \eqref{eqAtheta}.  The corresponding polarisation and velocity fields, ${\bf p}(x,y,\alpha)$, ${\bf v}(x,y,\alpha)$ are unique solutions of \eqref{rescaled_system}.  

For $\alpha=0$ we have flat front solution $\rho(x,0)=0$, $\theta(0)=1$, and ${\bf p}(x,y,0)={\bf P}(y)$,  ${\bf v}(x,y,0)={\bf V}(y)$, where ${\bf P}(y)=(0, e^y)$ and ${\bf V}(y)= (0,V_y(y))$ with $V_y(y)$ given by \eqref{Vy}. Now, for small $\alpha$ we consider asymptotic expansions
\begin{align}
	\theta(\alpha)&=1+b\alpha^2+O(\alpha^4)\label{expan_theta}\\ 
	C_\rho&=V^{(0)}+\alpha^2 V^{(2)}+O(\alpha^4),\label{expan_Crho}\\ 
	\rho(x,\alpha)&=\alpha \rho^{(1)}(x)+\alpha^2\rho^{(2)}(x)+\alpha^3\rho^{(3)}(x)+O(\alpha^4),\\ 
	{\bf p}(x,y,\alpha)&={\bf P}(y)+\alpha {\bf p}^{(1)}(x,y)+\alpha^2 {\bf p}^{(2)}(x,y)+\alpha^3 {\bf p}^{(3)}(x,y)+O(\alpha^4),\label{expan_p}\\ 
	{\bf v}(x,y,\alpha)&={\bf V}(y)+\alpha {\bf v}^{(1)}(x,y)+\alpha^2 {\bf v}^{(2)}(x,y)+\alpha^3 {\bf v}^{(3)}(x,y)+O(\alpha^4),\label{expan_v}
\end{align}
where $V^{(0)}$ is given by \eqref{V0}, $\rho^{(k)}(x)$ are even functions with zero mean value, and in view of Remark \ref{rem_even}, the expansions \eqref{expan_theta}--\eqref{expan_Crho} contain only even powers of $\alpha$. Our main objective here is to identify the sign of the parameter $b$ that %determines 
is decisive for the stability of the traveling wave solution.

We substitute the expansions above into \eqref{rescaled_system}--\eqref{eqAtheta} and equate the terms corresponding to the same powers of $\alpha$. %At zero order we find the flat solution.  
At the order $\alpha$ we arrive at the linearized problem \eqref{linearized_system} for ${\bf p}={\bf p}^{(1)}$ and ${\bf v}={\bf v}^{(1)}$ coupled with the equation $\mathcal{L}\rho^{(1)}= 0$ for $\Pi_0$-periodic even fuction $\rho^{(1)}$. According to spectral analysis of the operator $\mathcal{L}$ (Section \ref{prop_A}) we have, up to a multiplicative constant, 
\begin{align}
	\rho^{(1)}=\cos q_0 x,
	\label{rho1}
\end{align} 
and
\begin{align}
	\label{p_1}
	p_x^{(1)}= q_0 e^{y\sqrt{q_0^2+1}} \sin q_0 x, \ p_y^{(1)}=- e^{y\sqrt{q_0^2+1}} \cos q_0 x,
	%&{\bf v}^{(1)}\ \text{is given by \eqref{verylongformula}}, \notag\\
	%&v_y^{(1)}\big|_{y=0}\bigr.=\Lambda(q_0,\xi)\cos q_0x-\frac{\zeta}{2\mu}\cos q_0x \notag
\end{align}
while ${\bf v}^{(1)}$ is the real part of ${\bf v}= \zeta_i {\bf v}^t +\zeta{\bf v}^c+ \gamma{\bf v}^s$ explicitely found in Appendix \ref{mister_first}.

Next, equating terms of the order $\alpha^2$  in \eqref{rescaled_system} we find the following system
\begin{align}
	%	&\rho^{(2)}=\beta\cos 2q_0 x, \notag\\
	&\Delta {\bf p}^{(2)}={\bf p}^{(2)}+2b {\bf P} \quad &\text{for}\ y<0,\\
	&p^{(2)}_x=-\partial_y p_x^{(1)}\rho^{(1)}-[\rho^{(2)}]^\prime\quad  &\text{for}\ y=0,\\
	&p^{(2)}_y=-\partial_y P_y\rho^{(2)}-{\textstyle\frac{1}{2}}\partial^2_{yy} P_y[\rho^{(1)}]^2-  \partial_y p_y^{(1)}\rho^{(1)}
	-{\textstyle \frac{1}{2}}{[\rho^{(1)}]^\prime}^2 \quad &\text{for}\ y=0,
\end{align}
\begin{align}
	&\mu (\Delta {\bf v}^{(2)}+\nabla {\rm div}{\bf v}^{(2)})-\zeta {\rm div}\bigl({\bf p}^{(2)}\otimes {\bf P}+{\bf P}\otimes {\bf p}^{(2)} \bigr)\nonumber
	\\ 
	&\label{v^2_eq}\quad\quad -\zeta {\rm div}\bigl({\bf p}^{(1)}\otimes {\bf p}^{(1)}\bigr)
	-\xi {\bf v}^{(2)}-2\xi b {\bf V}
	+\zeta_i {\bf p}^{(2)} +\zeta_i b {\bf P}=0\quad \quad \text{for}\ y<0,\\
	&\mu\bigl(\partial_x v_y^{(2)}+\partial_y v_x^{(2)}\bigr)= 
	2\mu \partial_x v_x^{(1)}[\rho^{(1)}]^\prime -\mu \bigl(\partial^2_{xy}v_y^{(1)}+\partial^2_{yy}v_x^{(1)}\bigr)\rho^{(1)}
	\nonumber\\
	&\quad\quad\quad\quad\quad\quad\quad\quad-\zeta [\rho^{(2)}]^\prime -\gamma[\rho^{(1)}]^\prime[\rho^{(1)}]^{\prime\prime}
	\quad\quad\quad\quad\quad\quad\quad\quad\quad \ 
	\text{for}\ y=0,\label{v^2_boundarycondition_first}\\
	&2\mu \partial_yv_y^{(2)}= \mu \bigl(\partial_xv_y^{(1)}+\partial_yv_x^{(1)}\bigr)[\rho^{(1)}]^\prime
	-2\mu \partial^2_{yy}V_y \rho^{(2)}-\mu \partial^3_{yyy}V_y [\rho^{(1)}]^2 \nonumber\\
	&\quad\quad\quad\quad-2\mu \partial^2_{yy}v_y^{(1)}\rho^{(1)}%-\frac{1}{2}\zeta {[\rho^{(1)}]^\prime}^2
	+\gamma[\rho^{(2)}]^{\prime\prime}\quad\quad\quad\quad\quad\quad
	\quad\quad\quad\quad\quad\quad\quad
	\text{for}\ y=0,\label{v^2_boundarycondition_second}
\end{align}
and from \eqref{eqAtheta} we get
\begin{align}
	v^{(2)}_y+\partial_y V_y \rho^{(2)} =-\textstyle{\frac{1}{2}} \partial^2_{yy} V_y [\rho^{(1)}]^2+v_x^{(1)}[\rho^{(1)}]^\prime-\partial_y v_y^{(1)} \rho^{(1)}+V^{(2)}
	\quad
	\text{for}\ y=0.  
	\label{hochu}
\end{align}

Observe that   ${\bf p}^{(2)}$ can be represented as%the second term of the expansion \eqref{expan_p}  as 
%satisfies
%\begin{equation}
%\Delta p^{(2)}=p^{(2)}+2bP \quad \text{for}\ y<0,
%\end{equation}
\begin{equation}
	{\bf p}^{(2)}={\bf p}^{(21)}+ {\bf p}^{(22)},
	\label{splitting_p2}
\end{equation}
where 
%$p^{(21)}$ is a solution of
%\begin{equation}
%	\Delta p^{(21)}=p^{(21)}+2bP \quad \text{for}\ y<0,
%\end{equation}
\begin{align}
	&p^{(21)}_x=-{\textstyle\frac{q_0\sqrt{q_0^2+1}}{2}} \, e^{y\sqrt{4q_0^2+1}} \,\sin2q_0x, \label{p_21x}\\ 
	&p^{(21)}_y={\textstyle\frac{q_0^2+2\sqrt{q_0^2+1}-1}{4}}\, e^{y\sqrt{4q_0^2+1}} \cos2q_0x+
	{\textstyle\Bigl( by+\frac{2\sqrt{q_0^2+1}-q_0^2-1}{4}\Bigr)}e^y,\label{p_21y}
\end{align}
and ${\bf p}^{(22)}$ solves the problem
\begin{equation}
	\label{p22}
	\begin{cases}
		\Delta{\bf p}^{(22)}={\bf p}^{(22)} \quad &\text{for}\ y<0,\\ 
		p^{(22)}_x=-[\rho^{(2)}]^{\prime}, \quad 
		p^{(22)}_y=-\rho^{(2)} \quad &\text{for}\ y=0.
	\end{cases}
\end{equation}
Bearing in mind \eqref{splitting_p2} we represent the vector ${\bf v}^{(2)}$ as 
%satisfies
%\begin{equation}
%\Delta p^{(2)}=p^{(2)}+2bP \quad \text{for}\ y<0,
%\end{equation}
\begin{equation}
	{\bf v}^{(2)}={\bf v}^{(21)}+ {\bf v}^{(22)},
\end{equation}
with ${\bf v}^{(21)}$ found explicitly in Appendix \ref{app_v_21}
and ${\bf v}^{(22)}$ satisfying
\begin{equation}
	\label{v22}
	\begin{cases}
		\mu (\Delta {\bf v}^{(22)}+\nabla {\rm div}{\bf v}^{(22)})-\xi {\bf v}^{(22)}\\ 
		\quad\quad\quad\quad -\zeta{\rm div}\bigl({\bf p}^{(22)}\otimes {\bf P}+{\bf P}\otimes {\bf p}^{(22)}\bigr)+\zeta_i {\bf p}^{(22)}=0\quad & \text{for}\ y<0,\\
		\mu \bigl(\partial_x v_y^{(22)}+\partial_y v_x^{(22)}\bigr)=-\zeta [\rho^{(2)}]^{\prime}, 
		&\text{for}\ y=0,  \\ 
		2\mu\partial_y v_y^{(22)}
		=-2\mu \partial^2_{yy} V_y\rho^{(2)}+\gamma [\rho^{(2)}]^{\prime\prime}  &\text{for}\ y=0.
	\end{cases}
\end{equation} 
Then from \eqref{p22}, \eqref{v22} we have
\begin{align}
	v_y^{(22)}=-\frac{\zeta}{2\mu}\rho^{(2)}+\mathcal{L}\rho^{(2)}. 
	\label{vy22withL} 
\end{align}
Substituting expressions $	\rho^{(1)}$, ${\bf v}^{(1)}$, $v_y^{(21)}$ into 
\eqref{hochu} and taking into account \eqref{vy22withL}
we conclude that $\mathcal{L}\rho^{(2)}\in Span\{1, \cos 2q_0x\}$. This in turn implies that $\rho^{(2)}=\beta\cos 2q_0 x$ and we find that constants $\beta$ and $V^{(2)}$ are given by formulas \eqref{beta} and \eqref{bagatoV}, 
see Appendix \ref{app_v_21}. Having established $\rho^{(2)}$ we get the following explicit form of ${\bf p}^{(22)}$:
\begin{equation}
	p_x^{(22)}=2\beta q_0e^{y\sqrt{4q_0^2+1}}\sin 2q_0x, \quad p^{(22)}_y=-\beta e^{y\sqrt{4q_0^2+1}}\cos 2q_0x.
	\label{p_22}
\end{equation}

Now considering term of the order $\alpha^3$ in \eqref{rescaled_system} we get %the following problems for ${\bf p}^{(3)}$, ${\bf v}^{(3)}$:
\begin{align}
	&\label{eq_p3} \Delta {\bf p}^{(3)}={\bf p}^{(3)}+2b{\bf p}^{(1)}, \quad\quad\quad\quad\quad\quad\quad\quad\quad\quad\quad\quad\quad\quad\quad\quad\quad
	\quad\quad\  \text{for}\ y<0,\\
	&p^{(3)}_x=-\partial_y p_x^{(1)}\rho^{(2)}
	-\partial_y p_x^{(2)}\rho^{(1)}-{\textstyle\frac{1}{2}}\partial^2_{yy} p_x^{(1)}[\rho^{(1)}]^2%\nonumber\\ 
	-[\rho^{(3)}]^\prime +{\textstyle\frac{1}{2}}{[\rho^{(1)}]^\prime}^3,%\quad  \text{for}\ y=0,
	\nonumber\\
	&p^{(3)}_y=-\partial_y p_y^{(1)}\rho^{(2)}-\partial_y p_y^{(2)}\rho^{(1)} -{\textstyle\frac{1}{2}}\partial^2_{yy} p_y^{(1)}\, [\rho^{(1)}]^2%\nonumber\\ 
	-\partial_y P_y\,\rho^{(3)} -\partial^2_{yy}P_y\, \rho^{(1)}\rho^{(2)}\nonumber\\  &\quad\quad\quad-{\textstyle\frac{1}{6}}\partial^3_{yyy}P_y\, [\rho^{(1)}]^3 
	-[\rho^{(1)}]^\prime[\rho^{(2)}]^\prime,  \quad\quad\quad\quad\quad\quad\quad\quad\quad\quad\quad\quad\text{for}\ y=0;\label{bc_p3}
\end{align}
\begin{align}
	&\mu (\Delta {\bf v}^{(3)}+\nabla {\rm div}{\bf v}^{(3)})
	-\zeta {\rm div}\, \bigl({\bf p}^{(3)}\otimes {\bf P}+{\bf P}\otimes {\bf p}^{(3)}+ {\bf p}^{(1)}\otimes{\bf p}^{(2)}+{\bf p}^{(2)}\otimes {\bf p}^{(1)} \bigr),\label{eq_v3}
	\nonumber\\ 
	&\quad\quad\quad 
	%-\zeta \nabla\,\cdot\, \bigl({\bf p}^{(1)}\otimes{\bf p}^{(2)}+{\bf p}^{(2)}\otimes {\bf p}^{(1)} \bigr)
	-\xi {\bf v}^{(3)}-2\xi b {\bf v}^{(1)}
	+\zeta_i {\bf p}^{(3)} + \zeta_i b {\bf p}^{(1)}=0\quad\quad \quad\quad\quad\quad\quad\quad\quad \text{for}\ y<0,
\end{align}
and for $y=0$ we have
\begin{align}
	&\mu\bigl(\partial_xv_y^{(3)}+\partial_yv_x^{(3)}\bigr)= 
	2\mu\partial^2_{xy}v_x^{(1)}\rho^{(1)}[\rho^{(1)}]^\prime+
	2\mu \partial_x v_x^{(1)}[\rho^{(2)}]^\prime+2\mu \partial_xv_x^{(2)}[\rho^{(1)}]^\prime
	\nonumber\\
	&\quad -\mu \bigl(\partial^2_{xy}v_y^{(1)}+\partial^2_{yy}v_x^{(1)}\bigr)\rho^{(2)}
	-{\textstyle\frac{\mu}{2}}\bigl(\partial^3_{xyy}v_y^{(1)}+\partial^3_{yyy}v_x^{(1)}\bigr)[\rho^{(1)}]^2
	\nonumber\\
	&\quad -\mu \bigl(\partial^2_{xy}v_y^{(2)}+\partial^2_{yy}v_x^{(2)}\bigr)\rho^{(1)}
	-\zeta [\rho^{(3)}]^\prime %\nonumber\\
	%&\quad\quad\quad-
	-\gamma[\rho^{(1)}]^{\prime\prime}[\rho^{(2)}]^\prime
	-\gamma[\rho^{(2)}]^{\prime\prime}[\rho^{(1)}]^\prime, %\quad\quad \text{for}\ y=0,
	\label{eto_est_predposlednii_cond}
\end{align}
\begin{align}
	&2\mu \partial_yv_y^{(3)}= \mu \bigl(\partial_xv_y^{(1)}+\partial_yv_x^{(1)}\bigr)[\rho^{(2)}]^\prime
	+\mu \bigl(\partial^2_{xy}v_y^{(1)}+\partial^2_{yy}v_x^{(1)}\bigr)\rho^{(1)}[\rho^{(1)}]^\prime 
	\nonumber\\
	&\quad+\mu\bigl(\partial_xv_y^{(2)}+\partial_yv_x^{(2)}\bigr)[\rho^{(1)}]^\prime
	-2\mu \partial^2_{yy}V_y \rho^{(3)}-2\mu \partial^3_{yyy}V_y \rho^{(1)}\rho^{(2)}
	\nonumber\\
	&\quad-{\textstyle\frac{\mu}{3}}  \partial^4_{yyyy}V_y [\rho^{(1)}]^3 -2\mu \partial^2_{yy}v_y^{(1)} \rho^{(2)}
	-\mu \partial^3_{yyy}v_y^{(1)}[\rho^{(1)}]^2-2\mu \partial^2_{yy}v_y^{(2)} \rho^{(1)}
	\nonumber\\
	&\quad+\gamma [\rho^{(3)}]''-\gamma b[\rho^{(1)}]''-{\textstyle\frac{3}{2}}\gamma {[\rho^{(1)}]^\prime}^2[\rho^{(1)}]'',%\quad\quad\quad\quad
	%\text{for}\ y=0,
	\label{eto_est_poslednii_cond}
\end{align}
Finally, collecting terms of the order $\alpha^3$ in \eqref{eqAtheta} we derive that for $y=0$
\begin{align}
	v_y^{(3)}&+\partial_yV_y \rho^{(3)}+\partial^2_{yy}V_y \rho^{(1)}\rho^{(2)}
	+{\textstyle\frac{1}{6}}\partial^3_{yyy}V_y [\rho^{(1)}]^3 +\partial_yv_y^{(1)}\rho^{(2)}
	+{\textstyle\frac{1}{2}}\partial^2_{yy}v_y^{(1)}[\rho^{(1)}]^2 \nonumber\\
	& +\partial_yv_y^{(2)}\rho^{(1)}
	-\partial_yv_x^{(1)}\rho^{(1)}[\rho^{(1)}]^\prime
	-v_x^{(2)}[\rho^{(1)}]^\prime-v_x^{(1)}[\rho^{(2)}]^\prime=0.%\quad  	\text{for}\ y=0.
	\label{fse} 
\end{align}
Similarly to the previous two steps the solution of \eqref{eq_p3}--\eqref{eto_est_poslednii_cond} %\eqref{fse} 
can be found in an explicit form. However, this task is rather cumbersome. Instead, we only identify the parameter $b$ appearing in the expansion  \eqref{expan_theta} of $\theta$, 
whose sign determines the bifurcation type, $b>0$ ($b<0$) corresponds to the supercritical (subcritical) bifurcation (see \eqref{asymptotic_lambda} and Figure \ref{fig:krasota}). 
%We notice that the coefficient of $\cos q_0x$ in \eqref{fse} is a linear function of $b$, $k_1b+k_2$, moreover 
We notice that calculations of the Fourier coefficient corresponding to $\cos q_0x$ in \eqref{fse} %of the function in the left hand side of 
give a linear function of $b$, $k_1b+k_2$, moreover 
\begin{align}
	k_1\cos q_0x= \partial_\theta\mathcal{L}(1)\cos q_0x=\partial_\theta\Lambda(q_0/\theta)\Bigr|_{\theta =1}\cos q_0x=-q_0 \partial_q \Lambda(q_0)\cos q_0x.
	\label{pokraschennya}
\end{align}
To find $k_2$ first we represent $	{\bf p}^{(3)}$, ${\bf v}^{(3)}$ as
\begin{equation}
	{\bf p}^{(3)}={\bf p}^{(31)}+ {\bf p}^{(32)},
	\quad 	{\bf v}^{(3)}={\bf v}^{(31)}+ {\bf v}^{(32)},
\end{equation}
where  ${\bf p}^{(32)}$ and ${\bf v}^{(32)}$ are solutions of problem \eqref{linearized_system} with $\rho=\rho^{(3)}$. Observe that $\bigl.\bigl(v_y^{(32)}+\partial_yV_y \rho^{(3)}\bigr)\bigr|_{y=0}\,$ is orthogonal to $\cos q_0x$ in $L^2(0,\Pi_0)$. 
Next we write ${\bf p}^{(31)}$ and ${\bf v}^{(31)}$ in the form
\begin{equation}
	{\bf p}^{(31)}={\bf p}^{(311)}+{\bf p}^{(312)},\quad {\bf v}^{(31)}={\bf v}^{(311)}+{\bf v}^{(312)},
\end{equation}
where %${\bf v}^{311}$ and ${\bf v}^{312}$ correspond to different Fourier modes,
\begin{equation*}
	{v}^{(311)}_x =\hat {v}^{(311)}_x(y) \sin q_0 x, \quad  {v}^{(311)}_y =\hat {v}^{(311)}_y(y) \cos q_0 x,
\end{equation*}
while $p^{(312)}_x$  and $v^{(312)}_x$ ($p^{(312)}_y$ and $v^{(312)}_y$) absorbs all the terms that contain the factor $\sin 3q_0x$ ($\cos 3 q_0x$) or/and $b$. 
An explicit form \eqref{v311_x}--\eqref{v311_y} of the vector function ${\bf v}^{(311)}$ is found in Appendixes \ref{predfinish}, \ref{finish}.
%, and $v^{(312)}_y$ contains all the terms that have the factor $\cos 3q_0x$  or/and $b$.
Then considering the Fourier coefficients of functions in \eqref{fse} we find (see Appendix \ref{finish_nutochnovzhe})  that for $y=0$
\begin{align}
	v_y^{(311)}& +\Bigl(\beta D_1+\textstyle{\frac{D_2}{8\mu\sqrt{4 q_0^2+{ \textstyle\frac{\xi}{2\mu}}}}}+\textstyle{\frac{1}{4\mu}}\Bigl( -\frac{\zeta}{2}+\zeta \sqrt{q_0^2+1} -\frac{\zeta_i}{8}+\frac{\zeta\xi}{16\mu}-\frac{5\zeta q_0^2}{8}-\beta \gamma q_0^2\Bigr)
	\Bigr. \notag\\
	&\Bigl.+\textstyle{\frac{3\beta}{4}}\Bigl(\frac{\xi\zeta}{2\mu(4\mu+\sqrt{2\mu \xi})} -
	\frac{\zeta_i}{\sqrt{2\mu\xi}+2\mu}+\frac{\zeta}{\mu}
	\Bigr)   -bq_0 \partial_q \Lambda(q_0)  \Bigr)\cos q_0 x =0, \label{ReLaTiOn} 
\end{align}
where $D_1$ and $D_2$ are defined by \eqref{I_znovu_zdraste_D1} and  \eqref{D2}.  Substituting \eqref{v311_y} into \eqref{ReLaTiOn} we get %the final 
equation \eqref{alligator} for finding $b$.

%\textcolor{green}{We are interested in  determining ${\bf v}^{(311)}$, which appears in the solvability condition. The latter condition
	%allows us to determine the coefficient $b$}. 

\section{Conclusions}
\label{conclusions_graf}

In this section, we present and discuss several numerical results relevant to 
the bifurcation of nonflat traveling waves.  
Computations are done 
for the typical value of the characteristic length $L_c=25\, \mu m$ \cite{AleBlaCas2019} by adjusting 
to the case of the unit length via a spatial rescaling.

The finger-like pattern in the shape of the traveling wave solution is shown in Figure~\ref{fig:krasotaaaaa}, where the approximate shape corresponding to 
the two-term expansion $\rho=\alpha \rho^{(1)} +\alpha^2 \rho^{(2)}$ ($\alpha=0.5$) is depicted. %Computations are performed.
 The shape is computed   
by using explicit formulas  \eqref{nyashechka} for $\rho^{(1)}$ and \eqref{veselytsya_i_psihue_ves_narod2}  for $\rho^{(2)}$, taking some typical values of parameters \cite{AleBlaCas2019}. 

\begin{figure}[h!]
	\begin{center}
		\includegraphics[width=0.5\textwidth]{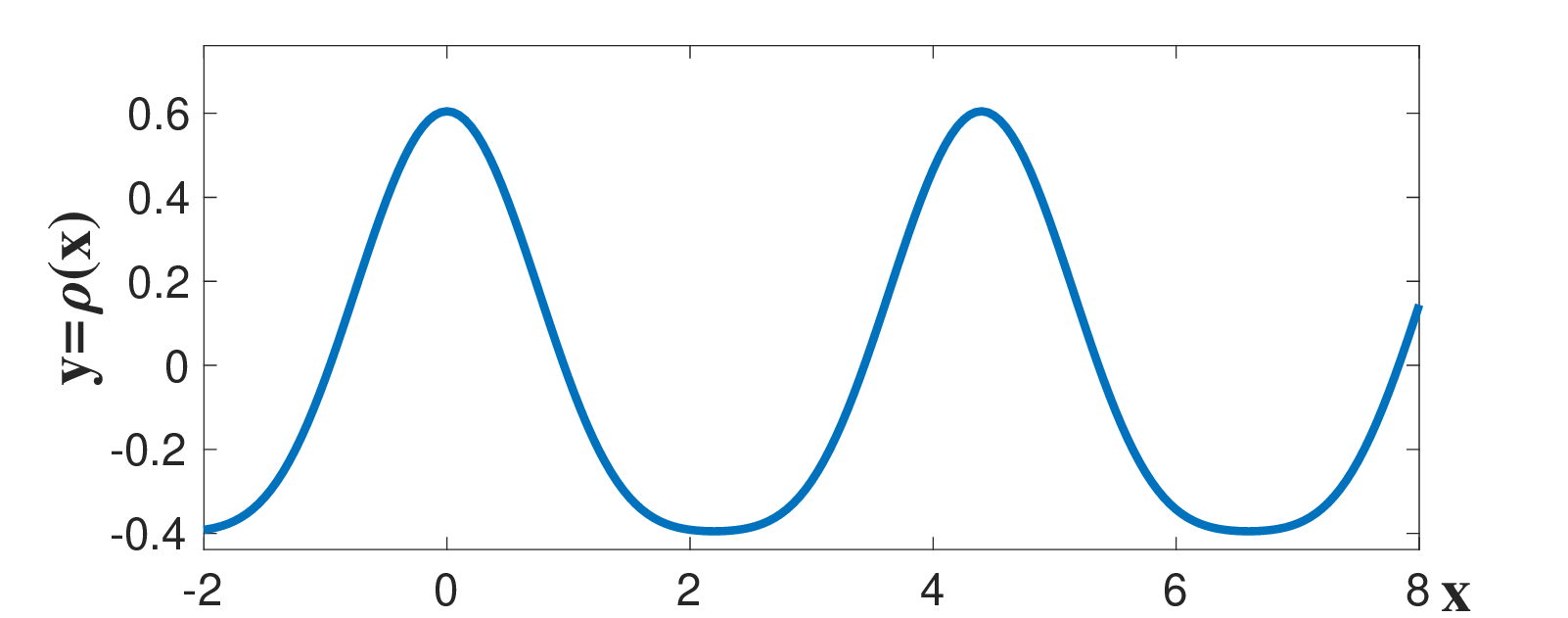}
		\caption{ Approximate shape of the traveling wave. 
%\textcolor{red}{Parameter values are $-\zeta=20\,{\rm kPa}$, $\zeta_i =0.1\, {\rm kPa}/{\rm \mu m}$, $\xi= 100 \, {\rm Pa\cdot s}/{\rm \mu m}^2$, $\gamma= 0.2 \, {\rm mN}/{\rm m}$, $\mu = 25\, {\rm MPa}\cdot {\rm s}$.}
			\label{fig:krasotaaaaa}
		}
	\end{center}
\end{figure} 

Figure \ref{fig:krasota} depicts graphs of the eigenvalue $\Lambda(q)$ (growth rate, computed by the formula \eqref{Lambda_formula})
for different values of the intercellular contractility $-\zeta$.  
\begin{figure}[H]
	\begin{center}
		\includegraphics[width=0.5\textwidth]{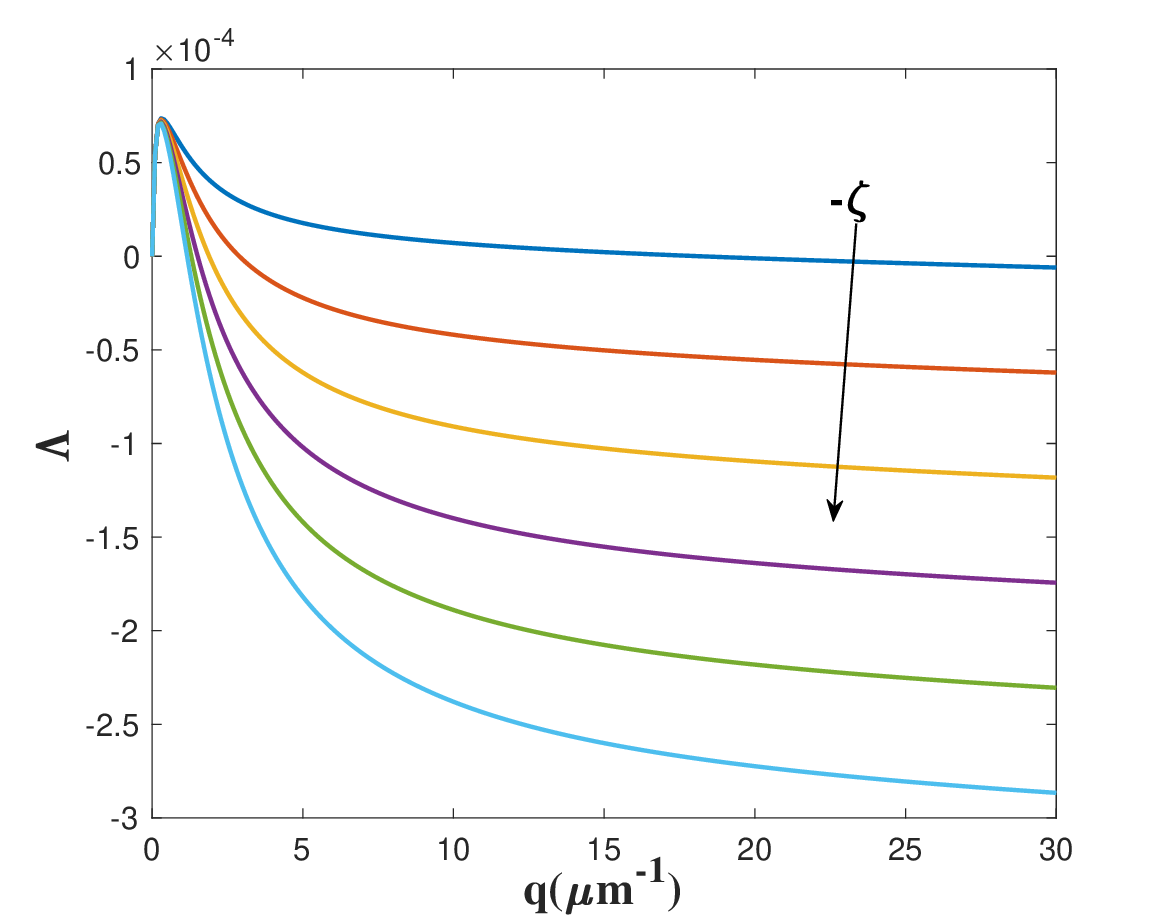}
		\caption{ Growth rate $\Lambda(q)$ as a function of  wave number $q$ for different contractilities $-\zeta$. For this plot $-\zeta=0$, $6$, $12$, $18$, $24$, $30$ (kPa). 
			Other parameters are $\zeta_i =0.1\, {\rm kPa}/{\rm \mu m}$, $\xi= 100 \, {\rm Pa\cdot s}/{\rm \mu m}^2$, $\gamma= 0.2 \, {\rm mN}/{\rm m}$, $\mu = 25\, {\rm MPa}\cdot {\rm s}$.
			\label{fig:krasota}
		}
	\end{center}
\end{figure}

Next, we study the dependence of the critical period $\Pi_0=\frac{2\pi}{q_0}$ on the contractility. It amounts to numerical solving of the equation  $\Lambda(q_0)=0$. 
The results are depicted in the Figure \ref{fig:krasota1}.
 Note that for contractility $-\zeta\sim 20$ kPa  the value of the period is close to one hundred micrometers which is in agreement with the measured finger spacing \cite{Vish2018}.
\begin{figure}[h!]
	\begin{center}
		\includegraphics[width=0.5\textwidth]{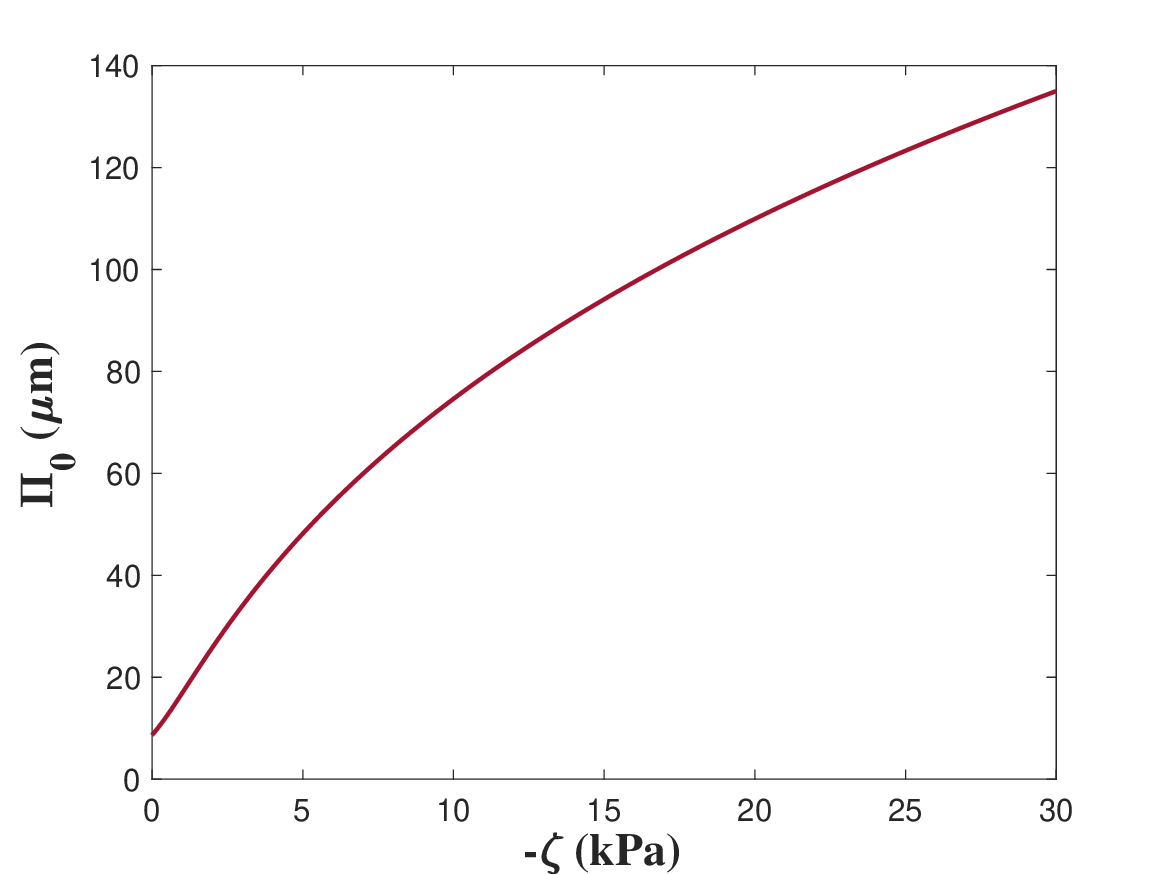}
		\caption{Dependence of the critical period on the contractility coefficient $-\zeta$. Computations carried out for  $\zeta_i =0.1\, {\rm kPa}/{\rm \mu m}$, $\xi= 100 \, {\rm Pa\cdot s}/{\rm \mu m}^2$, $\gamma= 0.2 \, {\rm mN}/{\rm m}$, $\mu = 25\, {\rm MPa}\cdot {\rm s}$
			\label{fig:krasota1}
		}
	\end{center}
\end{figure}

%Next Figures \ref{fig:krasota3} and \ref{fig:krasota2} describe the behavior of the period of traveling waves in the neighborhood of the bifurcation point. Namely, they represent graphs of the coefficient $b$ in the asymptotic expansion \eqref{expan_theta} of $\theta(\alpha)=\Pi/\Pi_0$. The sign of $b$ determines the type of the bifurcation, for
%$b>0$ we have the supercritical bufurcation, and $b<0$  coresponds the subcritical one. Notice that $b<0$ for large values of the contractility as seen from Figure \ref{fig:krasota3}.
%
Finally Figures \ref{fig:krasota3} and \ref{fig:krasota2} present results of computations of the coefficient $b$. Recall that $b$ is the coefficient in the asymptotic expansion \eqref{expan_theta} of $\theta(\alpha)=\Pi/\Pi_0$. 
It follows from  \eqref{asymptotic_lambda} that the sign of the smallest in absolute value eigenvalue of the operator linearized around the traveling wave solution
coincides  (for sufficiently small $\alpha$)  with the sign of the product $b\partial_q\Lambda(q_0)$, while other nonzero eigenvalues have a negative real part. Since $\partial_q\Lambda(q_0)<0$ (see Figure \ref{fig:krasota}), $b>0$ correspond 
to stable case, while $b<0$ correspond to unstable case. In other words, for $b>0$ we have a supercritical bifurcation, while for $b<0$ we have a  subcritical one. 
Notice that $b<0$ for large values of the contractility as seen from Figure \ref{fig:krasota3}. 
%{asymptotic_lambda}
%\textcolor{red}{ we only identify the parameter $b$ appearing in the expansion  \eqref{expan_theta} of $\theta$, 
%whose sign determines the bifurcation type, $b>0$ ($b<0$) corresponds to the supercritical (subcritical) bufurcation (see \eqref{asymptotic_lambda} and Figure \ref{fig:krasota}).}
\begin{figure}[h!]
	\begin{center}
		\includegraphics[width=0.5\textwidth]{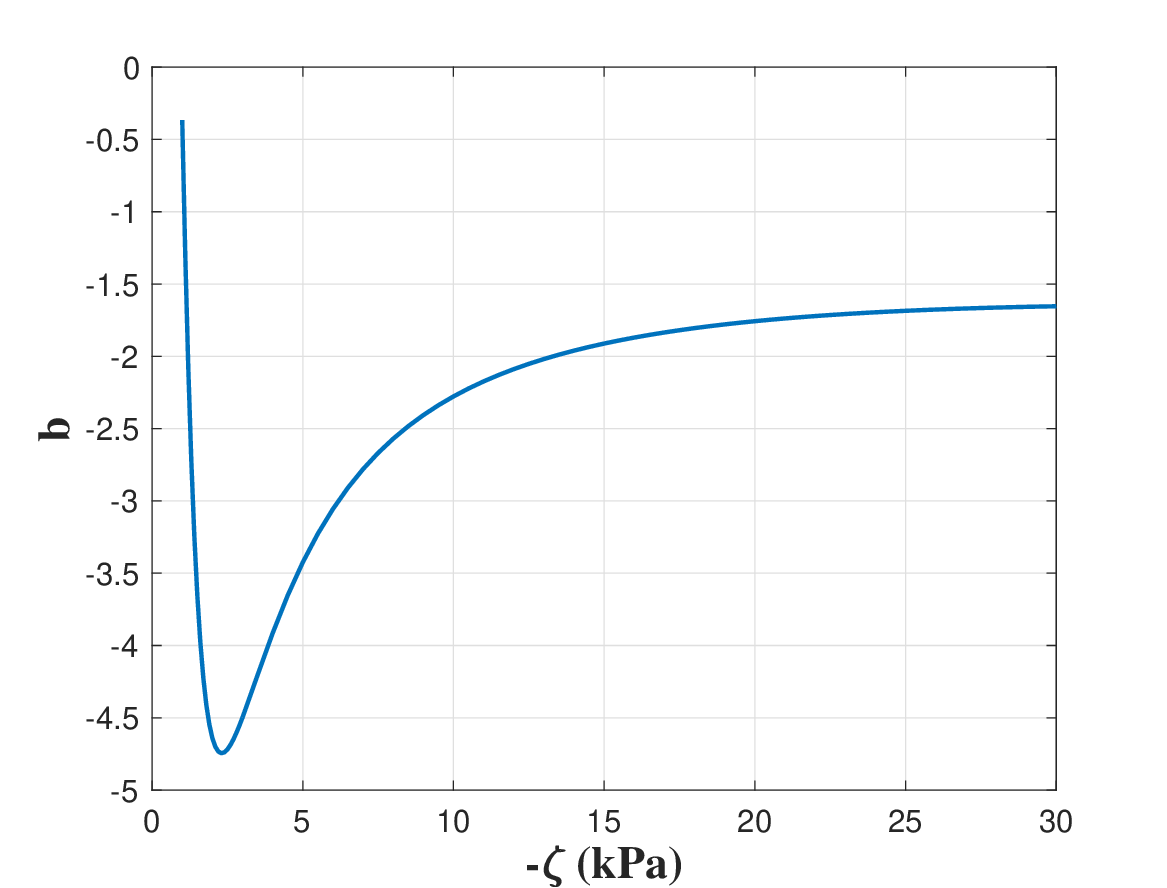}
		\caption{Coefficient $b$ for contractility  $-\zeta$ from $1$ kPa to $30$ kPa. Other parameters 
			are $\zeta_i =0.1\, {\rm kPa}/{\rm \mu m}$, $\xi= 100 \, {\rm Pa\cdot s}/{\rm \mu m}^2$, $\gamma= 0.2 \, {\rm mN}/{\rm m}$, $\mu = 25\, {\rm MPa}\cdot {\rm s}$.
			\label{fig:krasota3}
		}
	\end{center}
\end{figure}
\begin{figure}[h!]
	\begin{center}
		\includegraphics[width=0.5\textwidth]{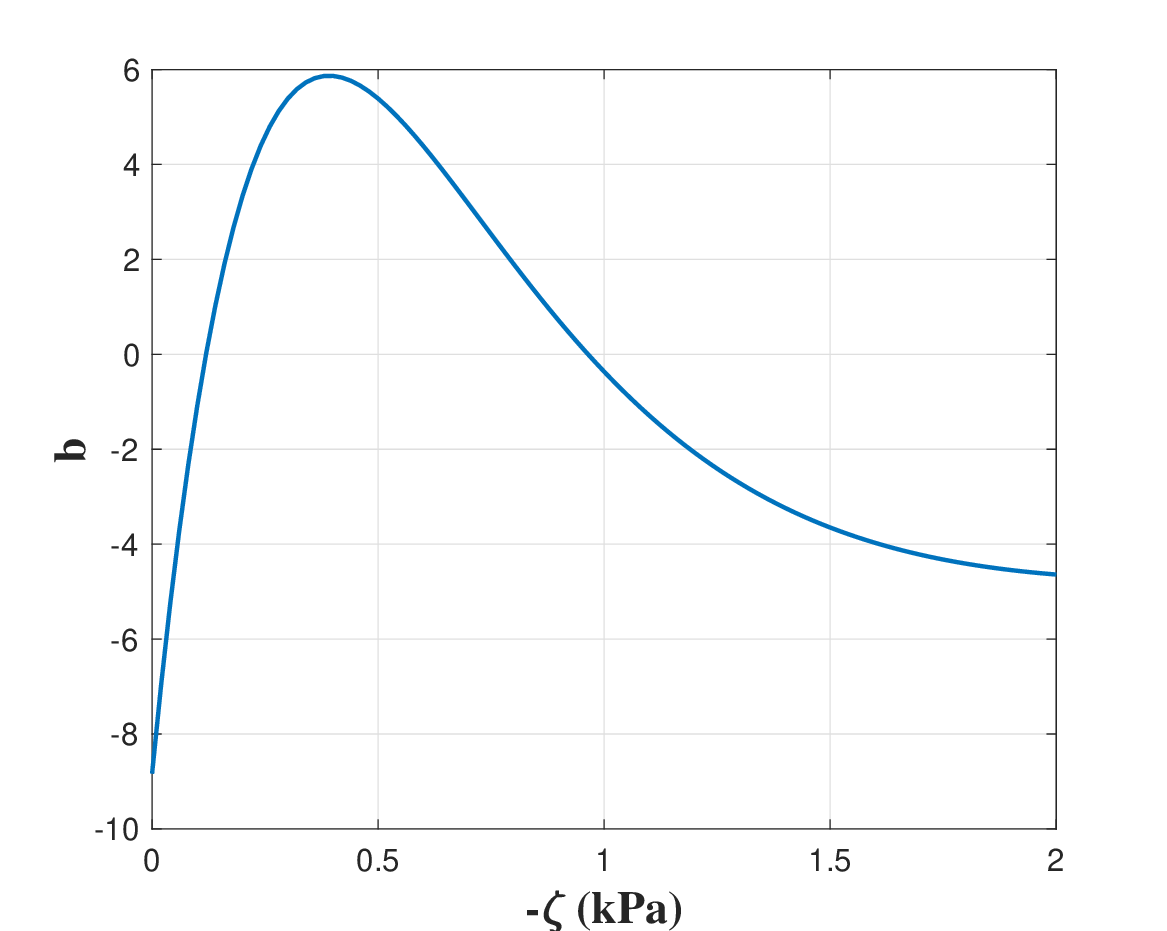}
		\caption{Coefficient $b$ for contractility $-\zeta$ from $0$ to $2$ kPa. Other parameters are $\zeta_i =0.1\, {\rm kPa}/{\rm \mu m}$, $\xi= 100 \, {\rm Pa\cdot s}/{\rm \mu m}^2$, $\gamma= 0.2 \, {\rm mN}/{\rm m}$, $\mu = 25\, {\rm MPa}\cdot {\rm s}$.
			\label{fig:krasota2}
		}
	\end{center}
\end{figure}

\noindent It is interesting to observe in Figure \ref{fig:krasota2} that for smaller $-\zeta$ both cases $b>0$ and $b<0$ occur. Thus the model exhibits both subcritical and supercritical bifurcation.

%\section{Discussion}\label{sec12}
%
%Discussions should be brief and focused. In some disciplines use of Discussion or `Conclusion' is interchangeable. It is not mandatory to use both. Some journals prefer a section `Results and Discussion' followed by a section `Conclusion'. Please refer to Journal-level guidance for any specific requirements. 
%
%\section{Conclusion}\label{sec13}
%
%Conclusions may be used to restate your hypothesis or research question, restate your major findings, explain the relevance and the added value of your work, highlight any limitations of your study, describe future directions for research and recommendations. 
%
%In some disciplines use of Discussion or 'Conclusion' is interchangeable. It is not mandatory to use both. Please refer to Journal-level guidance for any specific requirements. 

%\backmatter

\section{Acknowledgments}
The work of L. Berlyand and C. Safsten was supported by the NSF grant no. DMS-2005262. The work of V. Rybalko 
was partially supported by the SSF grant for Ukrainian scientists Dnr UKR22-0004 and NSF grant no. DMS-2005262.
We are grateful to J. Casademunt, R. Alert, and C. Blanch-Mercader for the very useful discussion on the physical aspects of the problem.

\begin{appendices}

\section{Fourier analysis of the linearized operator}\label{mister_first}

%\subsection{Fourier analysis of  the linearized operator}
%\label{mister_first} 
The  solution of \eqref{lin pr1}-\eqref{lin bc2}  can be represented as ${\bf v}= \zeta_i {\bf v}^t +\zeta{\bf v}^c+ \gamma{\bf v}^s$ %by \eqref{tryKrysyvyhVektora} 
with $ {\bf v}^t$,  ${\bf v}^c$ and ${\bf v}^s$
solving
\begin{equation}
	\label{split 1}	\begin{cases}
		\mu (\Delta  v^t_x+\partial_x {\rm div }  {\bf v}^t)-\xi v^t_x=iq  e^{iqx +\sqrt{q^2+1}y} & \text{for}\ y<0,\\ 
		\mu (\Delta  v^t_y+\partial_y {\rm div } {\bf v}^t)-\xi  v^t_y= e^{iqx +\sqrt{q^2+1}y}  & \text{for}\ y<0,\\
		\mu (\partial_x v^t_y+\partial_y v^t_x)=0  & \text{for}\ y=0,\\
		2\mu \partial_y v^t_y=\frac{\sqrt{2\mu}}{\sqrt{\xi}+\sqrt{2\mu}}e^{iqx} \quad & \text{for} \ y=0,
	\end{cases}
\end{equation}
\begin{equation}
	\label{split 2}
	\begin{cases}
		\mu (\Delta  v^c_x+\partial_x {\rm div }  {\bf v}^c)-\xi v^c_x=-iq(\sqrt{q^2+1}+1) e^{iqx +y+y\sqrt{q^2+1}} & \text{for}\ y<0,\\ 
		\mu (\Delta  v^c_y+\partial_y {\rm div } {\bf v}^c)-\xi  v^c_y &\\ 
		\hskip 0.5 cm =q^2 e^{iqx +y+y\sqrt{q^2+1}} -2(\sqrt{q^2+1}+1)
		e^{iqx +y+y\sqrt{q^2+1}} & \text{for}\ y<0, \\
		\mu (\partial_x v^c_y+\partial_y v^c_x)=-iq e^{iqx} & \text{for}\ y=0,\\
		2\mu \partial_y v^c_y=-\bigg(\frac{\xi}{4\mu+\sqrt{2\mu\xi}}+2\bigg) e^{iqx} & \text{for}\ y=0, %=-\zeta\frac{\xi+2\sqrt{2\mu\xi}+8\mu}{4\mu+\sqrt{2\mu\xi}} e^{iqx}
	\end{cases}
\end{equation}
\begin{equation}
	\label{split 3}
	\begin{cases}
		\mu (\Delta  {\bf v}^s+\nabla {\rm div }  {\bf v}^s)-\xi{\bf v}^s=0& \text{for}\ y<0, \\ 
		\mu (\partial_x v^s_y+\partial_y v^s_x)=0 & \text{for}\ y=0,\\
		2\mu \partial_y v^s_y=-q^2 e^{iqx} \quad& \text{for}\ y=0.
	\end{cases}
\end{equation}
We find explicit solutions to these problems,
starting with problem \eqref{split 1}.  Represent the equations in \eqref{split 1} as 
\begin{align}
	%	\label{split 1}
	\mu (&\Delta {\bf v}^t+\nabla {\rm div }  {\bf v}^t)-\xi {\bf v}^t\notag \\ 
	&= \Bigl(\sqrt{q^2+1}-q^2\Bigr)  \nabla e^{iqx +\sqrt{q^2+1}y} +
	iq\Bigl(1-\sqrt{q^2+1}\Bigr)  \nabla^\perp e^{iqx +\sqrt{q^2+1}y},		
\end{align}
where $\nabla^\perp= (-\partial_y, \partial_x)$. Then
\begin{align}
	\label{vt}
	\notag	{\bf v}^{t}&= {\textstyle \frac{-q^2 +\sqrt{q^2 +1}}{2\mu-\xi}}\nabla  e^{iqx +\sqrt{q^2+1}y}
	+iq {\textstyle\frac{1-\sqrt{q^2 +1}}{\mu-\xi}}\nabla^\perp e^{iqx +\sqrt{q^2+1}y}\\
	&+A_1\tilde {\bf v}^{(1)}+A_2\tilde {\bf v}^{(2)},
\end{align}
where 
$$
\tilde{\bf v}^{(1)}=\nabla^\perp e^{iqx +y\sqrt{q^2+\xi/\mu}}, \quad \tilde {\bf v}^{(2)}=\nabla e^{iqx +y\sqrt{q^2+\xi/(2\mu)}} 
$$ 
are linearly independent solutions of the corresponding homogenous equations, vanishing as 
$y\to -\infty$. Substituting \eqref{vt} into boundary conditions of \eqref{split 1} we get a linear system for constants $A_1$, $A_2$, resolving which we obtain
\begin{align}
	{\bf v}^{t}&= {\textstyle\frac{\sqrt{q^2+1}-q^2}{2\mu-\xi}}
	\Bigl(\nabla e^{iqx +y\sqrt{q^2+1}} + A_{111} \tilde{\bf v}^{(1)} + A_{112} \tilde{\bf v}^{(2)} \Bigr) \nonumber \\
	&+iq {\textstyle\frac{\sqrt{q^2+1}-1}{\mu-\xi}}
	\Bigl( \nabla^\perp  e^{iqx +y\sqrt{q^2+1}} 
	+ A_{121}\tilde{\bf v}^{(1)}  + A_{122} \tilde{\bf v}^{(2)} \Bigr)\nonumber \\
	&+ A_{131}\tilde{\bf v}^{(1)}  + A_{132} \tilde{\bf v}^{(2)} \label{v_t}
\end{align}
with
\begin{equation}
	\notag
	A_{111} =-{\textstyle\frac{2\mu iq \sqrt{q^2+1}}{D(q,\mu,\xi)}}\Bigl(2\mu q^2+\xi- 2\mu \sqrt{q^2+\xi/(2\mu)}\sqrt{q^2+1}\Bigr),
\end{equation}
\begin{equation}
	\notag
	A_{112} ={\textstyle\frac{2\mu \sqrt{q^2+1}}{D(q,\mu,\xi)}}\Bigl( (2\mu q^2+\xi) \sqrt{q^2+1}-2\mu q^2 \sqrt{q^2+\xi/\mu}  \Bigr),
\end{equation}
\begin{equation}
	\notag
	A_{121}={\textstyle\frac{2\mu^2}{D(q,\mu,\xi)} \sqrt{q^2+\xi/(2\mu)} }\Bigl( (2q^2+1)\sqrt{q^2+\xi/(2\mu)}-2q^2\sqrt{q^2+1}\Bigr),
\end{equation}
\begin{equation}
	\notag
	A_{122}=-{\textstyle\frac{iq\mu}{D(q,\mu,\xi)}}\Bigl( 2\mu(2q^2+1) \sqrt{q^2+\xi/\mu} - 2(2\mu q^2+\xi) \sqrt{q^2+1}\Bigr),
\end{equation}
\begin{equation}
	\notag
	A_{131}={\textstyle\frac{-2\mu i q \sqrt{q^2+\xi/(2\mu)}}{D(q,\mu,\xi)}\frac{\sqrt{2\mu}}{\sqrt{\xi}+\sqrt{2\mu}}},\quad
	A_{132}={\textstyle\frac{-(2\mu q^2 +\xi)}{D(q,\mu,\xi)}\frac{\sqrt{2\mu}}{\sqrt{\xi}+\sqrt{2\mu}}}
\end{equation}
and $D(q,\mu,\xi)$ given by \eqref{det_D}.
Analogously we find solutions ${\bf v}^{c}$, ${\bf v}^{s}$ to problems \eqref{split 2} and
\eqref{split 3}. We have
\begin{align}
	{\bf v}^{c}&=
	{\textstyle \frac{q^2-\sqrt{q^2+1}-1}{4\mu(\sqrt{q^2+1}+1)-\xi}}
	\Bigl( \nabla e^{iqx+y+y\sqrt{q^2+1}}  + A_{211} \tilde{\bf v}^{(1)}  
	+A_{212} \tilde{\bf v}^{(2)} \Bigr) \nonumber \\
	&+ {\textstyle\frac{q^3}{(\xi -2\mu(\sqrt{q^2+1}+1))(\sqrt{q^2+1}+1)}}
	\Bigl( \nabla^\perp e^{iqx+y+y\sqrt{q^2+1}}  + A_{221} \tilde{\bf v}^{(1)} 
	+A_{222} \tilde{\bf v}^{(2)}\Bigr)\notag\\
	&+ A_{231} \tilde{\bf v}^{(1)} 
	+A_{232} \tilde{\bf v}^{(2)}, \label{v_c}
\end{align}
where
\begin{equation}
	\notag
	A_{211}=-{\textstyle\frac{2\mu iq (\sqrt{q^2+1}+1)}{D(q,\mu,\xi)}}\Bigl(2\mu q^2+\xi- 2\mu \sqrt{q^2+\xi/(2\mu)}\bigl(\sqrt{q^2+1}+1\bigr)\Bigr),
\end{equation}
\begin{equation}
	\notag
	A_{212}={\textstyle\frac{2\mu( \sqrt{q^2+1}+1)}{D(q,\mu,\xi)}}\Bigl((2\mu q^2+\xi) \bigl(\sqrt{q^2+1}+1\bigr) -2\mu q^2 \sqrt{q^2+\xi/\mu} \Bigr),
\end{equation}
\begin{equation}
	\notag
	A_{221}={\textstyle\frac{4\mu^2\sqrt{q^2+\xi/(2\mu)}}{D(q,\mu,\xi)}}  \Bigl( \sqrt{q^2+\xi/(2\mu)}\bigl(q^2+\sqrt{q^2+1}+1\bigr)
	-q^2\bigl(\sqrt{q^2+1}+1\bigr)\Bigr),
\end{equation}
\begin{equation}
	\notag
	A_{222}=-{\textstyle\frac{2\mu iq}{D(q,\mu,\xi)}}\Bigl(2 \mu \sqrt{q^2+\xi/\mu}\bigl(q^2+\sqrt{q^2+1}+1\bigr) - 	(2\mu q^2+\xi) \bigl(\sqrt{q^2+1}+1 \bigr) \Bigr),
\end{equation}
\begin{equation}
	\notag
	A_{231}={\textstyle\frac{2\mu iq}{D(q,\mu,\xi)}}\sqrt{q^2+\xi/(2\mu)}\Bigl({\textstyle\frac{\xi}{4\mu+\sqrt{2\mu \xi}}} - \sqrt{q^2+\xi/(2\mu)} +2 \Bigr),
\end{equation}
\begin{equation}
\notag
A_{232}={\textstyle\frac{1}{D(q,\mu,\xi)}}
\Bigl( 
(2\mu q^2+\xi)\Bigl( {\textstyle\frac{\xi}{4\mu+\sqrt{2\mu \xi}}} +2\Bigr) - 2\mu q^2 \sqrt{q^2+\xi/\mu} \Bigr),
\end{equation}
and 
\begin{align}
\label{v_s}
{\bf v}^{s}={\textstyle\frac{2\mu i q^3 \sqrt{q^2+\xi/(2\mu)}}{D(q,\mu,\xi)}} 
\tilde{\bf v}^{(1)} 
+{\textstyle\frac{q^2 (2\mu q^2+\xi)}{D(q,\mu,\xi)}}  \tilde{\bf v}^{(2)}.
\end{align}

%Here we also establish some relations for  ${\bf v}^{(1)}$ that will be needed for finding higher order terms in the expansion \eqref{expan_v}.
Notice that ${\bf} v^{(1)}$ appearing in the first order term of the expansion \eqref{expan_v} can be obtain by taking the real part of ${\bf v}= \zeta_i {\bf v}^t +\zeta{\bf v}^c+ \gamma{\bf v}^s$. This yields 
$v_y^{(1)}\left.\right|_{y=0}=\Bigl(\zeta_i\Lambda^t(q_0,\mu,\xi)+\zeta\Lambda^c(q_0,\mu,\xi)+\gamma\Lambda^s(q_0,\mu,\xi)\Bigr)\cos q_0 x$ (with $\Lambda^{t,c,s}(q_0,\mu,\xi)$ given by \eqref{Lambda_t}--\eqref{Lambda_s}). For $q_0$ satisfying \eqref{eq_for_q0} this formula is simplified to
\begin{align}
v_y^{(1)}\left.\right|_{y=0}=-\frac{\zeta}{2\mu}\cos q_0 x. 
\label{nyashechka} 
\end{align}
To find higher order terms in the expansions \eqref{expan_theta}-\eqref{expan_v} we will also need to compute $v_x^{(1)}\left.\right|_{y=0}$. Although  %we can write $v_x^{(1)}$ from the 
an explicit formula for $v_x^{(1)}$ is available via \eqref{v_t}--\eqref{v_s}, we can derive a more compact expression in the case $q=q_0$. Recall that ${\bf v}^{(1)}$ solves equations
%From  \eqref{lin pr1}, \eqref{lin pr2} we obtain the following equations for components of ${\bf v}^{(1)}$ for $y<0$
\begin{align}
\label{pervoedlyapervogo}
\mu \bigl(\Delta  v_x^{(1)}+\partial_x {\rm div }  {\bf v}^{(1)}\bigr)&-\xi v_x^{(1)}
=-q_0 \zeta_i e^{y\sqrt{q_0^2+1}}\sin q_0x \nonumber\\ 
&+q_0\zeta\Bigl(\textstyle{\sqrt{q_0^2+1}}+1\Bigr) e^{y+y\sqrt{q_0^2+1}}\sin q_0x, 
\end{align}
\begin{align}
\label{vtoroedlyapervogo}
\mu \bigl(\Delta  v_y^{(1)}+\partial_y {\rm div } {\bf v}^{(1)}\bigr)&-\xi  v_y^{(1)}=\zeta_i  e^{y\sqrt{q_0^2+1}} \cos q_0x+q_0^2\zeta e^{y+y\sqrt{q_0^2+1}}\cos q_0x \nonumber\\
& -2\zeta\Bigl(\textstyle{\sqrt{q_0^2+1}}+1\Bigr)
e^{y+y\sqrt{q_0^2+1}}\cos q_0x  
\end{align}
with boundary conditions for $y=0$
\begin{align}
&\mu (\partial_x v_y^{(1)}+\partial_y v_x^{(1)})=\zeta q_0 \sin q_0x \label{bc1} \\
&
2\mu \partial_y v_y^{(1)}
=\Bigl(\textstyle{-\zeta\Bigl(\frac{\xi}{4\mu+\sqrt{2\mu\xi}}+2\Bigr)  +\zeta_i\frac{\sqrt{2\mu}}{\sqrt{\xi}+\sqrt{2\mu}}  -\gamma q^2}  \Bigr)
\cos q_0x . 	\label{bc2}
\end{align}
%since v_y^{(1)}=\hat v(y) cosq_0x, \partial^2_{xx}v_y^{(1)}=-q_0^2 v_y^{(1)}
It follows from \eqref{vtoroedlyapervogo} that for $y=0$ 
\begin{align}
\label{d_yy v_y^1}
2\mu \partial^2_{yy}v_y^{(1)}=\bigl(\mu q_0^2+\xi\bigr)v_y^{(1)}-\mu \partial^2_{xy}v_x^{(1)}+\Bigl(\zeta_i+
q_0^2\zeta-2\zeta\bigl(\textstyle{\sqrt{q_0^2+1}}+1\bigr)\Bigr)\cos q_0x.
\end{align}
On the other hand differentiating %the third line of 
\eqref{bc1}  in $x$ we get %also uaing \partial^2_{xx}v_y^{(1)}=-q_0^2 v_y^{(1)}
%\begin{align}
$	\mu \partial^2_{xy}v_x^{(1)}=\mu q_0^2v_y^{(1)} +q_0^2\zeta\cos q_0x$ for $y=0$,
%\end{align}
and substituting \eqref{nyashechka} 
%from $\mathcal{L}\cos q_0x=0$ we find $v_y^{(1)}\left.\right|_{y=0}=-\frac{\zeta}{2\mu}\cos q_0x$, so
we obtain
\begin{align}
\label{39}
\mu \partial^2_{xy}v_x^{(1)}\Bigr|_{y=0}=\frac{q_0^2\zeta}{2}\cos q_0x.
\end{align}
Therefore \eqref{d_yy v_y^1} yields 
\begin{align}
\label{40}
2\mu \partial^2_{yy}v_y^{(1)}\Bigr|_{y=0}=\Bigl(\zeta_i-2\zeta\Bigl(\textstyle{\sqrt{q_0^2+1}}+1+\textstyle{\frac{\xi}{4\mu}}\Bigr)\Bigr)\cos q_0x.
\end{align}

Now we find ${\rm div}{\bf v}^{(1)}$. From \eqref{pervoedlyapervogo}--\eqref{vtoroedlyapervogo} 
\begin{align}
\label{eq div v}
2\mu \Delta{\rm div}{\bf v}^{(1)}&-\xi{\rm div}{\bf v}^{(1)} 
=\zeta_i\Bigl(\textstyle{\sqrt{q_0^2+1}} -q_0^2\Bigr) \cos q_0x \, e^{y\sqrt{q_0^2+1}}
\nonumber\\ 
&+2\zeta\Bigl(q_0^2\textstyle{\sqrt{q_0^2+1}}-2\textstyle{\sqrt{q_0^2+1}}-2 \Bigr)\cos q_0x \, e^{y+y\sqrt{q_0^2+1}}, 
\end{align}
also by  \eqref{39}--\eqref{40} we have for $y=0$ %it satisfies the equation 
%diff  {pr v1x}_x+{pr v1y}_y=...
\begin{align}
\label{bc_div_v}
\mu \partial_y {\rm div} {\bf v}^{(1)}=\Bigl(\textstyle{\frac{\zeta_i}{2}}-\zeta\Bigl(\sqrt{q_0^2+1}+1+
\textstyle{\frac{\xi}{4\mu}- \frac{q_0^2}{2}}\Bigr)\Bigr)\cos q_0x.
\end{align}
One can find an explicit solution to the problem \eqref{eq div v}--\eqref{bc_div_v}, in particular

\begin{equation}
\label{Div_v1}
{\rm div}{\bf v}^{(1)}\bigl.\bigr|_{y=0}=D_1\cos q_0 x,
\end{equation}
where
\begin{equation}
\begin{aligned}
	\label{I_znovu_zdraste_D1}
	D_1&=
	\textstyle{	\frac{\zeta_i\Bigl(\sqrt{q_0^2+{\textstyle \frac{\xi}{2\mu}}}+ q_0^2\Bigr)}
		{2\mu\sqrt{q_0^2+{\textstyle \frac{\xi}{2\mu}}}\Bigl(\sqrt{q_0^2+{\textstyle \frac{\xi}{2\mu}}}+\sqrt{q_0^2+1}\Bigr)} }
	\\
	&-\textstyle{\frac{\zeta}{ {\mu\sqrt{q_0^2+{\textstyle \frac{\xi}{2\mu}}}}}} 
	\Biggl(
	\sqrt{q_0^2+1}+1+\frac{\xi}{4\mu}- \frac{q_0^2}{2}+ \textstyle{\frac{q_0^2\sqrt{q_0^2+1}-2\sqrt{q_0^2+1}-2}{\sqrt{q_0^2+1}+1+\sqrt{q_0^2+{\textstyle \frac{\xi}{2\mu}}}}}
	\Biggr).
\end{aligned}
\end{equation}
Then taking into account \eqref{bc2} 
\begin{align}
\label{pervayatrudnyashtuka}
\partial_x v_x^{(1)}=\Bigl({\textstyle D_1+ \frac{\zeta\xi}{2\mu\left(4\mu+\sqrt{2\mu\xi}\right)}-\frac{\zeta_i}{\sqrt{2\mu\xi}+2\mu} +\frac{2\zeta +\gamma q_0^2}{2\mu}}\Bigr)\cos q_0 x,
\end{align}
and 
\begin{align}
\label{pervayatrudnyashtuka_v_x}
v_x^{(1)}={\textstyle \frac{1}{q_0}}\Bigl({\textstyle D_1+ \frac{\zeta\xi}{2\mu\left(4\mu+\sqrt{2\mu\xi}\right)}-\frac{\zeta_i}{\sqrt{2\mu\xi}+2\mu} +\frac{2\zeta +\gamma q_0^2}{2\mu}}\Bigr)\sin q_0 x.
\end{align}

%\textcolor{red}{$D<0, \ D\sim -\mu\xi q^2\quad\text{as} \  |q|\to \infty, \quad \bigl. D \bigr|_{q=0}=-\xi^2$,}

\section{ Calculations of ${\bf v}^{(2)}$ and $\rho^{(2)}$}\label{app_v_21}  
%Notice that $\mu\bigl(\partial_x v_y^{(1)}+\partial_y v_x^{(1)}\bigr)=\zeta q_0\sin q_0x $ for $y=0$
From \eqref{v^2_eq}   and taking into account explicit formulas \eqref{p_21x}--\eqref{p_21y} for ${\bf p}^{(21)}$ we have that ${\bf v}^{(21)}$ satisfies the equations 
\begin{align}
&\mu (\Delta v_x^{(21)}+\partial_x {\rm div}{\bf v}^{(21)})-\xi v_x^{(21)} \label{pr_v21_1}\\
&\quad\quad\quad ={\textstyle-\zeta_i p_x^{(21)} -\zeta  \frac{q_0\sqrt{q_0^2+1}(1+\sqrt{4q_0^2+1})}{2}}  e^{y+y\sqrt{4q_0^2+1}} \sin2q_0x \nonumber\\  
&\quad\quad \quad+\zeta q_0\bigl(q_0^2-\textstyle{\sqrt{q_0^2+1}}\bigr) e^{2y\sqrt{q_0^2+1}} \sin2q_0x, 	\nonumber  \\ 
&\mu (\Delta v_y^{(21)}+\partial_y {\rm div}{\bf v}^{(21)})-\xi v_y^{(21)} \label{pr_v21_2}\\ 
&\quad =2b\xi V_y-b\zeta_i P_y -\zeta_i p_y^{(21)}+\zeta \bigl(\textstyle{\sqrt{q_0^2+1}}-q_0^2\bigr) e^{2y\sqrt{q_0^2+1}} \cos2q_0x \nonumber\\  
&\quad+\zeta 
\biggl( {\textstyle\frac{(q_0^2+2\sqrt{q_0^2+1}-1)(1+\sqrt{4q_0^2+1})}{2}}-q_0^2\textstyle{\sqrt{q_0^2+1}}\biggr)
e^{y+y\sqrt{4q_0^2+1}} \cos2q_0x \nonumber\\  
&\quad +\zeta \textstyle{\sqrt{q_0^2+1}}e^{2y\sqrt{q_0^2+1}} +\zeta \bigl(2b+4by+2\textstyle{\sqrt{q_0^2+1}}-q_0^2-1\bigr) e^{2y}.	 \nonumber
\end{align}
Moreover, from \eqref{v^2_boundarycondition_first}--\eqref{v^2_boundarycondition_second} using 
\eqref{UrForVy}, \eqref{pervoedlyapervogo}, \eqref{bc1}, \eqref{40}, \eqref{pervayatrudnyashtuka} we find the boundary conditions for ${\bf v}^{(21)}$ as $y=0$:
$$
\mu (\partial_x v_y^{(21)}+\partial_y v_x^{(21)})=T_x\sin 2q_0x 
$$
\begin{equation}
2\mu\partial_y v_y^{(21)}
%&=-\mu q_0 \sin q_0x \bigl(\partial_x v_y^{(1)}+\partial_y v_x^{(1)}\bigr) -2\mu \cos q_0x \,\partial^2_{yy} v_y^{(1)} \nonumber \\ &- \mu\cos^2 q_0x\,\partial^3_{yyy} V_y(0)\nonumber\\ &=- \zeta q_0^2\sin^2q_0x -\left(\zeta_i-2\zeta\left(\sqrt{q_0^2+1}+1+\frac{\xi}{4\mu}\right)\right)\cos^2 q_0x\nonumber \\ &- \left(\frac{\zeta\xi}{4\mu}+2\zeta-\frac{\zeta_i}{2}\right)\cos^2 q_0x\nonumber\\
= T_y^{(1)}+ T_y^{(2)}\cos2q_0x, \notag
\end{equation} 
where
\begin{flalign*}
&T_x=\textstyle{\frac{q_0}{2}}\Bigl(\zeta_i -\zeta \sqrt{q_0^2+1}-\zeta\Bigr)
-\textstyle{\frac{\gamma q_0^3}{2}}\\
&\quad\quad
-\textstyle{\frac{1}{2q_0}}\left(4\mu q_0^2+\xi\right)
\Bigl(D_1+ \textstyle{\frac{\zeta\xi}{2\mu\left(4\mu+\sqrt{2\mu\xi}\right)}}-\textstyle{\frac{\zeta_i}{\sqrt{2\mu\xi}+2\mu}} +\textstyle{\frac{2\zeta +\gamma q_0^2}{2\mu}}\Bigr),\notag\\
%\label{Nu_naschitalsya}
&T_y^{(1)}=\zeta\Bigl(\textstyle{\sqrt{q_0^2+1}}-\textstyle{\frac{q_0^2}{2}}+\textstyle{\frac{\xi}{8\mu}}\Bigr)-\textstyle{\frac{1}{4}}\zeta_i, \quad
T_y^{(2)}=\zeta\left(\sqrt{q_0^2+1}+\textstyle{\frac{q_0^2}{2}}+\textstyle{\frac{\xi}{8\mu}} \right)-\textstyle{\frac{1}{4}}\zeta_i.\nonumber
\end{flalign*}
We represent ${\bf v}^{(21)}$ in the following way
\begin{align}
\label{explicit_v21x}
v_x^{(21)}&=\left(B^{(1)}_xe^{y\sqrt{4q_0^2+1}}+B_x^{(2)}e^{y+y\sqrt{4q_0^2+1}}
+B_x^{(3)}e^{2y\sqrt{q_0^2+1}}\right)\sin 2q_0x \nonumber\\
&-\left(2Eq_0 e^{y\sqrt{4q_0^2+\xi/(2\mu)}} +F\textstyle{\sqrt{4q_0^2+\xi/\mu}}e^{y\sqrt{4q_0^2+\xi/\mu}}\right)\sin 2q_0x, 
%+B_x^{(4)}e^{y}+B_x^{(5)}ye^{y}+B_x^{(6)}e^{2y}+B_x^{(7)}ye^{2y}+B_x^{(8)}e^{2y\sqrt{q_0^2+1}}+ B_x^{(9)}e^{y\sqrt{\xi}/\sqrt{2\mu}}
\end{align}
\begin{align}
\label{explicit_v21y}
v_y^{(21)}&=\left(B^{(1)}_ye^{y\sqrt{4q_0^2+1}}+B_y^{(2)}e^{y+y\sqrt{4q_0^2+1}}
+B_y^{(3)}e^{2y\sqrt{q_0^2+1}}\right)\cos 2q_0x \nonumber\\
&+\left(E\textstyle{\sqrt{4q_0^2+\xi/(2\mu)}} e^{y\sqrt{4q_0^2+\xi/(2\mu)}} +2q_0Fe^{y\sqrt{4q_0^2+\xi/\mu}}\right)\cos 2q_0x \nonumber\\
&+B_y^{(4)}e^{y}+B_y^{(5)}ye^{y}+B_y^{(6)}e^{2y}+B_y^{(7)}ye^{2y}+B_y^{(8)}e^{2y\sqrt{q_0^2+1}}\nonumber\\ 
&+ B_y^{(9)}ye^{y\sqrt{\xi}/\sqrt{2\mu}}+G e^{y\sqrt{\xi}/\sqrt{2\mu}}
\end{align}
where
%\end{equation*}
\begin{flalign*}
B^{(1)}_x =\textstyle{\frac{\zeta_i}{4\mu^2-6\mu\xi+2\xi^2}}
\Bigl( \bigl(\mu(4q_0^2+2)-\xi\bigr)q_0\sqrt{q_0^2+1}-
\mu q_0\sqrt{4q_0^2+1}\bigl(q_0^2+2\sqrt{q_0^2+1}-1\bigr)\Bigr), %\nonumber \\
&&
\end{flalign*}
\begin{flalign*}
&B_x^{(2)}=-\frac{\zeta q_0\sqrt{q_0^2+1}\bigl(1+\sqrt{4q_0^2+1}\bigr)\bigl(4\mu(q_0^2+\sqrt{4q_0^2+1}+1)-\xi\bigr)}{64\mu^2q_0^2+4(8\mu^2-3\mu\xi)(1+\sqrt{4q_0^2+1})+2\xi^2} \nonumber\\
&\quad\quad\quad+\frac{\zeta q_0\mu (1+\sqrt{4q_0^2+1})
		\bigl( (q_0^2+2\sqrt{q_0^2+1}-1)(1+\sqrt{4q_0^2+1})-2q_0^2\sqrt{q_0^2+1}\bigr)}
	{32\mu^2q_0^2+2(8\mu^2-3\mu\xi)(1+\sqrt{4q_0^2+1})+\xi^2},&&&&&&&&&&&&&&&
\end{flalign*}
\begin{flalign*}
B_x^{(3)}=\textstyle{\frac{\zeta\bigl(q_0^2-\sqrt{q_0^2+1}\bigr)}{32\mu^2-12\mu\xi+\xi^2}}
\Bigl(4\mu q_0^3+(8\mu-\xi)q_0-4\mu q_0\sqrt{q_0^2+1}\Bigr),&&
\end{flalign*}
\begin{flalign*}
B^{(1)}_y=\textstyle{\frac{\zeta_i}{2\mu^2-3\mu\xi+\xi^2}}
\Bigl( -\mu q_0^2 \sqrt{4q_0^2+1}\sqrt{q_0^2+1}-\frac{1}{4}\bigl(\mu(1-4q_0^2)-\xi\bigr)
\bigl(q_0^2+2\sqrt{q_0^2+1}-1\bigr)\Bigr),&&
%	\nonumber \\
\end{flalign*}
\begin{flalign*}
&B_y^{(2)}= \frac{\zeta q_0^2\mu \sqrt{q_0^2+1}\bigl(4q_0^2+2\sqrt{4q_0^2+1}+2\bigr)}
	{32\mu^2q_0^2+2(8\mu^2-3\mu\xi)(1+\sqrt{4q_0^2+1})+\xi^2}
\nonumber \\
& +\frac{\zeta \bigl(\mu(-4q_0^2+2\sqrt{4q_0^2+1}+2)-\xi\bigr)
		\bigl( (q_0^2+2\sqrt{q_0^2+1}-1)(1+\sqrt{4q_0^2+1})-2q_0^2\sqrt{q_0^2+1}\bigr)}
	{64\mu^2q_0^2+4(8\mu^2-3\mu\xi)(1+\sqrt{4q_0^2+1})+2\xi^2},&&&&&&&
\end{flalign*}
\begin{flalign*}
B_y^{(3)}=\textstyle{\frac{\zeta\bigl(q_0^2-\sqrt{q_0^2+1}\bigr)}{32\mu^2-12\mu\xi+\xi^2}}
\Bigl(-4\mu q_0^2\sqrt{q_0^2+1}+4\mu q_0^2-4\mu+\xi\Bigr),&&
%	\nonumber \\
\end{flalign*}
\begin{flalign*}
B_y^{(4)}= -\textstyle{\frac{\zeta_i}{2\mu-\xi}}
\Bigl(\textstyle{\frac{2\sqrt{q_0^2+1}-q_0^2-1}{4}}-b\Bigr),
\quad
B_y^{(5)}=-\frac{\zeta_ib}{2\mu-\xi},&&
%\nonumber \\
\end{flalign*}
\begin{flalign*}
B_y^{(6)}=\textstyle{\frac{\zeta}{8\mu-\xi}}
\Bigl(2\sqrt{q_0^2+1}-q_0^2-1-2b\Bigr),
\quad B_y^{(7)}=\frac{4\zeta b}{8\mu-\xi},&&
%\nonumber \\
\end{flalign*}
\begin{flalign*}
B_y^{(8)}=\textstyle{\frac{\zeta \sqrt{q_0^2+1}}{8\mu(q_0^2+1)-\xi}},
\quad\quad
B_y^{(9)}=\frac{\zeta_i b}{2\mu-\xi}-\frac{\zeta b\xi }{2\mu (8\mu-\xi)}, &&
\end{flalign*}
\begin{align*}
E=& -\textstyle{\frac{1}{D(2q_0)}}\biggl(4\mu q_0 \sqrt{4q_0^2+\xi/\mu}
\left(-\mu \Bigl(B^{(1)}_x\sqrt{4q_0^2+1}+B_x^{(2)}\bigl(1+\sqrt{4q_0^2+1}\bigr)
+2B_x^{(3)}\sqrt{q_0^2+1}\Bigr)\right.\biggr.\nonumber \\
&\quad\quad\quad\quad\quad\biggl.\left.+2\mu q_0 \bigl(B^{(1)}_y+B_y^{(2)}+B_y^{(3)}\bigr) +T_x \right)\biggr.\nonumber \\ 
&\biggl. +(8\mu q_0^2+\xi)
\Bigl(-2\mu \Bigl( B^{(1)}_y\textstyle{\sqrt{4q_0^2+1}}+B_y^{(2)}\bigl(1+\textstyle{\sqrt{4q_0^2+1}}\bigr)
+2B_y^{(3)}\sqrt{q_0^2+1}\Bigr)+ T_y^{(2)}\Bigr)\biggr), \nonumber 
\end{align*}
\begin{align*}
F=& \textstyle{\frac{1}{D(2q_0)}}\biggl( (8\mu q_0^2+\xi)
\Bigl(-\mu \left(B^{(1)}_x\sqrt{4q_0^2+1}+B_x^{(2)}\bigl(1+\sqrt{4q_0^2+1}\bigr)
+2B_x^{(3)}\sqrt{q_0^2+1}\right)\Bigr.\biggr.\nonumber \\
&\quad\quad\quad\quad\quad\quad\quad\quad\quad\biggl.\Bigl.+2\mu q_0 \bigl(B^{(1)}_y+B_y^{(2)}+B_y^{(3)}\bigr) +T_x \Bigr)\biggr.\nonumber \\ 
&\biggl. +4\mu q_0 \textstyle{\sqrt{4q_0^2+\xi/(2\mu)}}
\Bigl(-2\mu \Bigl( B^{(1)}_y\textstyle{\sqrt{4q_0^2+1}}+B_y^{(2)}\bigl(1+\textstyle{\sqrt{4q_0^2+1}}\bigr)
+2B_y^{(3)}\textstyle{\sqrt{q_0^2+1}}\Bigr)+ T_y^{(2)}\Bigr)\biggr), \nonumber 
\end{align*}
\begin{align*}
G=\textstyle{\frac{1}{\sqrt{2\mu\xi}}} \, T_y^{(1)}
-\sqrt{\textstyle{\frac{2\mu}{\xi}}}\left( B_y^{(4)}+B_y^{(5)}+2B_y^{(6)}+B_y^{(7)}+2B_y^{(8)}\sqrt{q_0^2+1}+B_y^{(9)}\right).\nonumber 
\end{align*}

%\subsection{$\beta$} \label{App_beta}
%As $\rho^{(2)}=\beta \cos 2q_0x$, 
From 	\eqref{hochu} we get
\begin{align}
v_y^{(22)}&+\textstyle{\frac{\zeta}{2\mu}}\rho^{(2)}
\label{tyzhelayanonuzhnayaformula} \\
%&=	-\textstyle{\frac{1}{2}} \partial^2_{yy}V_y\cos^2 q_0x - q_0 v_x^{(1)}\sin q_0x-\partial_yv_y^{(1)}\cos q_0x +V^{(2)}\nonumber
&=-v_y^{(21)}+\textstyle{\frac{1}{2}}\bigl( \partial^2_{yy}V_y +\textstyle{\frac{\gamma q_0^2}{\mu}}  \bigr)\cos^2q_0 x -  q_0 v_x^{(1)}\sin q_0x+V^{(2)} 
\quad\text{for} \ y=0.\nonumber
\end{align}
Observe that the right-hand side of \eqref{tyzhelayanonuzhnayaformula} 
is a linear combination of a constant function and $\cos 2q_0x$. Then it follows from the spectral representation \eqref{spectral_representation} (for $\mathcal{L}=\mathcal{L}(1)$) that
$\rho^{(2)}=\beta \cos 2q_0x$. Moreover,
%$v_y^{(22)}+\partial_yV_y\beta \cos 2q_0x=\beta \Lambda(2q_0)\cos 2q_0x$, 
using \eqref{UrForVy}, \eqref{bc2}, \eqref{pervayatrudnyashtuka}, \eqref{explicit_v21x}--\eqref{explicit_v21y} we obtain
\begin{align}
\beta&= -\textstyle{\frac{1}{ \Lambda(2q_0)}}
\left(B^{(1)}_y+B_y^{(2)}+B_y^{(3)}+E\sqrt{4q_0^2+\xi/(2\mu)}+2q_0F\right) \nonumber\\
&+\textstyle{\frac{1}{ 2\Lambda(2q_0)}}\Bigl(D_1+\textstyle{\frac{\gamma q_0^2}{\mu}}\Bigr)+\textstyle{\frac{3}{4\Lambda(2q_0)}}\biggl(
\textstyle{\frac{\zeta\xi}{2\mu \bigl(4\mu+\sqrt{2\mu\xi}\bigr)}}-\textstyle{\frac{\zeta_i}{\sqrt{2\mu\xi}+2\mu}}+
\textstyle{\frac{\zeta}{\mu}}\biggr),
\label{beta}
\end{align}
and
\begin{align}
V^{(2)}&=B_y^{(4)}
+B_y^{(6)}+B_y^{(8)}+G +\textstyle{\frac{1}{2}}D_1+\textstyle{\frac{1}{4}}\Bigl(
\textstyle{\frac{\zeta\xi}{2\mu \bigl(4\mu+\sqrt{2\mu\xi}\bigr)}}-\textstyle{\frac{\zeta_i}{\sqrt{2\mu\xi}+2\mu}}+
\textstyle{\frac{\zeta}{\mu}}\Bigr)\nonumber
\\ 
&=-b V^{(0)}+\Bigl(2\textstyle{\sqrt{q_0^2+1}}-q_0^2-1\Bigr)\Bigl( \textstyle{\frac{\zeta_i}{4 \bigl(\sqrt{2\mu\xi}+\xi\bigr)}}- \textstyle{\frac{\zeta}{2\sqrt{2\mu\xi}+\xi}\Bigr)}
\nonumber\\
&
-\textstyle{\frac{\zeta\sqrt{q_0^2+1}}{2\sqrt{2\mu\xi}\sqrt{q_0^2+1}+\xi}} +\textstyle{\frac{1}{\sqrt{2\mu\xi}}} \, T_y^{(1)}+\textstyle{\frac{1}{2}}D_1+\textstyle{\frac{1}{4}}\Bigl(
\frac{\zeta\xi}{2\mu \bigl(4\mu+\sqrt{2\mu\xi}\bigr)}-\frac{\zeta_i}{\sqrt{2\mu\xi}+2\mu}+\frac{\zeta}{\mu}\Bigr)
\nonumber\\
&=:V^{(21)}-b V^{(0)}.
\label{bagatoV}
%	&-\frac{1}{ 8\mu}
%\left(\frac{\zeta\xi}{4\mu+\sqrt{2\mu\xi}}+\frac{\zeta_i\xi}{\sqrt{2\mu\xi}+\xi}
%+2\zeta-\zeta_i +2\gamma q_0^2\right) \nonumber\\
%	&+\frac{1}{2}\left(D_1+ \frac{\zeta\xi}{2\mu\left(4\mu+\sqrt{2\mu\xi}\right)}-\frac{\zeta_i}{\sqrt{2\mu\xi}+2\mu} +\frac{2\zeta +\gamma q_0^2}{2\mu}\right)
%\nonumber  
\end{align}

Next, we establish the boundary values of components of ${\bf v}^{(2)}$ and some of their derivatives on the line $y=0$,
that will be used in  \eqref{fse}  and in the calculations of the stress vector for ${\bf v}^{(3)}$. To this end substitute the expressions
for components of ${\bf V}$, ${\bf p}^{(1)}$ and ${\bf p}^{(2)}$ in \eqref{v^2_eq} (see  \eqref{Vy}, \eqref{p_1}, \eqref{splitting_p2}--\eqref{p_21y}, \eqref{p_22}):
%write down the equations } 
%Next we establish some formulas for ${\bf v}^{(2)}$ and its derivatives on the line $y=0$ that will be used below to calculate the explicit form of boundary conditions for ${\bf v}^{(3)}$. To this end we write the equations satisfied by ${\bf v}^{(2)}$: 
\begin{align}
&\mu (\Delta v_x^{(2)}+\partial_x {\rm div}{\bf v}^{(2)})-\xi v_x^{(2)} %\nonumber\\ 
%&\quad 
=\zeta_i q_0\Bigl(\textstyle{\frac{\sqrt{q_0^2+1}}{2}}-2\beta\Bigr) \, e^{y\sqrt{4q_0^2+1}} \,\sin2q_0x \nonumber\\  
&\quad-\zeta q_0\bigl(1+\textstyle{\sqrt{4q_0^2+1}}\bigr) \Bigl(\textstyle{\frac{\sqrt{q_0^2+1}}{2}}-2\beta \Bigr) e^{y+y\sqrt{4q_0^2+1}} \sin2q_0x \nonumber\\  
&\quad+\zeta q_0\bigl(q_0^2-\textstyle{\sqrt{q_0^2+1}}\bigr) e^{2y\sqrt{q_0^2+1}} \sin2q_0x,
\label{strashni_rivnyannya_dlya_v^2_1}
\end{align}
\begin{align}
\mu &(\Delta v_y^{(2)}+\partial_y {\rm div}{\bf v}^{(2)})-\xi v_y^{(2)}
=\zeta_i \Bigl(\beta-\textstyle{\frac{q_0^2+2\sqrt{q_0^2+1}-1}{4}}\Bigr)\, e^{y\sqrt{4q_0^2+1}} \cos2q_0x \nonumber\\ 
& 	-\zeta_i \Bigl(\textstyle{\frac{2b\xi}{2\mu-\xi}}+b+
by+\textstyle{\frac{2\sqrt{q_0^2+1}-q_0^2-1}{4}}\Bigr) e^y+2b\xi\Bigl( \textstyle{\frac{\zeta_i}{2\mu-\xi}}\textstyle{\frac{\sqrt{2\mu}}{\sqrt{\xi}}}-
\textstyle{\frac{\zeta}{8\mu-\xi}}\textstyle{\frac{\sqrt{\xi}}{\sqrt{2\mu}}}\Bigr) e^{\sqrt{\xi}y/\sqrt{2\mu}}
\nonumber\\  
&+\zeta\Bigl(\textstyle{\frac{4b\xi}{8\mu-\xi}}+2b+4by+2\sqrt{q_0^2+1}-q_0^2-1\Bigr)e^{2y}
\label{strashni_rivnyannya_dlya_v^2_2}\\ 
&+\zeta \Bigl(\textstyle{\sqrt{q_0^2+1}}-q_0^2\Bigr) e^{2y\textstyle{\sqrt{q_0^2+1}}} \cos2q_0x +\zeta\sqrt{q_0^2+1} e^{2y\sqrt{q_0^2+1}}
\nonumber\\  
&+\zeta 
\Bigl( \textstyle{\frac{(q_0^2+2\sqrt{q_0^2+1}-1 -4\beta)(1+\sqrt{4q_0^2+1})}{2}}-q_0^2\sqrt{q_0^2+1}
+4\beta q_0^2\Bigr)
e^{y+y\sqrt{4q_0^2+1}} \cos2q_0x. \nonumber
\end{align}
Let ${\bf w}^{(2)}(x,y)$ denote the vector function  obtained by subtracting from ${\bf v}^{(2)}(x,y)$ its average in $x$ over the period.
%We compute $w_x^{(2)}$ for $y=0$ arguing similarly to deriving of \eqref{pervayatrudnyashtuka_v_x} for $v_x^{(1)}$. Namely, 
Find the divergence of  ${\bf w}^{(2)}(x,y)$ for $y=0$. Taking derivative of   \eqref{v^2_boundarycondition_first} in $x$ and using \eqref{pervayatrudnyashtuka}, \eqref{pervoedlyapervogo} we obtain that
% $\mu\left(\partial^2_{xx} w_y^{(2)}+\partial^2_{xy} w_x^{(2)}\right)$
\begin{align}
\mu&\bigl(\partial^2_{xx} w_y^{(2)}+\partial^2_{xy} w_x^{(2)}\bigr)\notag\\
&=
\bigl(\zeta_i q_0^2 - \zeta q_0^2 \bigl(\textstyle{\sqrt{q_0^2+1}}+1\bigr)\bigr)\cos 2q_0 x+\left(4\zeta q_0^2 \beta -\gamma q_0^4\right) \cos 2q_0 x
\label{veselytsya_i_psihue_ves_narod1}
\\
& -(4\mu q_0^2+\xi)
\Bigl(D_1+ \textstyle{\frac{\zeta\xi}{2\mu\left(4\mu+\sqrt{2\mu\xi}\right)}}-
\textstyle{\frac{\zeta_i}{\sqrt{2\mu\xi}+2\mu}} +\textstyle{\frac{2\zeta +\gamma q_0^2}{2\mu}}\Bigr)\cos 2q_0 x, 	\notag	
\end{align}
for $y=0$, also  thanks to \eqref{tyzhelayanonuzhnayaformula} we have%get the formula for $v_y^{(2)}$ 
\begin{align}
v_y^{(2)}&=-\textstyle{\frac{\zeta\beta}{2\mu}} \cos 2q_0x
+\textstyle{\frac{1}{2}}\Bigl(D_1+ \textstyle{\frac{3\zeta\xi}{4\mu\left(4\mu+\sqrt{2\mu\xi}\right)}}-
\textstyle{\frac{3\zeta_i}{2\sqrt{2\mu\xi}+4\mu}} +
\textstyle{\frac{6\zeta +4\gamma q_0^2}{4\mu}}\Bigr)		\cos 2 q_0x \notag \\
&-\textstyle{\frac{1}{2}} D_1+\textstyle{\frac{1}{4}}\Bigl(\textstyle{\frac{\zeta_i}{\sqrt{2\mu\xi}+2\mu}}-
\textstyle{\frac{\zeta\xi}{2\mu\left(4\mu+\sqrt{2\mu\xi}\right)}} -\textstyle{\frac{\zeta}{\mu}}\Bigr)+V^{(2)}.
\label{veselytsya_i_psihue_ves_narod2}
\end{align}
%Then from \eqref{strashni_rivnyannya_dlya_v^2_2} 
Since
$$
\mu (\Delta w_y^{(2)}+\partial_y {\rm div}{\bf w}^{(2)})-\xi w_y^{(2)}=
2  \mu \partial_y {\rm div} {\bf w}^{(2)}-
\mu\bigl(\partial^2_{xx} w_y^{(2)}+\partial^2_{xy} w_x^{(2)}\bigr) - \bigl(8\mu q_0^2+\xi \bigr)w^{(2)}_y,
$$
from \eqref{strashni_rivnyannya_dlya_v^2_2}--\eqref{veselytsya_i_psihue_ves_narod2} we get the following boundary condition for $y=0$:
\begin{align}
\label{krokodilchik}
&2\mu\partial_y {\rm div} {\bf w}^{(2)} =\bigl(2\mu q_0^2 -\textstyle{\frac{\xi}{4}}\bigr) 
\Bigl( \textstyle{\frac{\zeta\xi}{2\mu\left(4\mu+\sqrt{2\mu\xi}\right)}}-
\textstyle{\frac{\zeta_i}{\sqrt{2\mu\xi}+2\mu}} +\textstyle{\frac{\zeta}{\mu}}\Bigr)\cos2q_0x\nonumber\\
&\quad+ \Bigl(-\textstyle{\frac{\xi}{2}}D_1  +\gamma q_0^4+\zeta_i q_0^2+2\zeta q_0^2\bigl(4\beta-\sqrt{q_0^2+1}-1\bigr)+\zeta \sqrt{q_0^2+1} \Bigr)
\cos2q_0x\\
&\quad +\Bigl(\zeta	\textstyle{\frac{1+\sqrt{4q_0^2+1}}{2}}-\textstyle{\frac{\zeta_i}{4}} \Bigr)
\Bigl(q_0^2+2\sqrt{q_0^2+1}-1-4\beta\Bigr)\cos2q_0x.\notag
\end{align}
Then using \eqref{strashni_rivnyannya_dlya_v^2_1}--\eqref{strashni_rivnyannya_dlya_v^2_2} we obtain that $ {\rm div} {\bf w}^{(2)}$ satisfies the equation  
\begin{align}
&2\mu \Delta {\rm div} {\bf w}^{(2)}-\xi {\rm div} {\bf w}^{(2)} 
= 2\zeta \Bigl(q_0^2-\textstyle{\sqrt{q_0^2+1}}\Bigr)^2 e^{2y\sqrt{q_0^2+1}} \cos 2q_0x  \nonumber\\ 
&+\zeta_i \Bigl( q_0^2 \bigl( \textstyle{\sqrt{q_0^2+1}} -4\beta\bigr) +\textstyle{\frac{1}{4}} \textstyle{\sqrt{4q_0^2+1}}\bigl(4\beta -q_0^2 -2\sqrt{q_0^2+1} +1\bigr)\Bigr) \, e^{y\sqrt{4q_0^2+1}}\cos 2q_0x 
\nonumber\\  
& +\textstyle{\frac{\zeta}{2}} \bigl(1+\textstyle{\sqrt{4q_0^2+1}}\bigr)
\Bigl(  \bigl(q_0^2+2\sqrt{q_0^2+1}-1-4\beta\bigr)(1+\textstyle{\sqrt{4q_0^2+1}})
\nonumber\\
&\quad\quad \quad \quad\quad\quad \quad \quad\quad \Bigl. -2 q_0^2\textstyle{\sqrt{q_0^2+1}}+8\beta q_0^2\Bigr)
e^{y+y\sqrt{4q_0^2+1}} \cos2 q_0x, \nonumber
\end{align}
solving which with boundary condition \eqref{krokodilchik} we derive 
\begin{align}
\label{div_v2}
{\rm div} {\bf w}^{(2)}\bigl|_{y=0}\bigr.=\textstyle{\frac{D_2}{2\mu\sqrt{4 q_0^2+{ \textstyle\frac{\xi}{2\mu}}}}} \cos 2q_0x,
\end{align}
where 
\begin{align}
\label{D2}
&D_2= 
-\textstyle{\frac{\xi}{2}}D_1  +\gamma q_0^4+\zeta_i q_0^2+2\zeta q_0^2\bigl(4\beta-\sqrt{q_0^2+1}-1\bigr)+\zeta \sqrt{q_0^2+1} 
\notag\\
&+ (2\mu q_0^2 -\xi/4) 
\Bigl( \textstyle{\frac{\zeta\xi}{2\mu\left(4\mu+\sqrt{2\mu\xi}\right)}}-
\textstyle{\frac{\zeta_i}{\sqrt{2\mu\xi}+2\mu}} +\textstyle{\frac{\zeta}{\mu}}\Bigr)\notag\\
& + \textstyle{\frac{1}{4}}\bigl(2\zeta+2\zeta\sqrt{4q_0^2+1}-\zeta_i\bigr)
\bigl(q_0^2+2\sqrt{q_0^2+1}-1-4\beta\bigr)\nonumber\\
&
+\textstyle{\frac{\zeta_i}{ \bigl(\sqrt{4 q_0^2+{ \textstyle\frac{\xi}{2\mu}}}
	+\sqrt{4q_0^2+1}\bigr)}}\Bigl( 4\beta q_0^2- q_0^2 \sqrt{q_0^2+1}  +\frac{1}{4} \sqrt{4q_0^2+1}\bigl( q_0^2 +2\sqrt{q_0^2+1}-1-4\beta\bigr)\Bigr)\nonumber\\  
&-\textstyle{\frac{2\zeta}{\bigl(\sqrt{4 q_0^2+{ \textstyle\frac{\xi}{2\mu}}}+2\sqrt{q_0^2+1}\bigr)}} \bigl(q_0^2-\sqrt{q_0^2+1}\bigr)^2 
%e^{2y\sqrt{q_0^2+1}}  
\nonumber\\
&-\textstyle{\frac{\zeta\bigl(1+\sqrt{4q_0^2+1}\bigr)}{\bigl(\sqrt{4 q_0^2+{ \textstyle\frac{\xi}{2\mu}}}+1+\sqrt{4q_0^2+1}\bigr)}}
%{2\Bigl(4\mu+4\mu \sqrt{4q_0^2+1} -\xi\Bigr)} 
\Bigl(4\beta q_0^2 -q_0^2\sqrt{q_0^2+1}
\Bigr.\nonumber\\
& \quad\quad\quad \quad \quad\quad\quad \quad \quad\quad \Bigl. + \textstyle{\frac{1}{2}}\bigl(q_0^2+2\sqrt{q_0^2+1}-1-4\beta\bigr)\bigl(1+\sqrt{4q_0^2+1}\bigr)\Bigr).
% 	e^{y+y\sqrt{4q_0^2+1}} 
\end{align}
Now using \eqref{v^2_boundarycondition_second} we get 
\begin{align}
\partial_y v^{(2)}_y&=\textstyle{\frac{1}{2\mu}}\Bigl( -\textstyle{\frac{3}{4}}\zeta q_0^2 -\textstyle{\frac{\zeta_i}{4}} +\zeta \sqrt{q_0^2+1}+\textstyle{\frac{\zeta\xi}{8\mu}} \Bigr)
+\textstyle{\frac{1}{2\mu}}\Bigl(
\textstyle{\frac{3\zeta}{4}}
q_0^2+\textstyle{\frac{\zeta\xi}{8\mu}}+\zeta\sqrt{q_0^2+1}  \nonumber\\
&-2\mu \beta \Bigl(\textstyle{\frac{\xi\zeta}{2\mu\bigl(4\mu+\sqrt{2\mu \xi}\bigr)}} -
\textstyle{\frac{\zeta_i}{\sqrt{2\mu\xi}+2\mu}}+\textstyle{\frac{\zeta}{\mu}}
\Bigr)
-4\gamma q_0^2 \beta   -\textstyle{\frac{\zeta_i}{4}} 
\Bigr) \cos 2 q_0 x,
\label{before_bezymyashka}
\end{align}
therefore
\begin{align}
v^{(2)}_x&=\textstyle{\frac{\sin 2 q_0 x}{4q_0\mu}}\Bigl( 
\textstyle{	\frac{D_2}{\sqrt{4 q_0^2+{ \textstyle\frac{\xi}{2\mu}}}}} -
\textstyle{	\frac{\zeta}{4}}
\Bigl( 3q_0^2+\textstyle{\frac{\xi}{2\mu}}+4\sqrt{q_0^2+1} -4\beta\bigl( \textstyle{\frac{\xi}{4\mu+\sqrt{2\mu \xi}}}+2\bigr)\Bigr)\nonumber\\
&+4\gamma q_0^2 \beta -\zeta_i  
\Bigl( \textstyle{\frac{\beta\sqrt{2\mu}}{\sqrt{\xi}+\sqrt{2\mu}}}-\textstyle{\frac{1}{4}} \Bigr)	\Bigr).
\label{bezymyashka}
\end{align}   

%%%%%%%%%%%%%%%%%%  

\section{Representations of ${\bf p}^{(3)}$ and ${\bf v}^{(3)}$} \label{predfinish}
The vector functions ${\bf p}^{(3)}, {\bf v}^{(3)}$ appearing at 
%Recall that 
the order $\alpha^3$ in the expansions \eqref{expan_p}, \eqref{expan_v} are represented as 
%satisfies
%\begin{equation}
%\Delta p^{(3)}=p^{(3)}+2bp^{(1)} \quad \text{for}\ y<0,
%\end{equation}
\begin{equation}
{\bf p}^{(3)}={\bf p}^{(311)}+{\bf p}^{(312)}+ {\bf p}^{(32)}, \ \ 
{\bf v}^{(3)}={\bf v}^{(311)}+{\bf v}^{(312)}+ {\bf v}^{(32)},
\end{equation}
and to find the coefficient $b$ in \eqref{expan_theta} we need to calculate only ${\bf p}^{(311)}$ and ${\bf v}^{(311)}$, whose $x$-component ($y$-component) contains all terms with the factor $\sin q_0x$ ($\cos q_0x$), except for those additionally having the multiplier $b$. It follows from  \eqref{bc_p3} that on the line $y=0$
\begin{align}
p&^{(311)}_x=%\textstyle{\frac{bq_0y}{\sqrt{q_0^2+1}}}e^{y\sqrt{q_0^2+1}}\sin q_0x\nonumber\\
\Bigl(\beta q_0\textstyle{\frac{\sqrt{q_0^2+1}}{2}}+
\sqrt{4q_0^2+1}\bigl(\textstyle{\frac{q_0}{4}}\sqrt{q_0^2+1}- \beta q_0\bigr)-\textstyle{\frac{q_0^3}{2}}-\textstyle{\frac{q_0}{8}}\Bigr)\sin q_0x
%&+ \Bigl(\textstyle{\sqrt{4q_0^2+1}}\bigl(\textstyle{\frac{q_0}{4}}\sqrt{q_0^2+1}- \beta q_0\bigr)-\beta q_0\textstyle{\frac{\sqrt{q_0^2+1}}{2}}-\textstyle{\frac{q_0}{8}}\Bigr)e^{y\sqrt{9q_0^2+1}}	\sin 3q_0x ,
\end{align}
\begin{align}
p^{(311)}_y&=%\textstyle{\frac{-by}{\sqrt{q_0^2+1}}}e^{y\sqrt{q_0^2+1}}\cos q_0x%\nonumber\\
\Bigl(\textstyle{\frac{\beta}{2}}\bigl(\sqrt{q_0^2+1}-2q_0^2-1\bigr)-
\textstyle{\frac{\sqrt{q_0^2+1}}{2}} +\textstyle{\frac{5}{8}}q_0^2+\textstyle{\frac{1}{2}}\nonumber\\
&	%\Bigr.\nonumber\\	\quad\quad\Bigl.
-\textstyle{\sqrt{4q_0^2+1}}\Bigl(\textstyle{\frac{\sqrt{q_0^2+1}}{4}}+
\textstyle{\frac{q_0^2}{8}}-\textstyle{\frac{1}{8}}-\textstyle{\frac{\beta}{2}}
\Bigr) \Bigr)\cos q_0x.
%	&+ \Bigl(  q_0^2\bigl(\beta+\textstyle{\frac{1}{8}}\bigr)-\textstyle{\frac{\beta}{2}}\bigl(1-\sqrt{q_0^2+1}\bigr)+\textstyle{\frac{1}{12}}\Bigr.\nonumber\\&\quad\quad\Bigl.-	\textstyle{\sqrt{4q_0^2+1}}\Bigl(\textstyle{\frac{\sqrt{q_0^2+1}}{4}}+\textstyle{\frac{q_0^2}{8}}-\textstyle{\frac{1}{8}}-\textstyle{\frac{\beta}{2}}	\Bigr) \Bigr) e^{y\sqrt{9q_0^2+1}}\cos 3q_0x.
\end{align}
Then since $\Delta {\bf p}^{(311)}={\bf p}^{(311)}$ we have ${\bf p}^{(311)}(x,y)={\bf p}^{(311)}(x,0)e^{y\sqrt{q_0^2+1}}$. We substitute this expression for  ${\bf p}^{(311)}$ in place of ${\bf p}^{(3)}$ in \eqref{eq_v3} and use \eqref{p_1} and  \eqref{splitting_p2}--\eqref{p_21y} with \eqref{p22} to conclude that components of  ${\bf v}^{(311)}$ satisfy the following equations for $y<0$ %(actually in the entire $\mathbb{R}^2$)
\begin{align}
\mu (\Delta& v_x^{(311)}+\partial_x{\rm div}{\bf v}^{(311)})-\xi v_x^{(311)}\nonumber\\ 
&=\bigl(-\zeta_i H^{(1)}_x+\zeta H^{(2)}_x e^y	+\zeta H^{(3)}_xe^{y\sqrt{4q_0^2+1}}\bigr)e^{y\sqrt{q_0^2+1}}\sin q_0x,
\nonumber
\end{align}
\begin{align}
\mu (\Delta& v_y^{(311)}+\partial_x{\rm div}{\bf v}^{(311)})-\xi v_y^{(311)}
\nonumber\\ 
& =\bigl(-\zeta_i H^{(1)}_y+\zeta H^{(2)}_y e^y	+\zeta H^{(3)}_ye^{y\sqrt{4q_0^2+1}}\bigr)e^{y\sqrt{q_0^2+1}}\cos q_0x,\nonumber
\end{align}
where
\begin{align}
H^{(1)}_x&=%bq_0+
\beta q_0\textstyle{\frac{\sqrt{q_0^2+1}}{2}}+
\textstyle{\sqrt{4q_0^2+1}}\Bigl(\textstyle{\frac{q_0}{4}}\textstyle{\sqrt{q_0^2+1}}- \beta q_0\Bigr)-\textstyle{\frac{q_0^3}{2}}-\textstyle{\frac{q_0}{8}}, \nonumber
\end{align}
\begin{align}
H^{(1)}_y&=%-2b + 
\textstyle{\frac{\beta}{2}}\bigl(\textstyle{\sqrt{q_0^2+1}}-2q_0^2-1\bigr)-
\textstyle{\frac{\sqrt{q_0^2+1}}{2}}
+\textstyle{\frac{5}{8}}q_0^2+\textstyle{\frac{1}{2}}\\
&-\textstyle{\sqrt{4q_0^2+1}}\Bigl(\textstyle{\frac{\sqrt{q_0^2+1}}{4}}+
\textstyle{\frac{q_0^2}{8}}-\textstyle{\frac{1}{8}}-\textstyle{\frac{\beta}{2}}	\Bigr),  	\nonumber
\end{align}
\begin{align}
H^{(2)}_x&= 
%\frac{bq_0}{\sqrt{q_0^2+1}}+ b q_0+ 
q_0\bigl(1+\textstyle{\sqrt{q_0^2+1}}\bigr)\Bigl(
\frac{\sqrt{q_0^2+1}}{2}-\frac{q_0^2}{4}-\frac{1}{4}\Bigr)\nonumber\\
&+\bigl(1+\textstyle{\sqrt{q_0^2+1}}\bigr)	\Bigl(\beta q_0
\textstyle{\frac{\sqrt{q_0^2+1}}{2}}+
\textstyle{\sqrt{4q_0^2+1}}\Bigl(\textstyle{\frac{q_0}{4}}\textstyle{\sqrt{q_0^2+1}}- \beta q_0\Bigr)-\textstyle{\frac{q_0^3}{2}}-\textstyle{\frac{q_0}{8}}\Bigr),\nonumber
\end{align}
\begin{align}
H^{(2)}_y&= 
%\frac{-2b}{\sqrt{q_0^2+1}}\Bigl(1+\sqrt{q_0^2+1}\Bigr)^2 +	
\Bigl(\textstyle{\frac{\sqrt{q_0^2+1}}{2}}	-\textstyle{\frac{q_0^2}{4}}-\textstyle{\frac{1}{4}}\Bigr)
\bigl(
q_0^2-2\bigl(1+\textstyle{\sqrt{q_0^2+1}}\bigr)\Bigr)\nonumber\\
&+q_0\Bigl(\beta q_0\textstyle{\frac{\sqrt{q_0^2+1}}{2}}+
\textstyle{\sqrt{4q_0^2+1}}\Bigl(\textstyle{\frac{q_0}{4}}\sqrt{q_0^2+1}- \beta q_0\Bigr)-\textstyle{\frac{q_0^3}{2}}-\textstyle{\frac{q_0}{8}}\Bigr)\nonumber\\
&+2\bigl(1+\textstyle{\sqrt{q_0^2+1}}\bigr)
\Bigl( \textstyle{\frac{\beta}{2}}\bigl(\textstyle{\sqrt{q_0^2+1}}-2q_0^2-1\bigr)-
\textstyle{\frac{\sqrt{q_0^2+1}}{2}}
+\textstyle{\frac{5}{8}}q_0^2+\textstyle{\frac{1}{2}}
\Bigr.\nonumber\\	&\Bigl.-\textstyle{\sqrt{4q_0^2+1}}\Bigl(\textstyle{\frac{\sqrt{q_0^2+1}}{4}}+
\textstyle{\frac{q_0^2}{8}}-\textstyle{\frac{1}{8}}-\textstyle{\frac{\beta}{2}}
\Bigr) \Bigr),\nonumber
\end{align}
\begin{align}
H^{(3)}_x&= 
\textstyle{\frac{q_0^3}{2}}\textstyle{\sqrt{q_0^2+1}} -2\beta q_0^3+\bigl(\textstyle{\sqrt{q_0^2+1}}+\textstyle{\sqrt{4q_0^2+1}}\bigr)\Bigl(
\textstyle{\frac{q_0}{8}}-\textstyle{\frac{q_0^3}{8}}-\textstyle{\frac{\beta q_0}{2}}\Bigr),\notag
\end{align}
\begin{align}
H^{(3)}_y&= 
\textstyle{\frac{q_0^2}{8}}-\textstyle{\frac{q_0^4}{8}}-\textstyle{	\frac{\beta q_0^2}{2}} 
-\bigl(\textstyle{\sqrt{q_0^2+1}}+\textstyle{\sqrt{4q_0^2+1}}\bigr)\Bigl(
\textstyle{\frac{q_0^2}{4}}+\textstyle{\frac{\sqrt{q_0^2+1}}{2}}-\textstyle{\frac{1}{4}}-\beta\Bigr).\nonumber
\end{align}

%\begin{align}.
%	C_1=&\left[\beta q_0\frac{\sqrt{q_0^2+1}}{2}+
%	\sqrt{4q_0^2+1}\left(\frac{q_0}{4}\sqrt{q_0^2+1}- \beta q_0\right)-\frac{q_0^3}{2}-\frac{q_0}{8}\right] \nonumber\\
%C_2=	& \left[
%	\sqrt{4q_0^2+1}\left(\frac{q_0}{4}\sqrt{q_0^2+1}- \beta q_0\right)-\beta q_0\frac{\sqrt{q_0^2+1}}{2}-\frac{q_0}{8}\right]
%\end{align}
%\begin{align}
%C_3&=\left[  \frac{\beta}{2}\bigl(\sqrt{q_0^2+1}-2q_0^2-1\bigr)-\frac{\sqrt{q_0^2+1}}{2}
%	+\frac{5}{8}q_0^2+\frac{1}{2}-b
%	\right.\nonumber\\	&\quad\quad\left.-\sqrt{4q_0^2+1}\left(\frac{\sqrt{q_0^2+1}}{4}+\frac{q_0^2}{8}-\frac{1}{8}-\frac{\beta}{2}
%	\right) \right]
%	\nonumber\\
%	C_4=& \left[  q_0^2\bigl(\beta+\frac{1}{8}\bigr)-\frac{\beta}{2}\bigl(1-\sqrt{q_0^2+1}\bigr)+
%	\frac{1}{12}\right.\nonumber\\
%	&\quad\quad\left.-	\sqrt{4q_0^2+1}\left(\frac{\sqrt{q_0^2+1}}{4}+\frac{q_0^2}{8}-\frac{1}{8}-\frac{\beta}{2}
%	\right) \right]
%\end{align}
We also have boundary conditions for $y=0$:
\begin{align}
&\mu (\partial_xv_y^{(311)}+\partial_y v_x^{(311)})= Q\sin q_0x,\label{bound_cond_v311first}\\ 
&2\mu \partial_yv_y^{(311)}=R\cos q_0x,	\label{bound_cond_v311second}\
\end{align}
where constants $Q, R$ are obtained from \eqref{eto_est_predposlednii_cond}--\eqref{eto_est_poslednii_cond} (see computations in Appendix \ref{finish}).
%We are interested in  determining ${\bf v}^{311}$, which appears in the solvability condition. The latter conditionallows us to determine the coefficient $b$.  
The solution ${\bf v}^{311}$ is represented as follows:  
\begin{align}
\label{v311_x}
v_x^{(311)}&=
\bigl( I_x^{(1)}e^{y}+ I_x^{(2)}
+I_x^{(3)}e^{y\sqrt{4q_0^2+1}}\bigr) e^{y\sqrt{q_0^2+1}}\sin q_0x\nonumber\\	
&-\Bigl(q_0 M e^{y\sqrt{q_0^2+\xi/(2\mu)}} +\textstyle{\sqrt{q_0^2+\textstyle \frac{\xi}{\mu}}} N e^{y\sqrt{q_0^2+\xi/\mu}}
%\Bigr)
\Bigr)\sin q_0x, 
\end{align}
\begin{align}
\label{v311_y}
v_y^{(311)}&=\bigl( I_y^{(1)}e^{y}+ I_y^{(2)}
+I_y^{(3)}e^{y\sqrt{4q_0^2+1}}\bigr) e^{y\sqrt{q_0^2+1}} \cos q_0x \notag\\
&+\left(  \textstyle{ \sqrt{q_0^2+\frac{\xi}{2\mu}}} M e^{y\sqrt{q_0^2+\xi/(2\mu)}}+q_0 N e^{y\sqrt{q_0^2+\xi/\mu}} \right)\cos q_0x,
\end{align}
where 
\begin{align}
%were \tilde I^{(2)}_x and _y
I^{(1)}_x&= \textstyle{\frac{\zeta}{d_1}} 
\left( H_x^{(2)}\Bigr(\mu\bigl(q_0^2+4\textstyle{\sqrt{q_0^2+1}}+4\bigr)-\xi \Bigl) +
H_y^{(2)}\mu q_0\bigl(1+\textstyle{\sqrt{q_0^2+1}}\bigr)  \right), \nonumber\\
I^{(1)}_y&=
\textstyle{\frac{\zeta}{d_1}} 
\left( H_y^{(2)}\Bigl(\mu\bigl(2\textstyle{\sqrt{q_0^2+1}}+2-q_0^2\bigr)-\xi \Bigr) -H_x^{(2)}\mu q_0\bigl(1+\textstyle{\sqrt{q_0^2+1}}\bigr)
\right),\nonumber
\end{align}
\begin{align}
d_1=8\mu^2 q_0^2 +(16\mu^2 -6\mu\xi)\bigl(1+\textstyle{\sqrt{q_0^2+1}}\bigr)+\xi^2,\nonumber
\end{align}
\begin{align}
%were \tilde I^{(4)}_xy
I^{(2)}_x&=-\textstyle{\frac{\zeta_i}{2\mu^2-3\mu\xi+\xi^2}} \left( H_x^{(1)}(\mu q_0^2+2\mu-\xi)+H_y^{(1)}\mu q_0\sqrt{q_0^2+1}  \right),\nonumber
\\
I^{(2)}_y&=-\textstyle{\frac{\zeta_i}{2\mu^2-3\mu\xi+\xi^2} }
\left(
H_y^{(1)}(\mu-\mu q_0^2-\xi)-H_x^{(1)}\mu q_0\sqrt{q_0^2+1} \nonumber
\right),
\end{align}
\begin{align}
%were I^{(5)}_xy 
&I^{(3)}_x=\textstyle{\frac{\zeta}{d_2}}\left(\bigl(9\mu q_0^2+4\mu +4\mu\sqrt{4q_0^4+5q_0^2+1}-\xi\bigr) H_x^{(3)}	\right. \nonumber\\ 
& \quad\quad \quad \quad \quad \ \left.+\mu q_0 \bigl(\textstyle{\sqrt{q_0^2+1}}+\textstyle{\sqrt{4q_0^2+1}}\bigr) H_y^{(3)}\right),\nonumber\\
&I^{(3)}_y=\textstyle{\frac{\zeta}{d_2}}\left(\bigl(3\mu q_0^2+2\mu +2\mu\sqrt{4q_0^4+5q_0^2+1}-\xi\bigr) H_y^{(3)}	\right. \nonumber\\ 
& \quad\quad \quad \quad \quad \ \left.-\mu q_0 \bigl(\textstyle{\sqrt{q_0^2+1}}+\textstyle{\sqrt{4q_0^2+1}}\bigr) H_x^{(3)}\right),\nonumber\\
& d_2=4\mu q_0^2 \bigl(16\mu q_0^2 +18\mu-3\xi \bigr)%64\mu^2 q_0^4+3\mu q_0^2 (24\mu-4\xi)
\nonumber\\
&\quad \quad+2\mu\bigl(16\mu q_0^2 +8\mu-3\xi \bigr)\textstyle{\sqrt{4q_0^4+5q_0^2+1}}+16\mu^2-6\mu\xi+\xi^2.
\end{align}
The last two terms
%second lines 
in \eqref{v311_x} and \eqref{v311_y}  represent the linear combination   $$M\nabla (e^{y\sqrt{q_0^2+\xi/(2\mu)}}\cos q_0x)+N\nabla^\perp (e^{y\sqrt{q_0^2+\xi/\mu}}\sin q_0x)$$ of vector functions satisfying  the homogeneous 
equation $\mu(\Delta \,\cdot\,+\nabla {\rm div}\,\cdot\,)-\xi \,\cdot\,=0$, and coefficients $M$ and $N$ are found from the boundary conditions \eqref{bound_cond_v311first}--\eqref{bound_cond_v311second},
\begin{align}
&M =-\textstyle{\frac{1}{D(q_0,\mu,\xi)}}\Bigl( 2\mu q_0 \sqrt{q_0^2+\textstyle\frac{\xi}{\mu}}\,\tilde Q+ (2\mu q_0^2+\xi) \,\tilde R      \Bigr),
\\
&N=\textstyle{\frac{1}{D(q_0,\mu,\xi)}}\Bigl( (2\mu q_0^2+\xi)\,\tilde Q+ 2\mu q_0 \textstyle{\sqrt{q_0^2+\frac{\xi}{2\mu}}} \,\tilde R \Bigr),
\end{align}
where
\begin{align}
\tilde Q&= Q+\mu q_0 \bigl( I_y^{(1)} + I_y^{(2)} +I_y^{(3)}\bigr)\\
&-\mu\Bigl( %\bigl(1+\textstyle{\sqrt{q_0^2+1}}\bigr) 
I_x^{(1)} +\textstyle{\sqrt{q_0^2+1}}\bigl( I_x^{(1)} +I_x^{(2)} + I_x^{(3)}\bigr)    %\bigl(\textstyle{\sqrt{q_0^2+1}}
+\textstyle{\sqrt{4q_0^2+1}}%\bigr)  
I_x^{(3)} \Bigr),\notag
\end{align}
\begin{align}
\tilde R%&
= R    %\notag
%&
-2\mu%\Bigl( \Bigl(1+\textstyle{\sqrt{q_0^2+1}}\Bigr)  I_y^{(1)} +\textstyle{\sqrt{q_0^2+1}}  \, I_y^{(2)}  + \Bigl(\textstyle{\sqrt{q_0^2+1}}+\textstyle{\sqrt{4q_0^2+1}}\Bigr)  I_y^{(3)} \Bigr)
\Bigl( %\bigl(1+\textstyle{\sqrt{q_0^2+1}}\bigr) 
I_y^{(1)} +\textstyle{\sqrt{q_0^2+1}}\bigl( I_y^{(1)} +I_y^{(2)} + I_y^{(3)}\bigr) %\bigl(\textstyle{\sqrt{q_0^2+1}}
+\textstyle{\sqrt{4q_0^2+1}}%\bigr)  
I_y^{(3)} \Bigr),\notag
\end{align}
and $Q$, $R$ are obtained below. 
%in Appendix \ref{finish}.
%from \eqref{pidrahuvaly1} and \eqref{pidrahuvaly2}.

%%%%%%%%%%%%%%%%%

\section{Boundary conditions for %the stress vector of 
${\bf v}^{(311)}$} \label{finish}
In order to establish the coefficient $Q$ in \eqref{bound_cond_v311first} we consider each term in the right-hand side of the first boundary condition of \eqref{eto_est_predposlednii_cond} and collect coefficients in front of $\sin q_0x$. % for ${\bf v}^{(3)}$ 
%we consider each term of its right hand side.

Equation \eqref{39} yields
\begin{align}
2\mu\partial^2_{xy}v_x^{(1)}\rho^{(1)}[\rho^{(1)}]^\prime=-\textstyle{\frac{\zeta q_0^3}{4}}\sin q_0x	+C\sin 3q_0x.
\label{bratskayamogila_nachalo}
\end{align}
Hereafter $C$ denotes a generic constant whose value may possibly change from line to line.
Next, by \eqref{pervayatrudnyashtuka} we have
\begin{align}
&2\mu \partial_x v_x^{(1)}[\rho^{(2)}]^\prime\\
&=-2\mu \beta q_0\Bigl(D_1+ \textstyle{\frac{\zeta\xi}{2\mu\left(4\mu+\sqrt{2\mu\xi}\right)}}-\textstyle{\frac{\zeta_i}{\sqrt{2\mu\xi}+2\mu}} +\textstyle{\frac{2\zeta +\gamma q_0^2}{2\mu}}\Bigr)\sin q_0x  +C\sin 3q_0x,\nonumber
\end{align}
and \eqref{bezymyashka} entails 
\begin{align}
2\mu \partial_x v_x^{(2)}[\rho^{(1)}]^\prime&=\textstyle{\frac{q_0}{2}}\Bigl(\Biggr. 
\textstyle{\frac{D_2}{\sqrt{4 q_0^2+{ \textstyle\frac{\xi}{2\mu}}}}} -
\textstyle{\frac{\zeta}{4}}
\Bigl( 3q_0^2+\textstyle{\frac{\xi}{2\mu}}+4\sqrt{q_0^2+1} -4\beta\bigl( \textstyle{\frac{\xi}{4\mu+\sqrt{2\mu \xi}}}+2\bigr)\Bigr)\notag\\
&+4\gamma q_0^2 \beta -\zeta_i  
\Bigl( \textstyle{\frac{\beta\sqrt{2\mu}}{\sqrt{\xi}+\sqrt{2\mu}}}-\textstyle{\frac{1}{4}} \Bigr)	\Biggl. \Bigr) \sin q_0 x +C\sin 3q_0x.
\label{ya_est_partial_x_v_x}
\end{align}   
Then considering \eqref{pervoedlyapervogo} for $y=0$ and using \eqref{pervayatrudnyashtuka} we get 
\begin{align}
\label{number136}
\mu&\bigl(\partial^2_{xy}v_y^{(1)}+\partial^2_{yy}v_x^{(1)}\bigr)\rho^{(2)}
=-\textstyle{\frac{\beta}{2}}\sin q_0x\biggl(q_0\zeta\bigl(1+\textstyle{\sqrt{q_0^2+1}}\bigr)-q_0\zeta_i
\biggr. \nonumber\\
&+2\mu q_0 \Bigl(D_1+ \textstyle{\frac{\zeta\xi}{2\mu\left(4\mu+\sqrt{2\mu\xi}\right)}}-
\textstyle{\frac{\zeta_i}{\sqrt{2\mu\xi}+2\mu}} +
\textstyle{\frac{2\zeta +\gamma q_0^2}{2\mu}}\Bigr)
\\
&\biggl.+\textstyle{\frac{\xi}{q_0}}	\Bigl(D_1+ \textstyle{\frac{\zeta\xi}{2\mu\left(4\mu+\sqrt{2\mu\xi}\right)}}-
\textstyle{\frac{\zeta_i}{\sqrt{2\mu\xi}+2\mu}} +
\textstyle{\frac{2\zeta +\gamma q_0^2}{2\mu}}\Bigr)
\biggr) +C\sin 3q_0x. \nonumber
\end{align}
Differentiating	\eqref{pervoedlyapervogo} in $y$ and substituting \eqref{40} we obtain, for $y=0$ 
\begin{align}
-\textstyle{\frac{\mu}{2}}&\bigl(\partial^3_{xyy}v_y^{(1)}+\partial^3_{yyy}v_x^{(1)}\bigr)[\rho^{(1)}]^2 \\
&=\textstyle{\frac{q_0}{8}}\sin q_0x \Bigl(\zeta_i \textstyle{\sqrt{q_0^2+1}}-
\textstyle{\frac{\zeta\xi }{2\mu}}
-2\zeta \bigl(q_0^2+\textstyle{\sqrt{q_0^2+1}}+1\bigr)\Bigr)+C\sin 3q_0 x. \nonumber
\end{align}
Similarly to \eqref{number136} using equation \eqref{strashni_rivnyannya_dlya_v^2_1} and \eqref{bezymyashka} we get
\begin{align}
-\mu &\bigl(\partial^2_{xy}v_y^{(2)}+\partial^2_{yy}v_x^{(2)}\bigr)\rho^{(1)}=
-\textstyle{\frac{\zeta q_0}{2}}\bigl(q_0^2-\sqrt{q_0^2+1}\bigr) \sin q_0x\notag\\ 
&+
\textstyle{\frac{\zeta q_0}{4}} \bigl( \sqrt{q_0^2+1}-4\beta\bigr) \bigl(1+\sqrt{4q_0^2+1}\bigr)\sin q_0x-\textstyle{\frac{\zeta_i q_0}{4}}\bigl(\sqrt{q_0^2+1} -4\beta\bigr)\sin q_0x    \nonumber\\ 
&+\textstyle{\frac{-8\mu q_0^2 -\xi}{8\mu q_0 }}\Bigl( \Bigr.
\textstyle{\frac{D_2}{\sqrt{4 q_0^2+{ \textstyle\frac{\xi}{2\mu}}}}} -
\textstyle{\frac{\zeta}{4}}
\Bigl( 3q_0^2+\textstyle{\frac{\xi}{2\mu}}+4\sqrt{q_0^2+1} -4\beta\Bigl( \textstyle{\frac{\xi}{4\mu+\sqrt{2\mu \xi}}}+2\Bigr)\Bigr)
\notag
\\ 	&\quad\quad\quad\quad\quad+4\gamma q_0^2 \beta -\zeta_i  
\Bigl( \textstyle{\frac{\beta\sqrt{2\mu}}{\sqrt{\xi}+\sqrt{2\mu}}}-\textstyle{\frac{1}{4}} \Bigr)	\Bigl.\Bigr) \sin  q_0 x +C\sin 3q_0 x.
\end{align}     
Finally,
\begin{align}
-\gamma&[\rho^{(1)}]^{\prime\prime}[\rho^{(2)}]^\prime
-\gamma[\rho^{(2)}]^{\prime\prime}[\rho^{(1)}]^\prime=\beta\gamma q_0^3\sin q_0x+ C\sin 3q_0x.
\label{bratskayamogila_konets}
\end{align}
Thus combining  \eqref{bratskayamogila_nachalo}--\eqref{bratskayamogila_konets} we have
\begin{align}
Q=& 	(\beta\gamma -\zeta) q_0^3+
\frac{\zeta q_0}{4} \bigl( \textstyle{\sqrt{q_0^2+1}}-4\beta\bigr) \bigl(1+\textstyle{\sqrt{4q_0^2+1}}\bigr)
+\frac{\zeta q_0}{4}\textstyle{\sqrt{q_0^2+1}}
\nonumber\\
&
-\beta \bigl(\mu  q_0-\textstyle{\frac{\xi}{2q_0}}\bigr)  \Bigl(D_1+ \textstyle{\frac{\zeta\xi}{2\mu\left(4\mu+\sqrt{2\mu\xi}\right)}}-
\textstyle{\frac{\zeta_i}{\sqrt{2\mu\xi}+2\mu}} +
\textstyle{\frac{2\zeta +\gamma q_0^2}{2\mu}}\Bigr)
\nonumber\\
& +\textstyle{\frac{\beta}{2}}\Bigl(q_0\zeta\bigl(1+\textstyle{\sqrt{q_0^2+1}}\bigr)+q_0\zeta_i
\Bigr)-\textstyle{\frac{q_0}{8}}\Bigl(\zeta_i \textstyle{\sqrt{q_0^2+1}}+
\textstyle{\frac{\zeta\xi }{2\mu}}+2\zeta \Bigr)	\nonumber\\
&\left. 
-\textstyle{\frac{4\mu q_0^2 +\xi}{8\mu q_0 }}\Biggl( 
\textstyle{\frac{D_2}{\sqrt{4 q_0^2+{ \textstyle\frac{\xi}{2\mu}}}}}
+4\gamma q_0^2 \beta -\zeta_i  
\Bigl(\textstyle{ \frac{\beta\sqrt{2\mu}}{\sqrt{\xi}+\sqrt{2\mu}}}-\textstyle{\frac{1}{4}} \Bigr)  \Bigr.\right.\nonumber\\
&\Bigl.-
\textstyle{\frac{\zeta}{4}}
\left( 3q_0^2+\frac{\xi}{2\mu}+4\textstyle{\sqrt{q_0^2+1}} -4\beta\Bigl( \textstyle{\frac{\xi}{4\mu+\sqrt{2\mu \xi}}}+2\Bigr)\right)
\Biggr).
\end{align}

%\subsection{the second b.c. of (41)}
Next we consider \eqref{eto_est_poslednii_cond}  and establish that  
%%%%%%%%%%% 1 page below %%%%%%%%%%%%%%%%%%%%%%
%\begin{align}	
%	&2\mu \partial_yv_y^{(3)}= \cos q_0x \biggl( b\gamma q_0^2 -2\zeta b+\zeta_i b-2\beta \zeta q_0^2 -2\xi b V^{(0)}+\Bigl(\textstyle{\frac{\xi}{8}} -\textstyle{\frac{3\mu q^2_0}{4}}\Bigr)D_1
%	\biggr. \nonumber\\
%	& -\textstyle{\frac{3\mu q^2_0}{4}}	\Bigl(\textstyle{\frac{\zeta\xi}{2\mu\left(4\mu+\sqrt{2\mu\xi}\right)}}-\textstyle{\frac{\zeta_i}{\sqrt{2\mu\xi}+2\mu}} \Bigr)
%	+\beta\zeta\bigl(\textstyle{\sqrt{q_0^2+1}}+\textstyle{\sqrt{4q_0^2+1}}\bigr)+\textstyle{ \frac{3\zeta}{2}} -\textstyle{\frac{\zeta_i \beta}{2}}
%	\nonumber\\
%	&
%	-2\zeta\textstyle{\sqrt{q_0^2+1}} 
%	-\xi \Bigl(V^{(2)} -\textstyle{\frac{\zeta\beta}{4\mu}}\Bigr)
%-\textstyle{\frac{\zeta}{4}}\bigl(q_0^2+2\sqrt{q_0^2+1}-1 \bigr)\bigl(1+\sqrt{4q_0^2+1}\bigr)
%	\nonumber\\
%	&+\textstyle{\frac{3\zeta_i}{8}}\sqrt{q_0^2+1}-\textstyle{\frac{\zeta_i}{4}}
%	\biggl.-\textstyle{\frac{1}{4}}\gamma q_0^4 +\zeta q_0^2
%	-\textstyle{\frac{\xi\gamma q_0^2}{8\mu}}
%	\biggr)-2\mu \partial^2_{yy}V_y \rho^{(3)} +C\cos 3q_0x\nonumber
%\end{align}
%or \textcolor{red}{ (work on $V^{(2)}=V^{(21)}-bV^{(0)}$), see 
\begin{align}
\label{pidrahuvaly2}
2\mu \partial_y v_y^{(3)}= b\bigl( \gamma q_0^2 -2\zeta +\zeta_i -\xi V^{(0)} \bigr)  \cos q_0 x   + R \cos q_0 x + C\cos 3q_0 x, 
\end{align}
where 
\begin{align}	R&= -2\beta \zeta q_0^2 -\textstyle{\frac{\zeta_i \beta}{2}}-\textstyle{\frac{\zeta_i}{4}}+\Bigl(\textstyle{\frac{\xi}{8}} -\textstyle{\frac{3\mu q^2_0}{4}}\Bigr)D_1 -\textstyle{\frac{3\mu q^2_0}{4}}	\Bigl(\textstyle{\frac{\zeta\xi}{2\mu\left(4\mu+\sqrt{2\mu\xi}\right)}}-
\textstyle{\frac{\zeta_i}{\sqrt{2\mu\xi}+2\mu}} \Bigr)
\nonumber\\
& 	+\beta\zeta\bigl(\textstyle{\sqrt{q_0^2+1}}+ \textstyle{\sqrt{4q_0^2+1}}\bigr)%-\frac{\beta\zeta_i}{2}
+ \textstyle{\frac{3\zeta}{2}}-2\zeta\textstyle{\sqrt{q_0^2+1}} 
-\xi \Bigl(V^{(21)} -\textstyle{\frac{\zeta\beta}{4\mu}}\Bigr)-\textstyle{\frac{\xi\gamma q_0^2}{8\mu}}
\nonumber\\
&-\textstyle{\frac{\zeta}{4}}\bigl(q_0^2+2\textstyle{\sqrt{q_0^2+1}}-1 \bigr)\bigl(1+\textstyle{\sqrt{4q_0^2+1}}\bigr)+\textstyle{\frac{3\zeta_i}{8}}\textstyle{\sqrt{q_0^2+1}}
-\textstyle{\frac{1}{4}}\gamma q_0^4 +\zeta q_0^2. 
\label{pidrahuvaly2iOfigeli} 
\end{align}
For the first term in the right hand side of  \eqref{eto_est_poslednii_cond}  we have, by \eqref{bc1},
\begin{align}
\mu \bigl(\partial_xv_y^{(1)}+\partial_yv_x^{(1)}\bigr)[\rho^{(2)}]^\prime=
-\beta \zeta q_0^2 \cos q_0x +C\cos 3q_0x.
\label{nefuh}
\end{align}
%Using equation \eqref{pervoedlyapervogo} for $y=0$ and \eqref{pervayatrudnyashtuka_v_x} we obtain %(see $\bigl(\partial^2_{xy}v_y^{(1)}+\partial^2_{yy}v_x^{(1)}\bigr)$ above)
%\begin{align}
%	\mu \bigl(\partial^2_{xy}&v_y^{(1)}+\partial^2_{yy}v_x^{(1)}\bigr)\rho^{(1)}[\rho^{(1)}]^\prime
%	=-\cos q_0x \biggl(
%	\textstyle{\frac{q^2_0\zeta}{4}}\bigl(1+\textstyle{\sqrt{q_0^2+1}}\bigr) \biggr. \nonumber\\
%	&+\textstyle{\frac{\xi}{4}}\Bigl(D_1+ \textstyle{\frac{\zeta\xi}{2\mu\left(4\mu+\sqrt{2\mu\xi}\right)}}-\textstyle{\frac{\zeta_i}{\sqrt{2\mu\xi}+2\mu}} +\textstyle{\frac{2\zeta +\gamma q_0^2}{2\mu}}\Bigr)-	\textstyle{\frac{q^2_0\zeta_i}{4}}\nonumber \\
%	&\biggl. +\textstyle{\frac{\mu q^2_0}{2}}	\Bigl(D_1+ \textstyle{\frac{\zeta\xi}{2\mu\left(4\mu+\sqrt{2\mu\xi}\right)}}-
%	\textstyle{\frac{\zeta_i}{\sqrt{2\mu\xi}+2\mu}} +
%	\textstyle{\frac{2\zeta +\gamma q_0^2}{2\mu}}\Bigr)
%	\biggr) +C\cos 3q_0x.
%\end{align}
Next two terms are transformed as follows
\begin{align}
\mu\bigl(&\partial_xv_y^{(2)}+\partial_yv_x^{(2)}\bigr)[\rho^{(1)}]^\prime
+\mu \bigl(\partial^2_{xy}v_y^{(1)}+\partial^2_{yy}v_x^{(1)}\bigr)\rho^{(1)}[\rho^{(1)}]^\prime=2\mu \partial_x v_x^{(1)} {[\rho^{(1)}]^\prime}^2\nonumber\\
&-\zeta [\rho^{(2)}]^\prime [\rho^{(1)}]^\prime -\gamma  {[\rho^{(1)}]^\prime}^2[\rho^{(1)}]^{\prime\prime}=\bigl(\textstyle{\frac{\gamma q_0^4}{4}}-\zeta\beta q_0^2\bigr)\cos q_0 x\nonumber\\
&+\textstyle{\frac{\mu q_0^2}{2}} \Bigl(D_1+ \textstyle{\frac{\zeta\xi}{2\mu\left(4\mu+\sqrt{2\mu\xi}\right)}}-
\textstyle{\frac{\zeta_i}{\sqrt{2\mu\xi}+2\mu}} +
\textstyle{\frac{2\zeta +\gamma q_0^2}{2\mu}}\Bigr)\cos q_0 x +C\cos 3q_0x,
\label{garna_f}
\end{align} 
where we have used   \eqref{v^2_boundarycondition_first}  and  \eqref{pervayatrudnyashtuka}.
%The boundary condition \eqref{v^2_boundarycondition_first} and equation \eqref{pervayatrudnyashtuka} for $y=0$ entail
%\begin{align}
%	&\mu\bigl(\partial_xv_y^{(2)}+\partial_yv_x^{(2)}\bigr)[\rho^{(1)}]^\prime= 
%	\cos q_0x \Bigl(
%\textstyle{	\frac{q^2_0\zeta}{4}}\bigl(1+\textstyle{\sqrt{q_0^2+1}}\bigr)- \beta\zeta q_0^2+\textstyle{\frac{\gamma q^4_0}{4}}-	\frac{q^2_0\zeta_i}{4} \Bigr.\nonumber \\
%	&+\bigl(\mu q_0^2+\textstyle{\frac{\xi}{4}}\bigr)\Bigl(D_1+ \textstyle{\frac{\zeta\xi}{2\mu\left(4\mu+\sqrt{2\mu\xi}\right)}}-
%	\textstyle{\frac{\zeta_i}{\sqrt{2\mu\xi}+2\mu}} +
%	\textstyle{\frac{2\zeta +\gamma q_0^2}{2\mu}}\Bigr)
%	\Bigl. 	
%	\Bigr) +C\cos 3q_0x.
%\end{align}
With the help of  \eqref{UrForVy}  we find, for  $y=0$
% (or directly differentiating  \eqref{Vy} at $y=0$) we obtain 
\begin{align}
&-2\mu \partial^3_{yyy}V_y \rho^{(1)}\rho^{(2)}
-\textstyle{\frac{\mu}{3}}  \partial^4_{yyyy}V_y [\rho^{(1)}]^3=\cos q_0x \Bigl(
\textstyle{\frac{\beta\zeta_i}{2}}-2\beta\zeta-\zeta+\textstyle{\frac{\zeta_i}{8}}
\Bigr. \nonumber\\
&\quad\quad\quad\quad\quad\Bigl.-\textstyle{\frac{\beta\zeta\xi}{4\mu}}-\textstyle{\frac{\zeta\xi}{8\mu}}
-\textstyle{\frac{\zeta\xi^2}{16\mu\left(4\mu+\sqrt{2\mu\xi}\right)}}
+\textstyle{\frac{\zeta_i \xi}{8(\sqrt{2\mu\xi}+2\mu)}} 
\Bigr) +C\cos 3q_0x.\quad\quad
\end{align}
It follows from \eqref{40} that
\begin{align}
-2\mu \partial^2_{yy}v_y^{(1)} \rho^{(2)}=\cos q_0x \Bigl(
\beta\zeta\bigl(1+\textstyle{\sqrt{q_0^2+1}}+\textstyle{\frac{\xi}{4\mu}}\bigr)
-\textstyle{\frac{\beta\zeta_i}{2}}\Bigr)
+C\cos 3q_0x.
\end{align}
Differentiating \eqref{pervoedlyapervogo} in $x$ and subtracting the derivative of \eqref{vtoroedlyapervogo} in $y$, setting $y=0$, and using \eqref{bc2}, \eqref{pervayatrudnyashtuka} we get
\begin{align}
-\mu \partial^3_{yyy}v_y^{(1)}[\rho^{(1)}]^2&=-\textstyle{\frac{3}{8}}\cos q_0x \Bigl(
\zeta_i\bigl(q_0 ^2+\sqrt{q_0^2+1}\bigr)-(2\mu q_0^2+\xi)D_1 \Bigr. \nonumber\\
& -2(\mu q_0^2+\xi) \Bigl( \textstyle{\frac{\zeta\xi}{2\mu\left(4\mu+\sqrt{2\mu\xi}\right)}}
-\textstyle{\frac{\zeta_i}{\sqrt{2\mu\xi}+2\mu}} +
\textstyle{\frac{2\zeta +\gamma q_0^2}{2\mu}}\Bigr)\quad\quad\quad\quad
\nonumber\\
&	 -2\zeta\bigl(1+\textstyle{\sqrt{q_0^2+1}}\bigr)^2\Bigl.\Bigr) +C\cos 3q_0x.
\end{align}
Since $2\mu \partial^2_{yy}v_y^{(2)} =\mu (\Delta v_y^{(2)} +\partial _y {\rm div} {\bf v}^{(2)})-\mu \partial_x (\partial_x v_y^{(2)}+\partial_y v_x^{(2)})$ we 
derive with the help of \eqref{strashni_rivnyannya_dlya_v^2_2}--\eqref{veselytsya_i_psihue_ves_narod2}  that 
%\eqref{strashni_rivnyannya_dlya_v^2_1}--\eqref{strashni_rivnyannya_dlya_v^2_2}, \eqref{hochu}, \eqref{veselytsya_i_psihue_ves_narod1}, \eqref{veselytsya_i_psihue_ves_narod2}
\begin{align}
&2\mu \partial^2_{yy}v_y^{(2)} \rho^{(1)}
=\biggl(\biggr.
\bigl(2\mu q_0^2+\textstyle{\frac{\xi}{4}}\bigr) D_1  +\xi \bigl(V^{(2)}+2b V^{(0)} -\textstyle{\frac{\zeta\beta}{4\mu}}+\textstyle{\frac{5\zeta}{8\mu}}\bigr)
\notag\\
&\quad+
\bigl(2\mu q_0^2+\textstyle{\frac{5\xi}{8}}\bigr)\Bigl(
\textstyle{\frac{\zeta\xi}{2\mu\left(4\mu+\sqrt{2\mu\xi}\right)}}
-\textstyle{\frac{\zeta_i}{\sqrt{2\mu\xi}+2\mu}}\Bigr)+\bigl(3\mu q_0^2+\xi\bigr) \textstyle{\frac{\gamma q_0^2}{2\mu}}
+\zeta_i\bigl(\textstyle{\frac{\beta}{2}}-b\bigr)\quad\quad
\nonumber		\\
&\quad+\textstyle{\frac{\zeta}{4}}\Bigl(\bigl(q_0^2+2\sqrt{q_0^2+1}-1 -4\beta\bigr)\bigl(1+\sqrt{4q_0^2+1}\bigr)-14\sqrt{q_0^2+1}-4q_0^2\Bigr)\nonumber		\\
&\quad+\zeta(2b-1)-\textstyle{\frac{3\zeta_i}{8}}\bigl(q_0^2+2\sqrt{q_0^2+1}-1\bigr)\biggl.\biggr)\cos q_0x
+C\cos 3q_0x. 
\end{align}
%}
Finally, we have 
\begin{align}
-\gamma b[\rho^{(1)}]''-\textstyle{\frac{3}{2}}\gamma {[\rho^{(1)}]^\prime}^2[\rho^{(1)}]''=
\cos q_0x \Bigl(b\gamma q_0^2 +\textstyle{\frac{3}{8}}\gamma q_0^4 \Bigr)+C\cos 3q_0x.
\label{fuh}
\end{align}
Combining \eqref{nefuh}--\eqref{fuh} (and taking into account \eqref{bagatoV}) we obtain \eqref{pidrahuvaly2}  with $R$ given by \eqref{pidrahuvaly2iOfigeli} .

\section{The equation for the coefficient $b$}
\label{finish_nutochnovzhe}

Now we consider the boundary condition \eqref{fse} and find the coefficient $b$ which determines the bifurcation type.
Since % $v_y^{(1)}$ satisfies 
$2\mu \partial_y v^{(1)}_y=-2\mu\partial^2_{yy}V_y\rho^{(1)}+\gamma [\rho^{(1)}]^{\prime\prime}$ for $y=0$, we have% \eqref{linearized_system} 
$$
\partial^2_{yy}V_y \rho^{(1)}\rho^{(2)}+\partial_yv_y^{(1)}\rho^{(2)}%=\frac{\gamma}{2\mu} [\rho^{(1)} ]^{\prime\prime} \rho^{(2)}
=-\textstyle{\frac{\beta \gamma q_0^2}{4\mu}} \cos q_0 x +C\cos 3q_0x.
$$
By \eqref{40}, \eqref{before_bezymyashka}--\eqref{bezymyashka} and  \eqref{pervayatrudnyashtuka_v_x} we get, correspondingly,
\begin{align*}
\textstyle{\frac{1}{2}}\partial_{yy}^2 v_y^{(1)}[\rho^{(1)}]^2=
\textstyle{\frac{3}{16\mu}}
\left(\zeta_i-2\zeta\bigl(\textstyle{\sqrt{q_0^2+1}}+1+\textstyle{\frac{\xi}{4\mu}}\bigr)\right)\cos q_0x+C\cos 3q_0x,
\end{align*}
\begin{align*}
\partial_y v^{(2)}_y\rho^{(1)}&-v_x^{(2)}{\rho^{(1)}}^\prime=\textstyle{\frac{1}{2\mu}}\Bigl( -\textstyle{\frac{3}{4}}\zeta q_0^2 -\textstyle{\frac{\zeta_i}{4}} +\zeta \textstyle{\sqrt{q_0^2+1}}+\textstyle{\frac{\zeta\xi}{8\mu}} \Bigr)
\cos q_0 x  \\
&+\textstyle{\frac{1}{8\mu}}\Bigl(
\frac{D_2}{\sqrt{4 q_0^2+{ \textstyle\frac{\xi}{2\mu}}}} +
\frac{
3\zeta}{4}
q_0^2+\frac{\zeta\xi}{8\mu}+\zeta\sqrt{q_0^2+1} -4\gamma q_0^2 \beta   -\frac{\zeta_i}{4}\Bigr.
\\
&\Bigl.
-2\mu \beta \Bigl(\textstyle{\frac{\xi\zeta}{2\mu(4\mu+\sqrt{2\mu \xi})}} -
\frac{\zeta_i}{\sqrt{2\mu\xi}+2\mu}+\frac{\zeta}{\mu}
\Bigr)
\Bigr)\cos q_0 x+C\cos 3q_0x,
\end{align*}
and
$$
- v_x^{(1)} [\rho^{(2)}]^\prime =\beta 
\Bigl(D_1+ \textstyle{\frac{\zeta\xi}{2\mu(4\mu+\sqrt{2\mu\xi})}}-\frac{\zeta_i}{\sqrt{2\mu\xi}+2\mu} +\frac{2\zeta +\gamma q_0^2}{2\mu}\Bigr)\cos q_0 x+C\cos 3q_0x.
$$
Using \eqref{UrForVy} we obtain
\begin{align*}
\textstyle{\frac{1}{6}}\partial_{yyy}^3 V_y [\rho^{(1)}]^3=
%	\frac{1}{12\mu}\Bigl( \frac{\zeta\xi}{2\mu} +4\zeta -\zeta_i\Bigr)\cos^3 q_0 x=
\frac{1}{16\mu}\Bigl( \frac{\zeta\xi}{2\mu} +4\zeta -\zeta_i\Bigr)\cos q_0 x +C\cos 3q_0x,
\end{align*}
and from \eqref{39} we have
$$
-\partial_y v_x^{(1)}\rho^{(1)}[\rho^{(1)}]^\prime =\frac{\zeta q_0^2}{8\mu} \cos q_0 x+C\cos 3q_0x.
$$
Substituting the expressions obtained above into \eqref{fse}, combining all the terms with $\cos q_0x$, and taking into account  \eqref{pokraschennya} we get the relation \eqref{ReLaTiOn}.
Finally, using \eqref{v311_x}--\eqref{v311_y} we obtain 
\begin{align}
\label{alligator}
&-bq_0 \partial_q \Lambda(q_0)=\textstyle{\frac{D(q_0)+2\mu\xi q_0^2+\xi^2}{2\mu q_0 D(q_0)}}Q +\frac{\xi \sqrt{q_0^2+\textstyle\frac{\xi}{2\mu}}}{D(q_0)}R-I_y^{(123)}
\notag\\ 
&\quad \quad +	\textstyle{\frac{D(q_0)+2\mu\xi q_0^2+\xi^2}{2q_0 D(q_0)}}\Bigl(
q_0 I_y^{(123)}\Bigr. 
-I_x^{(1)} 	
%\notag\\ 
%&\quad \quad\Bigl.
-\textstyle{\sqrt{q_0^2+1}}  I_x^{(123)}- 
\sqrt{4q_0^2+1} I_x^{(3)}  \Bigr)
\notag\\ 
&\quad \quad
-\textstyle{\frac{1}{4\mu}}\Bigl(\zeta \sqrt{q_0^2+1} -\frac{\zeta}{2} -\frac{\zeta_i}{8}+\frac{\zeta\xi}{16\mu}-\frac{5\zeta q_0^2}{8}-\beta \gamma q_0^2\Bigr)	\\
&\quad \quad-\textstyle{\frac{2 \mu \xi  \sqrt{q_0^2+\textstyle\frac{\xi}{2\mu}}}{D(q_0)}} 
%\Bigl(\bigl(1+\sqrt{q_0^2+1}\bigr) I_y^{(1)} +\sqrt{q_0^2+1}  I_y^{(2)} + \bigl(\sqrt{q_0^2+1}+\sqrt{4q_0^2+1}\bigr) I_y^{(3)}  \Bigr)
\Bigl(I_y^{(1)}+\sqrt{q_0^2+1} I_y^{(123)} +\sqrt{4q_0^2+1} I_y^{(3)}  \Bigr)\notag 
\\
&\quad \quad-\beta D_1-\textstyle{\frac{D_2}{8\mu\sqrt{4 q_0^2+{ \textstyle\frac{\xi}{2\mu}}}}}-\frac{3\beta}{4}\Bigl(\frac{\xi\zeta}{2\mu(4\mu+\sqrt{2\mu \xi})} -
\frac{\zeta_i}{\sqrt{2\mu\xi}+2\mu}+\frac{\zeta}{\mu}
\Bigr),\notag
\end{align}
where $I_x^{(123)}=I_x^{(1)} +  I_x^{(2)} +  I_x^{(3)}$ and $I_y^{(123)}=I_y^{(1)} +  I_y^{(2)} +  I_y^{(3)}$.

\end{appendices}

%\bibliography{refs-random-axons}
%\bibliography{sn-bibliography-ours}

\end{document}